\documentclass[11pt]{article}
\usepackage{graphicx, amsmath, amsthm, amsfonts, enumerate, authblk,
  amssymb, graphics, fullpage, esint}
\usepackage{empheq}
\usepackage{bigints, bbm}
\usepackage{epsfig}
\usepackage{amsmath}
\usepackage{scrlfile}
\usepackage{srcltx}
\usepackage{rotating}
\usepackage{color}
\usepackage[mathscr]{euscript}
\usepackage{relsize}
\usepackage{mathrsfs} 

\newtheorem{theorem}{Theorem}[section]
\newtheorem{lemma}[theorem]{Lemma}

\newtheorem{proposition}[theorem]{Proposition}
\newtheorem{definition}{Definition}[section]

\newtheorem{remark}{Remark}

\makeatletter
\def\blfootnote{\xdef\@thefnmark{}\@footnotetext}
\makeatother
\allowdisplaybreaks

\newcommand{\complex}{\mathbb{C}}

\newcommand{\reals}{\mathbb{R}}
\newcommand{\cH}{{\cal H}}

\newcommand{\norm}[1]{\left\|#1\right\|}

\newcommand{\convergesto}[2]{\xrightarrow[#1\to #2]{}}

\DeclareMathOperator{\dom}{dom}
\DeclareMathOperator{\clos}{clos}
\DeclareMathOperator{\ran}{ran}
\DeclareMathOperator*{\Ker}{\mathrm{Ker}}
\DeclareMathOperator{\re}{Re}
\DeclareMathOperator{\im}{Im}

\DeclareMathOperator{\diag}{diag}

\DeclareMathOperator*{\slim}{s-lim}

\def\hom{\text{\rm hom}}

\def\soft{{\rm soft}}
\def\stiff{{\rm stiff}}

\def\e{\varepsilon}

\usepackage{cite}
\begin{document}
\title{\sc Asymptotic analysis of operator families and applications to resonant media}

\author[1]{Kirill~D.~Cherednichenko}
\author[2]{Yulia~Yu.~Ershova}
\author[2]{Alexander~V.~Kiselev}
\author[3]{Vladimir~A.~Ryzhov}
\author[2]{Luis~O.~Silva}

\affil[1]{Department of Mathematical Sciences, University of Bath,
  Claverton Down, Bath, BA2 7AY, UK}



\affil[2]{Departamento de F\'{i}sica Matem\'{a}tica, Instituto de
  Investigaciones en Matem\'aticas Aplicadas y en Sistemas, Universidad
  Nacional Aut\'onoma de M\'exico, C.P. 04510, Ciudad de M\'exico,
  M\'exico}

\affil[3]{Unity Technologies, 30 3rd St, San Francisco, CA, United States}
\maketitle

\begin{abstract}
  We give an overview of operator-theoretic tools that have recently proved useful in the analysis of boundary-value and transmission problems for second-order partial differential equations, with a view to addressing, in particular, the asymptotic behaviour of resolvents of physically motivated parameter-dependent operator families.
   We demonstrate the links of this rich area, on the one hand, to functional frameworks developed by S. N. Naboko and
  his students, and on the other hand, to concrete applications of current
  interest in the physics and engineering communities. 
\end{abstract}

\par{\raggedleft\slshape In memoriam Sergey Naboko\par}

\section{Introduction}

\label{sec:introduction}

It has transpired recently that a number of operator-theoretical techniques which have been under active development for the past 60 years or so are extremely useful in the asymptotic analysis of highly inhomogeneous media. Apart from yielding sharp asymptotics of the corresponding Hamiltonians in the norm-resolvent topology, this research has resulted in a number of new important, yet mostly unexplored,  connections between certain areas of the modern operator and spectral theory. These include the theory of dilations and functional models of dissipative and non-selfadjoint operators in Hilbert spaces, the boundary triples theory in the analysis of symmetric operators, zero-range models with an internal structure and, finally, the theory of generalised resolvents and their out-of-space ``dilations''.

The present survey, based on our results published in \cite{KCher,CEK,CherErKis,CEKN,CENS,CherednichenkoKiselevSilva,CherKisSilva1,CherKisSilva2,CherKisSilva3,CENS,MR2330831,Ryzh_spec,Ryzhov_closed}, attempts to shed some light on these connections and to thus present the subject area of strongly inhomogeneous media under the spotlight of modern spectral theory. We aim to show that in many ways this novel outlook allows one to gain a better understanding of the mentioned area by providing a universal abstraction layer for all the main objects to be found in the asymptotic analysis. Moreover, in most cases one can then proceed in the analysis on a purely abstract level surprisingly far, essentially postponing the use of the specific features of the problem at hand till the very last stages.

For readers' convenience, we have included a rather detailed exposition of the relevant areas of operator and spectral theory, keeping in mind that some papers laying the foundations of these areas have been poorly accessible to date.

We start with Section 2, devoted to the now-classical theory of dilations of dissipative operators. The role of dissipative operators as opposed to self-adjoint ones is that whereas the latter represent physical systems with the energy conservation law (``closed'', or conservative, systems), the former allow for the consideration of a more realistic setup, where the loss of the total energy is factored in. The importance of dissipative systems has been a common place since at least the works of I.~Prigogine; it is well-known that such systems may possess certain rather unexpected properties. The main difference between the self-adjoint and dissipative theories can be clarified,
following M.~G.~Kre\u\i n, as follows: the major instruments of
self-adjoint spectral analysis arise from the Hilbert space geometry,
whereas this geometry doesn't work very well in the
non-selfadjoint situation, with modern complex analysis taking
the role of the main tool in the investigation.

Since the seminal contribution of B.~Sz.-Nagy and C.~Foia\c{s}, the main object of dissipative spectral analysis has been the so-called dilation, representing an out-of-space self-adjoint extension, in the sense of \eqref{DilationCondition} below, of the original dissipative operator $L$. Our argument actually goes as far as to suggest that this concept underpins the whole set of ideas and notions presented in the paper. B.~S.~Pavlov's explicit construction of dilation relies upon the second major ingredient, which is the characteristic function $S(z),$ see \eqref{DissCharF}, which is an analytic operator-valued contraction in $\mathbb{C}_+$. The analysis of the dissipative operator $L$ is reduced to the study of the function $S(z),$
and hence from this point onward it belongs to the domain of complex analysis. It turns out that the sole knowledge of $S(z)$ yields an explicit spectral representation of the dilation. Moreover, Naboko has shown that, in the same representation, a whole family of operators ``close'' to $L$, self-adjoint and non-selfadjoint alike, are modelled in an effective way. This idea in particular led to the description of absolutely continuous subspaces of all the operators considered as the closure of the so-called smooth vectors set. This latter is characterised as the collection of vectors such that the resolvent of the operator in question in the spectral representation maps them as the multiplication operator. An explicit construction of wave operators and scattering matrices then follows almost immediately.

In Section \ref{sec:ryzh-funct-model}, we give a systematic exposition of this approach applied to the family of extensions of a symmetric densely defined operator on a Hilbert space $H$ possessing equal deficiency indices. In so doing, we follow closely the strategy suggested by Sergey Naboko which he had applied in the analysis of additive relatively bounded perturbations of self-adjoint operators. We thus hope to provide a coherent presentation of the major contribution by Naboko to the spectral analysis of non-selfadjoint operators.

Our analysis is facilitated by the boundary triples theory, being an abstract framework, from which the extensions theory of symmetric operators, especially differential  operators, greatly benefits. That's why we start our exposition by introducing the main concepts of this theory. The formula obtained for the scattering operator in the functional model representation allows us to derive an explicit formula for the scattering matrix, formulated in terms of Weyl-Titchmarsh $M$-matrices, i.e., in the natural terms associated with the problem. In Section \ref{sec:theor-appl-funct}, we consider an application of this technique to an inverse scattering problem on a quantum graph, where we are able to give an explicit solution to the problem of reconstructing matching conditions at graph vertices.

In Section \ref{sec:funct-model-family}, we consider the possible generalisation of the approach described above to the case of partial differential operators (PDO), associated with boundary value problems (BVP). Although the theory of boundary triples has been successfully applied to the spectral analysis of BVP for ordinary differential operators and related setups, in its original form this theory is not suited for dealing with BVP for partial differential equations (PDE), see \cite[Section
7]{MR2418300} for a relevant discussion.
Recently, when the works \cite{Grubb_Robin, Grubb_mixed, BehrndtLanger2007, Gesztesy_Mitrea, Ryzh_spec, MR2418300} started to appear, it has transpired that, suitably modified, the boundary triples approach nevertheless admits a natural generalisation to the BVP setup, see also the seminal contributions by
J. W. Calkin \cite{Calkin}, M.\,S.\,Birman \cite{Birman}, L.\,Boutet de Monvel \cite{BdeM}, M.\,S.\,Birman and M.\,Z.\,Solomyak \cite{Birman_Solomyak}, G.\,Grubb \cite{Grubb_classic}, and M.\,Agranovich \cite{Agranovich}, which provide the analytic backbone for the related operator-theoretic constructions.

In all cases mentioned above, one can see the fundamental r\^{o}le of
a certain Herglotz operator-valued analytic function, which in
problems where a boundary is present (and sometimes even without an
explicit boundary \cite{Amrein}) turns out to be a natural
generalisation of the classical notion of a Dirichlet-to-Neumann
map. Moreover, it is precisely this object that permits to define the characteristic function which in turn facilitates the functional model construction.

In Section \ref{sec:models-potent-zero}, we pass over to the discussion of zero-range models with an internal structure. The idea of replacing a model of short-range interactions by an explicitly solvable one with a zero-radius potential (possibly with an internal structure)
\cite{Berezin_Faddeev,Pavlov_internal_structure,Pavlov_Helmholtz_resonator, MR0080271,MR0024574,MR0024575, MR0051404} has paved the way for an influx of methods of the theory of extensions (both
self-adjoint and non-selfadjoint) of symmetric operators to problems
of mathematical physics. In particular, we view zero-range potentials with an internal structure as a particular case of out-of-space self-adjoint extensions of symmetric operators, the theory of which is intrinsically related to the analysis of generalised resolvents. The latter area is introduced in Section \ref{gen_res1}. We argue that out-of-space self-adjoint extensions corresponding to generalised resolvents naturally supersedes the dilation theory as presented in Section \ref{sec:pavl-funct-model}.

On yet another level, we claim that zero-range perturbations (and more precisely, zero-range perturbations with an internal structure) appear naturally as the norm-resolvent limits of Hamiltonians in the asymptotic analysis of inhomogeneous media. This relationship is established using the apparatus of generalised resolvents, as explained in Section \ref{gen_res2}.

Finally, we mention here that the theory of functional models as presented in Section \ref{sec:functional-model} is directly applicable to the treatment of models of zero-range potentials with an internal structure. Its development yields a complete spectral analysis and an explicit construction of the scattering theory for the latter.

Two different models are considered in
Section \ref{sec:models-potent-zero}, one of these being a one-dimensional periodic model with critical contrast, unitary equivalent to the double porosity one. The PDE counterpart of the latter is discussed in Section \ref{sec:appl-cont-mech}. The second mentioned model pertains to the problem with a low-index inclusion in a homogeneous material. Our argument shows that the leading order term in the asymptotic expansion of its resolvent admits the same form as expected of a zero-range model; the difference is that here the effective model of the media is no longer ``zero-range'' per se; rather it pertains to a singular perturbation supported by a manifold. Therefore, this  allows us to extend the notion of internal structure to the case of distributional perturbations supported by manifolds.

The discussion started in Section \ref{sec:models-potent-zero} is then continued in Section \ref{sec:appl-cont-mech}. We note that in every model considered so far, the internal structure of the limiting zero-range model is necessarily the simplest possible, i.e., pertains to the out-of-space extensions defined on $H\oplus \mathbb{C}^1$, where $H$ is the original Hilbert space. It turns out that this is due to the fact that we only consider norm-resolvent convergence when the spectral parameter $z$ is restricted to a compact set in $\mathbb{C}.$

Passing over to a generic setup with $z$ not necessarily in a compact, we are able to claim that in some sense the internal structure can be arbitrarily complex, provided that the spectral parameter $z$ is allowed to grow with the large parameter $a$, which describes the inhomogeneity, increasing to $+\infty.$ This allows us to present an explicit example of a non-trivial internal structure in the leading order term of the norm-resolvent asymptotics in Section 4.4, which is supplemented with the discussion of the so-called scaling regimes which we introduce in Section 5.1.

The remainder of Section 5 is devoted to the analysis of a double porosity model of high-contrast homogenisation, where the leading order term of the asymptotic expansion is obtained by an application of the operator-theoretical technique based on the generalised resolvents, as in Section 4.

\section{Functional models for dissipative and nonselfadjoint operators}
\label{sec:functional-model}


Functional model construction for a contractive linear operator~$T$
acting on a Hilbert space~$H$ is a well developed domain of the
operator theory.
Since  the pioneering works by B.~Sz.-Nagy, C.~Foia\c{s}~\cite{MR2760647},
P.~D.~Lax, R.~S.~Phillips~\cite{MR0217440},  L.~de~Branges,
J.~Rovnyak~\cite{deBranges, deBrangesRovnyak}, and
M.~S.~Liv\v{s}ic~\cite{Livsic1954}, this field of research has attracted many specialists in operator theory, complex analysis, system control,
Gaussian processes and other disciplines.
Multiple studies culminated in the development of a comprehensive theory
complemented by various applications,
see~\cite{DyM,Fu,Nikolski,NikolskiiKhrushchev, NikolskiiVasyunin1998}
and references therein.
%

The underlying idea of a functional model is the fundamental theorem
of B.~Sz.-Nagy and C.~Foias establishing the existence of a unitary dilation for
any contractive (linear) Hilbert space operator~$T$, $\|T\| \leq 1$.
The unitary dilation~$U$ of $T$ is a unitary operator on a
Hilbert space~$\mathcal H \supset H$
such that $P_H U^n|_H = T^n$ for all
$n = 1,2,\dots$.
Here $P_H : {\mathcal H} \rightarrow H$ is an orthoprojection
from~$\mathcal H$ to its subspace~$H$.
The dilation~$U$ is called minimal if the linear set~$\vee_{n>0} U^n H$
is dense in~$\mathcal H$.
The minimal dilation~$U$ of a contraction is unique up to unitary
equivalence.
The spectrum of~$U$ is absolutely continuous and covers the
unit circle~$\mathbb T = \{z \in \mathbb C  : |z| = 1\}$.
If one denotes by~$\mu$ the spectral measure of~$U$,
the spectral theorem yields that operator~$T$ is unitarily equivalent
to its model~$T = P_H z |_H$,  where ~$f \mapsto z f$  is the
operator of multiplication on the spectral representation space $L^2(\mathbb T, \mu)$ of~$U$.
%


Due to its abstract nature, a significant part of the functional model
research for contractions took place among specialists
in complex analysis and operator theory.
The parallel theory for unbounded
operators is based on the Cayley
transform~$T \mapsto  -{\rm i}(T+I)(T-I)^{-1}$ applied to a
contraction~$\|T \| \leq 1$.
Assuming that~$\clos\ran (T -I) = H$, the Cayley transform of~$T$,
$\|T\| < 1$  is a dissipative densely defined
operator~$L = -{\rm i}(T+I)(T-I)^{-1}$, not necessarily bounded in~$H$.
It is easily seen that the spectrum of~$L$ is situated in the closed upper half
plane~$\overline{\mathbb C}_+$ of the complex plane.
The imaginary part of~$L$ (defined in the sense of forms if needed)
is a non-negative operator.
%


Alongside the developments in operator theory, the second half
of the 20th century witnessed a huge progress in the spectral analysis
of linear operators pertaining to physical disciplines.
The principal tool of this was the method of Riesz projections, i.e., the contour
integration of the operator's resolvent in the complex plane of
spectral parameter.
The spectral analysis of self-adjoint operators of quantum mechanics
can be viewed as the prime example of highly successful application of contour
integration in the study of conservative systems, i.e., closed
systems with the energy preserved in the course of evolution.
%

Topical questions concerning the behavior of non-conservative systems, where
the total energy is not preserved, and of resonant systems motivated in-depth studies of (unbounded)
non-selfadjoint operators.
The analysis of non-conservative dynamical systems and of
non-selfadjoint operators especially relevant to the functional
model theory was pioneered
in the works
of M.~S.~Brodski\v{i},
 M.~S.~Liv\v{s}ic and their colleagues,
 see~\cite{Brodski,  BrodskiLivsic, LivsicBook}.
Starting with a (bounded) self-adjoint
operator~$A = A^*$
as the main operator of a closed conservative system,
these authors considered the
coupling of this system to the outside
world by means of externally attached channels.
This construction represents a model of the so-called
``open system'', that is, of a physical system
connected to its external environment.
The energy of such modified system
can dissipate through the external channels,
while at the same time the energy can be fed  into
the system from the outside in the course
of its evolution.
In works of  M.~S.~Brodski\v{\i} and
 M.~S.~Liv\v{s}ic, the external channels are modelled as an
additive perturbation of the main self-adjoint
operator~$A$ by a (bounded) non-selfadjoint
perturbation: $A \to L = A +{\rm i}V$, $V = V^*$.
The ``channel vectors'' form the Hilbert
space~$E = \clos \ran |V|$.
If~$V \geq 0$, the operator~$L$ is dissipative
(i.e., $\im (Lu,u) > 0$, $u \in H$);
it describes a non-conservative
system losing the total energy.
In turn, and quite analogously to the
case of contractions,
under the assumption~$ \mathbb C_- \subset \rho(L)$
(recall that dissipative operators satisfying this condition are
called {\em maximal}),
the self-adjoint dilation of~$L$ is a self-adjoint
operator~$\mathscr L$ on a wider
space~$\mathcal H \supset H$
such that
\begin{equation}\label{DilationCondition}
(L-zI)^{-1} = P_H (\mathscr L - zI)^{-1}|_H, \quad
z \in \mathbb C_-,
\end{equation}
where~$P_H$ is an orthogonal projection
from $\mathcal H$ onto $H$.
The operator~$\mathscr L$ describes a (larger) system with
the state space~$\mathcal H$, in which
the energy is conserved, whereas~$L$ describes
its subsystem losing its total energy.
In the general case, a non-dissipative~$L$
corresponds to an
open system where both the
energy loss and the energy supply
coexist.
%


The analysis of a non-selfadjoint operator~$L$ relies on the
notion of its characteristic function~\cite{MR0020719,Strauss1960}
discovered by M.~S.~Liv\v{s}ic in 1943--1944.
It is a bounded analytic
operator-function~$\Theta(z)$, $z \in \rho(L^*)$
defined on the resolvent set of~$L^*$ and acting
on the ``channel vectors'' from the space~$E$.
For dissipative $L$ the function~$\Theta$ coincides with the characteristic
function of a contraction~$T = (L-{\rm i}I)(L+{\rm i}I)^{-1}$
(the inverse Cayley transform of $L$), featured prominently in the works by B.~Sz.-Nagy and C.~Foias.
The characteristic function of a non-selfadjoint
operator~$L$ (or, alternatively, of its Cayley transform) determines the original operator $L$ uniquely up to a unitary
equivalence (see~\cite{Livsic1954, MR2760647}),
provided~$L$ has no non-trivial self-adjoint ``parts''.
Therefore, the study of non-selfadjoint operators
is reduced to the study of
operator-valued analytic functions.
In and of itself, this does not mean much as these functions might
be as complicated as the operators themselves.
A simplification is achieved when the values of
these functions are either matrices or belong to Schatten-von Neumann classes of
compact operators, which is often the case in physical applications.
%

Closely related to the Sz.-Nagy-Foias model for
contractions and to the open
systems framework are  the Lax-Phillips
scattering theory~\cite{MR0217440}
and the ``canonical model'' due to  L.~de Branges
and J.~Rovnyak~\cite{deBrangesRovnyak}.
The latter is developed
for completely non-isometric
contractions and their adjoints
with quantum-mechanical
applications in mind.
The Lax-Phillips  theory was
originally developed to facilitate the analysis of
scattering problems  for hyperbolic wave equations in exterior domains to compact scatterers.
It provides useful
intuition into the underpinnings of
the functional model construction and
this connection will be exploited
in the next section.

It was realized very early~\cite{MR0206711} that the three
theories, i.e., the open systems theory, the Sz.-Nagy-Foia\c{s} model, and
the Lax-Phillips scattering, all deal with
essentially the same objects. 
In particular, the characteristic function of a contraction (or of a
dissipative operator) emerges, albeit under disguise, in all three theories.
Being a purely theoretical abstract object in the Sz.-Nagy-Foia\c{s} theory, the characteristic function emerges as a
transfer function of a linear system according to M.~S.~Brodski\v{\i} and
 M.~S.~Liv\v{s}ic, and  as the scattering matrix in the Lax-Phillips theory.
%
The characteristic function of a contraction is
also the central component in
the L.~de Branges and J.~Rovnyak
model theory~\cite{deBrangesRovnyak}.
Deep connections between the Sz-Nagy-Foias and the
de Branges-Rovnyak
models are clarified in a series of
papers by N.~Nikolskii and V.~Vasyunin~\cite{NikolskiiVasyunin1986,
NikolskiiVasyunin1989, NikolskiiVasyunin1998}.
%
%
%

\subsection{Lax-Phillips theory}

The Lax-Phillips scattering theory~\cite{MR0217440}
for the acoustic waves by
a smooth compact obstacle in~$\mathbb R^n$ with  $n \geq 3$ odd
provides an excellent
illustration of the intrinsic links
between the operator theory and
mathematical physics.
%
%
A number of concepts found in the theory of functional
models of dissipative operators
find their direct counterparts here, expressed in the
language of realistic physical processes.
For instance, the characteristic function
of the operator governing the
scattering process
is realized as the scattering matrix,
the self-adjoint dilation corresponds to
the operator of ``free'' dynamics, i.e., the
wave propagation process observed
in absence of the obstacle, and the scattering
channels are a direct analogue of the channels found in
the Brodski\v{\i}-Liv\v{s}ic  constructions.
In this section we briefly recall
the main concepts of Lax-Phillips scattering theory.

Let $\cal H$ be a Hilbert space
with two mutually
orthogonal subspaces $\cal
D_\pm \subset H$,  $\cal D_- \oplus \cal D_+ \neq \mathcal H$.
Denote by $\mathcal K$ the orthogonal
complement of~$\cal D_- \oplus \cal D_+$
in $\mathcal H$.
%
%
Assume the existence of a single parameter evolution group of
unitary operators~$\{U(t)\}_{t\in \mathbb R}$ with the following properties
\begin{equation}\label{conditionsUt}
\begin{aligned}
U(t)\mathcal{D}_- & \subseteq \mathcal{D}_-, \quad t \leq 0,  \\[0.3em]
U(t) \mathcal{D}_+ & \subseteq \mathcal{D}_+, \quad t \geq 0, \\[0.3em]
\cap_{t \in \mathbb R}U(t) \mathcal{D}_\pm &= \{0\},\\[0.3em]
\clos \cup_{t \in \mathbb R}U(t) \mathcal{D}_\pm & =  \mathcal{H}.
\end{aligned}
\end{equation}
In the acoustic scattering,
the space~$\mathcal H$ consists of
solutions to the wave equation
(i.e., acoustic waves) and is
endowed with the energy norm.
The group~$U(t)$ describes the
evolution of ``free'' waves in~$\mathcal H$,
that is,
the group~$U(t)$ maps
the Cauchy data of solutions at the time~$t = 0$
to their Cauchy data at the time $t$.
Since~$U(t)$ is unitary for all $t\in\mathbb R$,
the energy of solutions is preserved
under the time
evolution~$f = U(0) f \mapsto  U(t)f$, for any
Cauchy data~$f \in \mathcal H$.
Correspondingly, the infinitesimal
generator~$\mathscr B$ of~$U(t)$ is a
self-adjoint
operator in $\mathcal H$ with purely absolutely continuous spectrum
covering the whole real line.
The subspaces~$\cal D_\pm$ are called incoming and
outgoing subspaces of~$U(t)$.
These names are well justified. Indeed,
the subspace~$\mathcal D_-$ consists of solutions
that do not interact with the
obstacle prior to the moment~$t = 0$,
whereas $\mathcal D_+$ consists of scattered waves
which do not interact with the obstacle after $t = 0$.
There exist two representations for the generator~$\mathscr B$
associated with $\mathcal D_\pm$ (the so called incoming and
outgoing translation representations), in which
the group~$U(t)$ acts as the right shift
operator~$U(t): u(x)\mapsto u(x-t)$ on $L^2(\mathbb R, E)$
with some auxiliary Hilbert space~$E$.
In these representations the subspaces~$\mathcal D_\pm$
are mapped to $L^2(\mathbb R_\pm, E)$.
It is not difficult to see that this
construction satisfies the assumptions~(\ref{conditionsUt}).
Denote  by
$P_\pm : \mathcal{H} \to  [\mathcal{D}_\pm]^\perp$
the orthogonal projections to the complements of
$\mathcal{D}_\pm$ in $\cal H$.
The elements of~$\mathcal K$ are the scattering waves that are neither
incoming in the past, nor outgoing in the future,
i.e., the waves localized in the vicinity of the obstacle.
The interaction of incoming waves with the obstacle, i.e.,
the scattering
process, is then described by the compression of  the group~$U(t)$ to the
neighborhood of the obstacle:
$$
Z(t) = P_+ U(t) P_- = P_{\mathcal K} U(t) P_{\mathcal K}, \quad t \geq 0.
$$
Here~$ P_{\mathcal K} = P_- P_+$ is an orthoprojection on~$\mathcal H$.
The operator family~$\{Z(t)\}_{t \geq 0}$ forms a strongly
continuous semigroup acting on~$\mathcal{K}$.
Since $Z(t)$ is a compression of the unitary group, it is clear that
$\|Z(t)\| \leq 1$ for all $t \geq 0$.
The infinitesimal generator~$B$ of the semigroup~$\{Z(t)\}_{t \geq 0}$
turns out to be a
linear operator with purely discrete spectrum~$\{\lambda_k\}$
with~$\re(\lambda_k) < 0$,
$k = 1,2,\dots$.
The poles~$\{\lambda_k\}$ of the resolvent of~$B$
and the corresponding
eigenvectors are interpreted
as the scattering resonances.
These resonances correspond
to the poles of the scattering matrix
defined as an operator-valued function acting in the
space $L^2(\mathbb R_-, E)$, i.e., the space of vector
functions taking values in~$E$.
The scattering matrix is mapped by the Fourier transform to the
analytic in the lower half-plane operator function~$S(z)$, $z\in\mathbb C_-$
with zeroes at~$z_k = -{\rm i}\lambda_k$.
The boundary values $S(k - {\rm i}0)$ on the real
axis exist almost
everywhere in the strong operator topology
and are
unitary for almost all $k \in \mathbb R$.
The function~$S(z)$
permits an analytic continuation~$S(z) = [S^*(\bar z)]^{-1}$
to the upper half-plane, where it is meromorphic.

The results of~\cite{MR0206711} show
that $\theta(z): = S^*(\bar z)$ coincides with
the Liv\v{s}ic's characteristic
function of a dissipative operator being
unitary equivalent to the infinitesimal
generator of~$Z(t)$.
Consequently,~$\theta\bigl((z+{\rm i})/(z-{\rm i})\bigr)$
 is the characteristic function
of its Cayley transform as
defined by B.~Sz.-Nagy
and C.~Foia\c{s}~\cite{MR2760647}.
Finally,
the resolvents of
operators~$\mathscr L = -{\rm i}\mathscr  B$ and $L = -{\rm i}B$
satisfy the dilation equation~(\ref{DilationCondition}).
In other words, the operator corresponding to the free dynamics
is the self-adjoint dilation of the dissipative operator
that governs the scattering process.
%

\subsubsection{Minimality, non-selfadjointness, resolvent}

The beautiful geometric interpretation of scattering processes provided by the Lax-Phillips
theory is not entirely transferable to
the modelling of a general
dissipative (or contractive) operator.
For instance, given an arbitrary
dissipative operator~$L$ on a Hilbert
space~$K$, its  selfadjoint
dilation $\mathscr L$ does
not exists {\it a priori} and must be
explicitly constructed first.
In addition, such a
dilation~$\mathscr L$  should
be {\it minimal}, that is,
it must contain no
reducing self-adjoint parts
unrelated to the operator~$L$.
Mathematically, the minimality condition is
expressed by the equality
\[
\clos\bigvee_{z \notin R} {(\mathscr L -  zI)^{-1}
\mid_K}  =
\mathcal H
\]
where $\mathcal H$ is the dilation space~$\mathcal H
\supset K$.
The construction of a dilation satisfying this
condition is a highly non-trivial task which was successfully addressed for contractions
by B.~Sz.-Nagy and C.~Foia\c{s}~\cite{MR2760647}, aided by a theorem of M.~A.~Na\u{\i}mark \cite{Naimark1943},
and then by B.~Pavlov~\cite{MR0510053, Drogobych}
in two
important cases of dissipative
operators arising in
mathematical physics. Later, this construction was generalized to a generic setting. We dwell on this further in the following sections.
%


For obvious reasons, the functional model theory of non-selfadjoint operators deals with operators possessing no non-trivial reducing
self-adjoint parts.
Such operators are called {\em completely non-selfadjoint} or, using
a somewhat less accurate term, {\em simple}.
The rationale behind this condition is easy to
illustrate within the Lax-Phillips framework.
Let the dissipative operator~$L = -{\rm i}B$ governing
the wave dynamics in a vicinity of an obstacle possess
a non-trivial self-adjoint part.
This part is then a self-adjoint operator
acting on the subspace spanned
by the eigenvectors of~$L$
corresponding to its real eigenvalues.
The restriction of $Z(t)$ to this subspace
is an isometry for all~$t\in \mathbb R$, and
the energy of these states remains constant
(recall that the space~$K$ is equipped with
the energy norm).
Therefore these waves stay in
a bounded region adjacent
to the obstacle at all times.
These bound states do
not participate in scattering, as they
are invisible to the scattering process
describing the asymptotic
behaviour as $t \to\pm\infty$, so that
the matrix~$S$ (or the characteristic
function of~$L$) contains no
information pertaining to them.
In operator theory, this is known as the
claim that the characteristic function of a dissipative
operator is oblivious to its self-adjoint part.
It is also well-known, that the characteristic
function uniquely determines the
completely non-selfadjoint part of a
dissipative operator.
The interaction dynamics of the Lax-Phillips scattering
supposes neither minimality nor complete
non-selfadjointness.
One can envision incoming waves that do not interact with the
obstacle during their evolution and eventually become outgoing as~$t\to+\infty$.
 It is also easy to conceive of trapping
 obstacles preventing waves
 from leaving the neighbourhood of the
obstacle at~$t\to+\infty$.
The geometry of such
obstacles cannot be fully recovered from
the scattering data because, as explained above,
these standing waves
do not participate in the
scattering process.
%
%


In the applications discussed below,
unless explicitly stated otherwise,
all  non-selfadjoint operators are
assumed closed, densely defined
with regular points in both lower
and upper half planes.
The latter condition can be relaxed but is adopted in what follows for the sake of convenience.

\subsection{Pavlov's functional model and its spectral form}
\label{sec:pavl-funct-model}

Functional models for prototypical dissipative operators
of mathematical physics (as
opposed to the model for contractions), alongside explicit constructions of self-adjoint
dilations, were investigated by
B.~Pavlov in his
 works~\cite{MR0365199, MR0510053, Drogobych}.
Two classes of dissipative operators were considered:
the Schr\"odigner operator in $L^2(\mathbb R^3)$ with
a complex-valued potential, and the operator
generated by the differential
expression~$-y'' + q(x)y $ on the interval~$[0, \infty)$
with a dissipative boundary condition at $x = 0$.
In both cases the self-adjoint dilations are constructed explicitly in terms of the problem at hand, and
supplemented by the model
representations known today
as ``symmetric''
and commonly referred to
as the Pavlov's model.
The results of~\cite{MR0365199, MR0510053, Drogobych}
were extensively employed in
various applications and provided a foundation for the subsequent constructions
of self-adjoint dilations and functional models for
general non-selfadjoint
operators.

\subsubsection{Additive perturbations~\cite{MR0365199, MR0510053}}\label{sec:add_perturbations}

Let $A = A^*$ be a selfadjoint unbounded operator on a Hilbert space~$K$ and
$V$ a bounded non-negative operator~$V  = V^* = \alpha^2/2\geq 0$, where
$\alpha := (2V)^{1/2}$.

The paper~\cite{MR0510053} studies the dissipative
Schr\"odinger operator~$L = A +({\rm i}/2)\alpha^2$ in $\mathbb R^3$
defined by the differential expression~ $-\Delta + q(x)+({\rm i}/2)\alpha^2(x)$
with real continuous functions~$q$ and  $\alpha$
such that $0 \leq \alpha\leq C < \infty$.
The operators~$A = -\Delta + q$ and~$V = \alpha^2/2$ are the
real and imaginary parts
of~$L$ defined on $\dom(L) = \dom(A)$.
Assuming the operator~$L$ has no non-trivial self-adjoint components
 and the  resolvent set of~$L$ contains points in both upper and
lower half planes,
the operator $L$ is a maximal completely non-selfadjoint, densely defined
dissipative operator on~$K$.

According to the general theory, there exists a minimal dilation of
$L$, which is a self-adjoint operator~$\mathscr L$ on a Hilbert space $\mathcal
H \supset K$ such that
\begin{equation}\label{eqn:DilationEq}
 (L - zI)^{-1} = P_K  ( {\mathscr  L} - zI)^{-1}|_K, \quad z \in\mathbb C_-,
\end{equation}
where $P_K$  is the orthogonal projection
from $\mathcal H$ onto its subspace~$K$.

The dilation constructed in~\cite{MR0510053}  closely
resembles the generator $\mathscr B$ of the unitary group $U(t)$ in the Lax-Phillips theory: Pavlov realised that a natural way to construct a self-adjoint dilation would be to add the missing ``incoming'' and ``outgoing'' energy channels to the original non-conservative dynamics, thus mimicking the starting point of Lax and Phillips.
%
%
The challenge here is to determine the operator that describes the ``free''
evolution of the dynamical system given only its ``internal'' part, thus in some sense ``reversing" the Lax-Phillips approach.

Denote $E:= \clos \ran \alpha$ and define the dilation
space as the direct sum of
$K$ and the equivalents of incoming and
outgoing channels~ $\mathcal D_\pm = L^2(\mathbb
R_\pm, E)$,
\begin{equation*}
\mathcal H =  \mathcal D_- \oplus K \oplus \mathcal D_+
\end{equation*}
Elements of~$\mathcal H$ are represented as
three-component vectors~$(v_-, u,
v_+)$ with
$v_\pm \in \mathcal D_\pm$ and $u \in K$.
The Lax-Phillips theory suggests that the
dilation~$\mathscr L$
restricted
to~$\mathcal D_- \oplus \{0\} \oplus\mathcal D_+$
should be the self-adoint generator~$A$ of
the continuous unitary
group of right shifts~$\exp({\rm i}At) =U(t) : v(x) \mapsto v(x -t)$
in $L^2(\mathbb R, E)$.
By Stone's theorem, one has
\[
{\rm i}A v = \lim\limits_{t\downarrow 0}t^{-1}[U(t) v - v]
=  \lim\limits_{t\downarrow 0}t^{-1}\bigl(v(x-t) - v(x)\bigr) = - v^\prime (x),
\]
so that
the generator of~$U(t)$ is the operator~$A : v \mapsto {\rm i}dv/dx$.
Hence, the action of~$\mathscr L$ on the channels~$\mathcal D_\pm$
is defined by
$\mathscr L : (v_-,0, v_+) \mapsto ( {\rm i}v_-^\prime,0,{\rm i}v_+^\prime)$.
The self-adjointness of~$\mathscr L = \mathscr L^*$ and
the requirement~(\ref{eqn:DilationEq}) yield
the form of dilation~$\mathscr L$
as found in~\cite{MR0510053},
\begin{equation}\label{DilationForm}
 \mathscr L  \left(
 \begin{array}{l} v_-
                     \\ u
                     \\   v_+
\end{array}\!\!\right)
=
 \begin{pmatrix} {\rm i}\dfrac{dv_-}{dx}
                     \\[0.7em]A u + \dfrac{\alpha}{2} \left[v_+(0) + v_-(0)  \right]
                     \\[0.5em]  {\rm i}\dfrac{dv_+}{dx}
\end{pmatrix},
\end{equation}
defined on the domain
\[
 \dom(\mathscr L) = \left \{(v_-, u, v_+ ) \in \mathcal H \mid
 v_\pm \in W_2^1(\mathbb R_\pm, E), u \in \dom(A),
 v_+(0) - v_-(0) = i\alpha u
 \right\}
\]
Embedding theorems for the Sobolev space~$W_2^1$
guarantee the existence of  boundary values $v_\pm(0)$.
The ``boundary condition''~$v_+(0) - v_-(0) = {\rm i}\alpha u$
can be interpreted
as a concrete form of coupling between the incoming and outgoing channels
$\mathcal D_\pm$ realized by the
imaginary part of~$L$ acting
on $E.$
%

%
When $\alpha = 0$, the right hand side of~(\ref{DilationForm}) is the
orthogonal sum of two self-adjoint operators, that is,
the operator~$A$ on $K$
and the operator~${\rm i}d/dx$ acting
in the orthogonal sum of
channels~$L^2(\mathbb R, E )
= \mathcal D_-\oplus\mathcal D_+$.
The characteristic function of $L$ is the contractive
operator-valued
function defined by the formula
\begin{equation}
	\label{DissCharF}
S(z) =  I_E + {\rm i}\alpha(L^* - zI)^{-1}\alpha : E \to E,
\quad z \in \mathbb C_+
\end{equation}
According to the fundamental result of Adamyan and Arov~\cite{MR0206711}, the
function~$S$ coincides with
the scattering operator of the pair $(\mathscr L,
\mathscr
L_0)$ where $\mathscr L_0$ is defined by (\ref{DilationForm})
with $u = 0$ and  $\alpha = 0$.

\subsubsection{Extensions of symmetric operators~\cite{Drogobych}}

Consider the differential expression
\[
\ell y = -y'' +q(x)y
\]
in $K= L^2(\mathbb R_+)$ with a real function~$q$
such that the Weyl limit point case takes place.
Denote by $\varphi$ and $\psi$ the standard solutions to
the equation $\ell y = z y$ with $z
\in \mathbb C_+$, satisfying the boundary conditions
\[
\varphi(0,z) = 0, \quad  \varphi^\prime (0,z) = -1,
\quad \psi(0,z) = 1, \quad \psi^\prime(0,z) = 0
\]
Then the Weyl solution~$\chi = \varphi + m_\infty(z)\psi\in L^2(\mathbb R_+)$, where $m_\infty(z)$ is the Weyl
function pertaining to~$\ell$ and
corresponding to the boundary condition~$y(0) = 0$, is defined
uniquely.
The function $m_\infty(z)$ is analytic with positive
imaginary part for $z
\in \mathbb C_+$.

Define the operator $L$ in $K =L^2(\mathbb R_+)$
by the expression~$\ell$
supplied with the non-selfadjoint
boundary condition at $x = 0$
\[
( y^\prime - \left . h y) \right |_{x =0} = 0 , \quad \text{where} \quad  \im h
= \frac{\alpha^2}{2}, \quad \alpha > 0
\]
A short calculation ascertains  that $L$ is dissipative indeed.

The Pavlov's dilation of $L$ is the operator~$\mathscr L$ in the
space~$\mathcal H = \mathcal D_- \oplus K \oplus \mathcal D_+$,
where $\mathcal D_\pm = L^2(\mathbb R_\pm)$, defined
on elements
$( v_- , u, v_+) \in \mathcal H$
which satisfy
\begin{equation}\label{eqn:ExtensionsDomain}
\begin{aligned}
& v_\pm \in W_2^1(\mathbb R_\pm), \quad u,\ell u \in L^2({\mathbb R}_+),
\\[0.3em]
& \left . u^\prime - h u \right|_0 = \alpha v_-(0), \quad
 \left . u^\prime - \bar h u \right|_0 = \alpha v_+(0)
\end{aligned}
\end{equation}
The action of the operator $\mathscr L$ on this domain is set by the formula
\begin{equation}\label{eqn:ExtensionsAction}
 \mathscr L  \left(
 \begin{array}{l} v_-
                     \\ u
                     \\   v_+
\end{array}\!\!\right)
=
 \begin{pmatrix} {\rm i}\dfrac{dv_-}{dx}
                     \\[0.8em] \ell u
                     \\[0.4em]   {\rm i}\dfrac{dv_+}{dx}
\end{pmatrix}.
\end{equation}
The characteristic function of~$L$ is the scalar analytic in the
upper half-plane function~$S(z)$
given by
\[
S(z) = \frac{ m_\infty(z) - h }{m_\infty(z) - \bar h}, \quad \im(z) > 0.
\]
Note that since $\im h =\alpha^2/2$, the function $S(z)$ can be
rewritten in a form similar to that of the characteristic
function~(\ref{DissCharF}), i.~e.,
\begin{equation}\label{CharF_h}
 S(z) = 1 + {\rm i}\alpha\bigl(\bar h -
m_\infty(z)\bigr)^{-1}\alpha.
\end{equation}

\subsubsection{Pavlov's symmetric form of the dilation}\label{sec:pavlov_model}

According to the general theory~\cite{MR2760647},
once the characteristic function~$S$ is known, the
analysis of the completely non-selfadjoint part of
the operator~$L$ is reduced to the
analysis of~$S$.
Hence, the typical questions of the operator theory
(the spectral analysis, description of invariant
subspaces) are reformulated as problems pertaining
to analytic (operator-valued) functions.
Assume  that $L$ is completely non-selfadjoint
 and $\rho(L) \cap \mathbb C_\pm \neq \varnothing$.
Let $\mathscr L $ be its minimal self-adjoint dilation,
$S$ being the characteristic function of~$L$.
Owing to the general theory~\cite{MR2760647},
the operator~$L$ is unitary equivalent
to its model acting
in the spectral representation of~$\mathscr L$
in accordance with~(\ref{eqn:DilationEq}).
Recall that the characteristic
function~$S(z)$, $z \in \mathbb C_+$
is analytic in the upper half-plane taking values in
the set of contractions of~$E$,
\[
S(z) : E \to E, \quad \|S(z) \| \leq 1, \quad z \in \mathbb C_+
\]
Due to the operator version of Fatou's
theorem~\cite{MR2760647},
the nontangential boundary values of the
function~$S$
exist in the strong operator topology
almost everywhere on the real line.
Put~$S = S(k): =  \slim_{\varepsilon\downarrow 0}
S(k + {\rm i}\varepsilon)$
and~$S^* = S^*(k)
:= \slim_{\varepsilon\downarrow 0}[S(k + {\rm i}\varepsilon)]^*$,
both limits  existing for almost all~$k \in \mathbb R$.
The Fatou theorem guarantees that
the operators~$S(k)$ and $S^*(k)$ are
contractions on~$E$
for almost all~$k\in\mathbb R$.
The symmetric form of the dilation
is obtained
by completion
of the dense linear set in~$L^2(E) \oplus L^2(E)$
with respect to the norm
\begin{equation}\label{ModelSpace}
 \left\|\binom{\tilde g}{g}\right\|_{\mathscr H}^2 :=
 \int\limits_{\mathbb R}
  \left\langle
    \left(
      \begin{array}{cc} I &  S^* \\  S & I \end{array}
    \right)
    \binom{\tilde g}{g},
    \binom{\tilde g}{g}
  \right\rangle_{ E\oplus  E}
 dk,
\end{equation}
followed by factorisation by the elements of zero norm.
%
In the symmetric  representation, the incoming and outgoing
subspaces~$\mathcal  D_\pm$ admit their simplest possible
form.
On the other hand,
calculations related to
the space~$K$ can meet
certain difficulties,  since
the ``weight'' in~(\ref{ModelSpace})
can be singular.
Also note that the
elements of~$\mathscr H$ are not individual functions from
$L^2(E) \oplus L^2(E)$ but rather equivalence classes~\cite{NikolskiiKhrushchev, NikolskiiVasyunin1989}.
Despite these complications,
the Pavlov's symmetric model has been widely accepted in
the analysis of non-selfadjoint operators, and in particular of the operators
of mathematical physics.
Two alternative and equivalent forms
of the norm $\|\cdot\|_{\mathscr H}$ that
are easy to derive,
\begin{equation*}
\left\|\binom{\tilde g}{g}\right\|_{\mathscr H}^2 =
 \left\| S\tilde g + g \right\|_{L^2( E)}^2 +
 \left\|\Delta_* g \right\|_{L^2( E)}^2 =
\left\|\tilde g +  S^* g \right\|_{L^2( E)}^2 +
 \left\|\Delta \tilde g \right\|_{L^2( E)}^2,
\end{equation*}
where $\Delta := \sqrt{I- S^* S}$ and
$\Delta_* := \sqrt{I-  S S^*}$, show
that for each $\binom{\tilde g}{g} \in \mathscr H$
expressions
$ S\tilde g + g$, $\tilde g +  S^* g$, $\Delta
\tilde g$,
and $\Delta_*  g$ are in fact usual square summable
vector-functions
from~$L^2( E)$.
Moreover, due to these equalities the form~(\ref{ModelSpace})  is
positive-definite indeed and thus represents a norm.

The space
\[
\mathscr H = L^2\left(\begin{matrix}
                            I & S^* \\ S & I
                           \end{matrix}\right)
\]
with the norm defined
by~(\ref{ModelSpace}) is the space of spectral
representation for the self-adjoint
dilation~$\mathscr L$ of the operator~$L$. Henceforth we will denote the corresponding unitary mapping of $\mathcal{H}$ onto $\mathscr{H}$ by $\Phi$.
It means that
the operator of multiplication by the independent variable
acting on~$\mathscr H$, i.e., the
operator~$f(k) \mapsto k f(k) $, is unitary equivalent
to the dilation~$\mathscr L$.
Hence, for~$z \in\mathbb
C \setminus \mathbb R$, the
mapping~$\binom{\tilde g}{g} \mapsto (k - z)^{-1}
\binom{\tilde g}{g}$
is unitary equivalent to the resolvent~$(\mathscr L-z)^{-1}$
and therefore
\begin{equation*}
 (L - zI)^{-1} \simeq\left . P_{\mathscr K} (k - z)^{-1} \right
 |_{\mathscr K}, \quad z \in \mathbb C_-,
\end{equation*}
where the $\simeq$ sign is utilised to denote unitary equivalence.

The incoming and outgoing subspaces of the dilation
space~$\mathscr H$ admit the form
\begin{equation*}
 {\mathscr D}_+ := \binom{H^2_+(E)}{0}, \quad
 {\mathscr D}_- := \binom{0}{H^2_-( E)}, \quad
 \mathscr K := \mathscr H \ominus
  \left[
    {\mathscr D}_+ \oplus {\mathscr D}_-
  \right]
\end{equation*}
where $H_2^\pm(E)$ are the Hardy classes of $
E$-valued vector functions analytic in $\mathbb C_\pm$.
As usual~\cite{MR822228}, the functions from vector-valued Hardy
classes~$H_2^\pm(E)$ are identified with their
boundary values existing almost everywhere
on the real line.
They form two complementary mutually orthogonal subspaces
so that
$L^2(E) = H^2_+(E)\oplus H^2_-(E)$.

The image~$\mathscr K$ of $K$ under the spectral
mapping $\Phi$ of the dilation space~$\mathcal H$
to~$\mathscr H$ is the subspace
\begin{equation*}
 \mathscr K =
   \left\{
     \binom{\tilde g}{g} \in \mathscr H \,\, :\,\,
       \tilde g +  S^* g \in H^2_-( E), \,
        S \tilde g + g \in H^2_+ ( E)
   \right\}
\end{equation*}
The orthogonal projection~$P_\mathscr K$ from
$\mathscr H$ onto
$\mathscr K$ is defined by formula~(\ref{2.1.3.2})
on a dense set of functions from
$L^2(E) \oplus L^2(E)$ in $\mathscr H$
\begin{equation}\label{2.1.3.2}
  P_{\mathscr K} \binom{\tilde g}{g} =
    \binom{\tilde g - P_+ ( \tilde g +   S^* g)}
          {       g - P_- ( S \tilde g +   g)},
    \quad \tilde{g} \in L^2( E), \; g \in L^2( E).
\end{equation}
Here $P_\pm$ are the orthogonal projections of $L^2$
onto the Hardy
classes~$H_2^\pm$.

Further information on model representations can be found in the
series of papers~\cite{NikolskiiKhrushchev,NikolskiiVasyunin1986,
NikolskiiVasyunin1989, NikolskiiVasyunin1998}
and the treatise~\cite{Nikolski}.


%

%
%
%
%

\subsubsection{Naboko's functional model of non-selfadjoint operators}

The development of the functional
model approach for
contractions inspired the
search for such models of
non-dissipative operators.
The attempts to follow the blueprints of Sz.-Nagy-Foias
and Lax-Philitps meet serious challenges
rooted in the absence of a proper self-adjoint dilation for
non-dissipative operators: the dilatation in this case
is a self-adjoint operator acting on a space with an
indefinite metric~\cite{Davis}.
Consequently, the characteristic function of a non-dissipative
operator is an analytic operator-function, contractive with
respect to an indefinite metric~\cite{DavisFoias}, which
considerably hinders any further progress in this direction.
We mention the works~\cite{Ball, McEnnis}, the
monograph~\cite{Kuzhel} and references therein
for more details and examples.

An alternative approach was suggested in the late seventies with the
publication of papers~\cite{MR0500225,nabokozapiski2} and
especially~\cite{MR573902} by S.~Naboko who
found a way to represent a non-dissipative
operator in a model space of a suitably
chosen dissipative one.
%
%
%
We refer the reader to the relevant section of the paper on Sergey Naboko's mathematical heritage in the present volume for the details of
the mentioned approach and the relevant references. In the next section of the present paper we outline the main ingredients of an adaptation of the latter to the setting of extensions of symmetric operators, which was developed by Sergey's students.

We mention that this set of techniques allows one to significantly advance the spectral analysis of non-selfadjoint operators, including the definition of the absolutely continuous and singular subspaces and the study of spectral resolutions of identity. In particular, of major importance is the possibility to construct the wave and scattering operators in a natural representation. It should be noted that the self-adjoint scattering theory (and all the major versions of the latter) turns out to be included as a particular case of a much more general non-selfadjoint one.

\subsection{Functional model for a family of
extensions of a symmetric operator}
\label{sec:ryzh-funct-model}

The generic model constructed in~\cite{Mak_Vas, Ryzhov_closed}
lends itself as a powerful and universal tool for the analysis of
(completely) non-selfadjoint and (completely) non-unitary operators.
Since characteristic functions of such operators are
essentially unique and define the operators up to unitary
equivalence, all model considerations are immediately
available once this
function is known.
In many applications, however, the results sought  need to be
formulated in terms of the problem itself (i.e., in the natural terms), rather than in the abstract language
of characteristic functions (and their transforms).
One prominent example when the general theory is not
sufficient is the setting of extensions of symmetric operators and the associated setting of operators pertaining to boundary-value problems.
Within this setup, the results  are expected to be formulated
as statements concerning
the symmetric operator itself and the
relevant properties of the
extension parameters.
Some results in this direction were obtained by B.~Solomyak
in \cite{Solomyak2}, but the related calculations tend to be rather tedious due to the reduction to the case of contractions which is required.
%


The extension theory of symmetric operators, especially
differential  operators, greatly benefits from the abstract
framework known as the boundary triples theory.
The basic concepts of this operator-theoretic approach
can be found in the textbook~\cite{MR2953553}
by K.~Schm{\"u}dgen.
The recent monograph~\cite{BHS} contains a detailed
treatment of this area.

It is therefore quite natural to utilise this approach in conjunction with the functional models techniques outlined above in the analysis of non-selfadjoint extensions of symmetric operators.
This section briefly outlines the results pertaining to the functional model construction for dissipative and non-dissipative extensions of symmetric operators and the related developments, including an explicit construction of the wave and scattering operators and of the scattering matrices. Since all the considerations in this area are essentially parallel to the ones of Naboko in his development of spectral theory for additive perturbations of self-adjoint operators, one can consider this narrative as a rather detailed exposition of Naboko's ideas and results in a particular case, important for applications.

\subsubsection{Boundary triples}

The fundamentals of the boundary triples theory
have been introduced
in~\cite{BHS, MR2953553}, see also
references therein.
%

%

Denote by~$A$ a closed and densely defined symmetric
operator on the separable Hilbert space $H$ with the
domain~$\dom A$, having equal deficiency
indices~$0 < n_{+}(A)= n_{-}(A) \leq\infty$.

\begin{definition}[\cite{MR0365218}]\label{def:bonudaryTriple}
A triple $\{\mathcal{K}, \Gamma_0, \Gamma_1 \}$ consisting of an auxiliary
Hilbert space $\mathcal{K}$ and linear mappings $\Gamma_0, \Gamma_1$ defined
everywhere on  $\dom A^*$ is called a \emph{boundary triple} for $A^*$ if the following
conditions are satisfied:
\begin{enumerate}
\item
The abstract Green's formula is valid
    \begin{equation}\label{GreenFormula}
        (A^*f,g)_H - (f, A^*g)_H = (\Gamma_1 f, \Gamma_0 g)_{\mathcal{K}} -
        (\Gamma_0 f, \Gamma_1 g)_{\mathcal{K}},\quad f,g \in \dom A^*
    \end{equation}
\item
  For any $Y_0, Y_1 \in{\mathcal{K}}$ there exist $f \in
\dom A^*$, such that $\Gamma_0 f = Y_0$, $\Gamma_1 f = Y_1$.
In other words,
the mapping~$f \mapsto \Gamma_0 f \oplus \Gamma_1 f $, $f \in \dom A^*$
to ${\mathcal{K}}\oplus {\mathcal{K}}$ is surjective.
\end{enumerate}
\end{definition}
\noindent
It can be shown~(see \cite{MR0365218}) that a boundary
triple for~$A^*$
exists assuming only~$ n_{+}(A)=n_{-}(A) $.
Note also that a boundary triple
is not unique.
Given any bounded self-adjoint operator~$\Lambda = \Lambda^*$
on $\mathcal K$, the
collection~$\{\mathcal K, \Gamma_0 , \Gamma_1 + \Lambda \Gamma_0\}$
is a
boundary triple for~$A^*$ as well,
provided that~$\Gamma_1 + \Lambda \Gamma_0$
is surjective.

\begin{definition}\label{def:WeylFunction}
Let $\mathscr T = \{\mathcal{K}, \Gamma_0, \Gamma_1 \}$
be a boundary triple
of~$A^*$.
The Weyl function of $A^*$ corresponding to~$\mathscr T$ and
denoted $M(z)$, $z \in\mathbb C\setminus \mathbb R$ is an analytic
operator-function with a positive imaginary part for~$z \in \mathbb C_+$
(i.e., an operator $R$\nobreakdash-function) with values in
the algebra of bounded operators on $\mathcal K$ such that
\[
  M(z)\Gamma_0 f_z = \Gamma_1 f_z, \quad f_z \in {\rm ker}(A^* -zI),
  \quad z \notin \mathbb R
\]
For $z\in\mathbb C \setminus\mathbb R$ we
have~$(M(z))^* = (M(\bar z))$ and
$\im (z) \cdot \im(M(z)) > 0$.
\end{definition}

\begin{definition}
An extension~$\mathcal A$ of a closed densely defined
symmetric operator~$A$ is
called \emph{almost solvable (a.s.)} and denoted $\mathcal A = A_B$ if
there exist a boundary
triple~$\{\mathcal{K}, \Gamma_0, \Gamma_1 \}$ for~$A^*$
and a bounded operator~$B : \mathcal{K} \to \mathcal K$ defined
everywhere in $\mathcal K$ such that
\begin{equation*}
    f \in \dom A_B \iff  \Gamma_1 f =  B \Gamma_0 f
\end{equation*}
\end{definition}
\noindent
This definition implies the
inclusion~$\dom A_B \subset \dom A^*$
and that~$A_B$ is a restriction of~$A^*$ to
the linear set $\dom A_B := \{f \in \dom A^* : \Gamma_1 f =
B \Gamma_0 f\}$.
In this context, the operator~$B$ plays the role of a
parameter for the
family of extensions~$\{A_B \mid B : \mathcal K \to \mathcal K\}$.

It can be shown (see~\cite{CherKisSilva1} for references)
that if the
deficiency indices $n_\pm(A)$ are equal
and $A_B$ is an almost
solvable extension of~$A$, then
the resolvent set of~$A_B$ is
not empty (i.e. $A_B$ is maximal),
both $A_B$ and $(A_B)^* = A_{B^*}$
are restrictions of $A^*$ to their
domains, and
$A_B$ and $B$
are selfadjont (dissipative) simultaneously.
The spectrum of~$A_B$ coincides with the
set of points~$z_0 \in \mathbb C$ such that $(M(z_0) -B)^{-1}$
does not admit analytic continuation into it.

\subsubsection{Characteristic functions}\label{sec:char_functions_ext}

Assume that the parameter~$B$ of
an almost solvable extension $A_B$
is completely non-selfadjoint.
It can be represented as the sum of its real
and imaginary part
\[
 B = B_R + {\rm i}B_I, \quad B_R =B_R^*, \quad B_I = B_I^*
\]
These parts are well defined since $B$ is bounded.
The Green's formula implies that
for~$B_I \neq 0$ the imaginary part
of~$A_B$ (in the sense of its form) is non-trivial,
i.~e.,~$\im(A_B u, u) \neq 0$
at least for some $u \in\dom(A_B)$.
Hence~$A_B$ in this case is not a
self-adjoint operator.
It appears highly plausible that complete non-selfadjointness of~$A_B$
can be derived solely from complete non-selfadjointness of~$B$, assuming that~$A$
has no reducing self-adjoint parts.
However, no direct proof of this assertion
seems to be available in the existing
literature.

According to~(\ref{DissCharF}),
the characteristic function of~$B$
has the form
\[
 \Theta_B(z) = I_E + {\rm i}J\alpha(B^* - zI)^{-1}\alpha : E \to E,
 \quad z \in \rho(B^*),
\]
where $\alpha:=\sqrt{2|B_I|}$, $J:={\rm sign} B_I,$
and~$E:=\clos\ran(\alpha)$.
On the other hand, direct calculations according
to~\cite{Strauss1960} lead to the following
representation for the characteristic
function~$\Theta_{A_B}: E \to E$ of the
non-selfadjoint part of the extension~$A_B$
\[
\Theta_{A_B} = I_E + {\rm i}J \alpha (B^* - M(z))^{-1}\alpha,
\quad z \in \rho(A_B^*).
\]
These two formulae confirm an earlier observation
that goes back to B.~S.~Pavlov's work~\cite{Drogobych},
see~(\ref{CharF_h}) above.
The function~$\Theta_{A_B}$ is obtained
from $\Theta_B$ by  the
substitution of~$M(z)$ for $zI_E$,
$z\in \mathbb C_+$
\[
 \Theta_{A_B}(z) = \Theta_B(M(z)), \quad z \in \rho(A_B^*).
\]

Alongside $A_B$
introduce the dissipative almost solvable extension~$A_+$
parameterized by~$B_+:= B_R + {\rm i}|B_I|$.  Note that the characteristic function~$S$
of~$A_+$
is given by ({\it cf.} (\ref{DissCharF}))
\begin{equation}\label{eq:charf_ext}
 S(z) = I_E + {\rm i}\alpha(B_+^* - zI)^{-1}\alpha : E \to E,
 \quad z \in \mathbb{C}_+.
\end{equation}
%
%

%
%
Calculations of~\cite{MR2330831} (cf. \cite{MR0500225})  show that the
characteristic functions of $A_B$ and
$A_+$ are
related via an
operator linear-fractional transform known as
Potapov-Ginzburg transformation,
or PG-trans\-form~\cite{AzizovIokhvidov}.
This fact is essentially geometric.
It connects contractions on Kre\u{\i}n spaces
(i.e., the spaces with an indefinite metric defined
by the involution~$J = J^* = J^{-1}$) with contractions
on Hilbert spaces endowed with the regular metric.
The PG-trans\-form is invertible and
the following assertion pointed out in~\cite{MR0500225}
holds.
\begin{proposition}
\!\!The characteristic function~$\Theta_{A_B}$ is $J$\nobreakdash-contractive on its domain
and the PG\nobreakdash-trans\-form
maps it to the contractive characteristic
function $S$ of~$A_+$
as follows:
\begin{equation}\label{PGtransform}
\Theta_{A_B} \mapsto S = - (\chi^+ - \Theta_{A_B} \chi^-)^{-1}(\chi^- - \Theta_{A_B} \chi^+),
\qquad
S \mapsto \Theta_{A_B} = (\chi^- + \chi^+ S)(\chi^+ + \chi^- S)^{-1}
\end{equation}
where $\chi^\pm =(I_E \pm J)/2$ are orthogonal projections onto
subspaces of $\chi^+ E$ and  $\chi^- E$, respectively.
\end{proposition}
\noindent

It appears somewhat unexpected that two operator-valued functions
connected by formulae~(\ref{PGtransform})   can be explicitly
written down in terms of their ``main
operators''~$A_{B}$ and $A_+$.
This relationship between the characteristic
functions of~$A_B$ and
$A_+$  goes in fact
much deeper, see~\cite{Arov,AzizovIokhvidov}.
In particular, the self-adjoint dilation
of~$A_+$
and the $J$-self-adjoint dilation of $A_B$ are also related via a suitably
adjusted version of the PG\nobreakdash-transform.
Similar statements hold for the corresponding
linear systems or ``generating operators'' of the
functions~$\Theta_{A_B}$ and $S$, cf.~\cite{Arov,AzizovIokhvidov}.
This fact is crucial for the construction of a model of a
general closed and densely defined
non-selfadjoint operator,
see~\cite{Ryzhov_closed}.
%

%

 %
%
%
%
%
%

\subsubsection{Functional model for a family of
extensions}\label{sect:BoundaryTriplesModel}

Formulae of the previous secion are essentially the same as the
formulae of \cite{MR0500225}
connecting characteristic functions of
non-dissipative and dissipative operators.
Reasoning by analogy, this suggests an existence
of certain identities that would connect the resolvents of $A_B$ and $A_+$
corresponding to
parameters~$B = B_R + iB_I$ and $B_+ = B_R + i|B_I|$.
Such identities indeed exist; they are the
celebrated Kre\u{\i}n formulae for
resolvents of two extensions of a symmetric operator.
Their variant is readily derived within the framework of boundary
triplets, see~\cite{MR2330831} for calculations,
where all details of the following results can also be found.
\smallskip

The functional model of the dissipative extension~$A_+$ begins
with the derivation of its minimal
selfadjoint dilation~$\mathscr A$.
It is constructed following the recipe of
B. Pavlov~\cite{MR0365199,Drogobych,MR0510053}
and takes a form quite
similar to (\ref{eqn:ExtensionsDomain}), (\ref{eqn:ExtensionsAction})
\[
\mathscr A  \begin{pmatrix}
                v_- \\ u \\ v_+
             \end{pmatrix} =
             \begin{pmatrix}
              {\rm i}v^\prime_- \\[0.3em] A_+^* u \\[0.3em] {\rm i}v^\prime_+
             \end{pmatrix}, \quad
              \begin{pmatrix}
                v_- \\ u \\ v_+
             \end{pmatrix} \in \dom(\mathscr A)
\]
where $\dom(\mathscr A)$ consists of vectors~$(v_-,u,v_+) \in \mathcal H =
\mathcal D_- \oplus H\oplus \mathcal D_+$, with $v_\pm \in W_2^1(\mathbb R_\pm,
E)$, $u \in \dom(A_+^*)$ under  two ``boundary
conditions'' imposed on $v_\pm$ and $u$:
\[
 \Gamma_1 u - B_+ \Gamma_0 u = \alpha v_-(0), \quad
 \Gamma_1 u - B_+^* \Gamma_0 u = \alpha v_+(0)
\]

The functional model construction for  $A_+$ follows the recipe
by S.~Naboko~\cite{MR573902}. The following theorem holds.
\begin{theorem}
	\label{ModelTheoremDiss}
There exists a  mapping $\Phi$ from the
dilation space~$\mathcal H$ onto  Pavlov's model space~$\mathscr H$
defined by~(\ref{ModelSpace})
with the following properties
\begin{enumerate}
  \item  
    $\Phi$ is  isometric.
  \item  
    $\tilde g +  S^* g = \mathscr F_+  h$,
    $  S \tilde g + g = \mathscr F_-  h$, where
    $\binom{\tilde g}{g} = \Phi  h$, $ h \in \mathcal H$
  \item
    $
   \Phi \circ (\mathscr L - zI )^{-1} = (k - z)^{-1} \circ \Phi,
   \quad z \in \mathbb C \setminus\mathbb R
    $
  \item
    $ \Phi\mathcal H  = \mathscr H$, \quad $\Phi \mathcal D_\pm = {\mathscr
D}_\pm$, \quad $\Phi \mathcal K = \mathscr K$
  \item
    $\mathscr F_\pm \circ (\mathscr L - zI)^{-1} = (k - z)^{-1} \circ
    \mathscr F_\pm,
     \quad z\in \mathbb C \setminus \mathbb R$.
\end{enumerate}
where bounded maps~$\mathscr F_\pm : \mathcal H \to L^2(\mathbb R, E)$
are defined by the formulae
\begin{align*}
  \mathscr F_+ : h &  \mapsto - \frac{1}{\sqrt{2\pi}}\,
 \alpha \Gamma_0( A_+ - k + {\rm i}0)^{-1}
u +
 S_+^*(k) {\hat{v}}_-(k) +    \hat{v}_+(k),
  \\
  \mathscr F_- : h &  \mapsto - \frac{1}{\sqrt{2\pi}}\,
   \alpha \Gamma_0(A_+^*-k-{\rm i}0)^{-1}
u + {\hat{v}}_-(k) + S_+(k) \hat{v}_+(k),
 \end{align*}
where $h = (v_-, u, v_+) \in \mathcal H$ and $\hat v_\pm$ are the Fourier
transforms of $v_\pm \in L^2(\mathbb R_\pm, E)$.
%
\end{theorem}


Using the
Kre\u{\i}n formulae, which play the same role as
Hilbert resolvent identities in the case of additive perturbations,
one can obtain results similar to those of \cite{MR573902} and obtain an explicit description of $(A_+-z)^{-1}$ in the functional model representation.
An analogue of  this result
for more general extensions of~$A$
corresponding to a choice of parameter~$B$ in the
form~$B_R+\alpha\varkappa\alpha/2$
with bounded~$\varkappa:E\to E$ and~$B_R = B_R^*$
is proven along the same lines.
This program was realized in~\cite{Ryzhov_closed} for a particular
case~$\varkappa = {\rm i}J$ and in~\cite{CherKisSilva1}
for the family of
extensions~$A_\varkappa$ parameterized
by~$B_\varkappa =\alpha \varkappa\alpha/2.$
The latter form of extension parameter
utilizes the possibility of ``absorbing''
the part~$B_R= B_R^*$ into
the map~$\Gamma_1$, i.~e., passing from the
boundary triple~$\{\mathcal K, \Gamma_0, \Gamma_1\}$
to the triple~$\{\mathcal K,
\Gamma_0, \Gamma_1 + B_R \Gamma_0\}$.

\subsubsection{Smooth vectors and the absolutely continuous subspace}
\label{sec:smooth-vectors-scatt}
Here we characterise the absolutely continuous spectral subspace for an almost solvable extension of a densely defined symmetric operator with equal (possibly infinite) deficiency indices. The procedure we follow is heavily influenced by the ideas of Sergey Naboko, see \cite{MR573902, MR1252228} and is carried out essentially in parallel to the exposition of \cite{MR573902}. In contrast to the mentioned works, dealing with additive perturbations of self-adjoint operators, we are dealing with the case of extensions, self-adjoint and non-self-adjoint alike. The narrative below follows the argument presented in our papers \cite{CherednichenkoKiselevSilva,CherKisSilva1}.

Since we are not limiting the consideration to the case of self-adjoint operators, we first require the notion of the absolutely continuous spectral subspace applicable in the non-self-adjoint setup. In the functional model space $\mathscr{H}$ introduced in Section \ref{sec:pavlov_model} constructed based on the characteristic function $S(z)$ introduced in Section \ref{sec:char_functions_ext}
consider two subspaces $\mathscr{N}^\varkappa_\pm$ defined as follows:
\begin{equation*}
   \mathscr{N}^\varkappa_\pm:=\left\{\binom{\widetilde{g}}{g}\in\mathscr{H}:
  P_\pm\left(\chi_\varkappa^+(\widetilde{g}+S^*g)+\chi_\varkappa^-(S\widetilde{g}
  +g)\right)=0\right\},
\end{equation*}
where \begin{equation*}
  \chi_\varkappa^\pm:=\frac{I\pm{\rm i}\varkappa}{2}.
\end{equation*}
and $P_\pm$ are orthogonal projections onto their respective Hardy classes, as above.

These subspaces have a characterisation in terms of the resolvent of the operator $A_\varkappa.$ This, again, can be seen as a consequence of a much more general argument (see {\it e.g.} \cite{Ryzhov_closed, Ryzh_ac_sing}). 
\begin{theorem}
  \label{lem:similar-to-naboko-thm-4}
Suppose that $\ker \alpha =0.$ The following characterisation holds:
\begin{equation*}
 \mathscr{N}^\varkappa_\pm=\left\{\binom{\widetilde{g}}{g}\in\mathscr{H}:
  \Phi(A_{\varkappa}-z I)^{-1}\Phi^*P_K\binom{\widetilde{g}}{g}=
P_K\frac{1}{k-z}\binom{\widetilde{g}}{g}
\text{ for all } z\in\complex_\pm\right\}\,.
\end{equation*}
Here $\Phi$  denotes the unitary mapping of the dilation space $\mathcal{H}$ onto $\mathscr{H}$, as above.
\end{theorem}
Consider the counterparts of $\mathscr{N}^\varkappa_\pm$ in the original
Hilbert space $H:$
\begin{equation*}
  \widetilde{N}_\pm^\varkappa:=\Phi^*P_K\mathscr{N}^\varkappa_\pm\,,
\end{equation*}
which are linear sets albeit not necessarily subspaces.
In a way similar to \cite{MR573902}, one introduces the set
\begin{equation*}
  \widetilde{N}_{\rm e}^\varkappa:=\widetilde{N}_+^\varkappa\cap
 \widetilde{N}_-^\varkappa
\end{equation*}
of so-called \emph{smooth vectors} and its closure $N_{\rm e}^\varkappa:=\clos(\widetilde{N}_{\rm e}^\varkappa).$



The next assertion ({\it cf. e.g.} \cite{Ryzhov_closed, Ryzh_ac_sing}, for the case of general non-selfadjoint operators), is an alternative non-model characterisation of the linear sets
$\widetilde{N}_\pm^\varkappa$.
\begin{theorem}
  \label{lem:on-smooth-vectors-other-form}
The sets $\widetilde{N}_\pm^\varkappa$ are described as follows:
\begin{equation*}
\widetilde{N}_\pm^\varkappa=\{u\in\cH: \alpha\Gamma_0(A_{\varkappa}-z
I)^{-1}u\in H^2_\pm(E)\}.
\end{equation*}
\end{theorem}

Moreover, one shows  that for the functional model image of $\tilde{N}^\varkappa_{\rm e}$ the following representation holds:
\begin{align}
&\Phi\widetilde{N}^\varkappa_{\rm e}=\biggl\{P_K\binom{\widetilde{g}}{g}\in\mathscr{H}:\nonumber
\\
&\binom{\widetilde{g}}{g}\in\mathscr{H}\ {\rm satisfies}\
  \Phi(A_{\varkappa}-z I)^{-1}\Phi^*P_K\binom{\widetilde{g}}{g}=
P_K\frac{1}{k-z}\binom{\widetilde{g}}{g}\ \ \ \forall\,z\in{\mathbb C}_-\cup{\mathbb C}_+\biggr\},
\label{New_Representation}
\end{align}
which motivates the term ``the set of smooth vectors'' used for $\widetilde{N}^\varkappa_{\rm e}$.
(Note that the inclusion of the right-hand side of (\ref{New_Representation}) into $\Phi\tilde{N}^\varkappa_{\rm e}$ follows immediately from Theorem \ref{lem:similar-to-naboko-thm-4}.)


The above Theorem together with Theorem \ref{thm:on-smooth-vectors-a.c.equality} below motivates generalising the notion of the absolutely continuous subspace $\cH_{\rm ac}(A_{\varkappa})$ to the case of non-selfadjoint extensions $A_\varkappa$ of a symmetric operator $A,$ by identifying it with the set $N^\varkappa_{\rm e}.$ This generalisation follows in the footsteps of the corresponding definition by Naboko \cite{MR573902} in the case of additive perturbations (see also \cite{Ryzhov_closed, Ryzh_ac_sing} for the general case).
\begin{definition}
\label{abs_cont_subspace}
For a symmetric operator $A,$ in the case of a non-selfadjoint extension $A_\varkappa$ the absolutely continuous subspace $\cH_{\rm ac}(A_{\varkappa})$  is defined by the formula $\cH_{\rm ac}(A_{\varkappa}):=N^\varkappa_{\rm e}.$

In the case of a self-adjoint extension $A_\varkappa$, we understand $\cH_{\rm ac}(A_{\varkappa})$ in the sense of the classical definition of the absolutely continuous subspace of a self-adjoint operator.
\end{definition}

It turns out that in the case of self-adjoint extensions a rather mild additional condition guarantees that the non-self-adjoint definition above is equivalent to the classical self-adjoint one. Namely, we have the following

\begin{theorem}
  \label{thm:on-smooth-vectors-a.c.equality}
  Assume that $\varkappa=\varkappa^*,$  ${\rm ker}(\alpha)=\{0\}$
  and let  $\alpha\Gamma_0(A_{\varkappa}-z I)^{-1}$ be a Hilbert-Schmidt
  operator for at least one point $z\in\rho(A_\varkappa)$. If
  $A$ is completely non-selfadjoint, then the definition $\cH_{\rm ac}(A_{\varkappa})=N^\varkappa_{\rm e}$ is
  equivalent to the classical definition of the absolutely continuous subspace of a self-adjoint operator, {\it i.e.}
  \begin{equation*}
    N_{\rm e}^\varkappa = \cH_{\rm ac}(A_{\varkappa})\,.
  \end{equation*}
\end{theorem}

\begin{remark}
Alternative conditions, which are even less restrictive in general, that guarantee the validity of the assertion of Theorem
\ref{thm:on-smooth-vectors-a.c.equality} can be obtained along the lines of \cite{MR1252228}.
\end{remark}

\subsubsection{Wave and scattering operators}
\label{sec:wave-operators}
The results of the preceding section allow us, see \cite{CherednichenkoKiselevSilva,CherKisSilva1}, to calculate the wave
operators for any pair $A_{\varkappa_1},A_{\varkappa_2}$, where
$A_{\varkappa_1}$ and $A_{\varkappa_2}$ are two different extensions of a symmetric operator $A$, under the additional
assumption that the operator $\alpha$ has a trivial kernel.
For simplicity,
in what follows we set $\varkappa_2=0$ and write
$\varkappa$ instead of $\varkappa_1$. Note that $A_0$ is a self-adjoint
operator, which is convenient for presentation purposes.

In order to compute the wave operators of this pair, one first establishes the model representation for the function
$\exp(iA_\varkappa t)$, $t\in\reals$, of the
operator $A_\varkappa,$ evaluated on the set of smooth vectors $\widetilde{N}_{\rm e}^\varkappa.$ Due to \eqref{New_Representation}, it is easily shown that on this set $\exp(iA_\varkappa t)$ acts as an operator of multiplication by $\exp(ikt)$. We then utilise the following result.

\begin{proposition}(\cite[Section 4]{MR573902})
  \label{prop:pre-wave-op}
  If $\Phi^*P_K\binom{\widetilde{g}}{g}\in\widetilde{N}_{\rm e}^\varkappa$
  and $\Phi^*P_K\binom{\widehat{g}}{g}\in\widetilde{N}_{\rm e}^0$
  (with the same element\footnote{Despite the fact that
    $\binom{\widetilde{g}}{g}\in\mathscr{H}$ is nothing but a symbol,
    still $\widetilde{g}$ and $g$ can be identified with vectors in
    certain $L^2(E)$ spaces with  operators ``weights'', see
    details below in Section~\ref{sec:spectral-repr-ac}. Further, we recall that even then
    for $\binom{\widetilde{g}}{g}\in\mathscr{H}$, the components $\widetilde{g}$ and $g$ are not, in general, \emph{independent} of each other.} $g$), then
\begin{equation*}
  \norm{\exp(-{\rm i}A_\varkappa t)\Phi^*P_K\binom{\widetilde{g}}{g}-\exp(-{\rm i}A_0
  t)\Phi^* P_K\binom{\widehat{g}}{g}}_{\mathscr H}\convergesto{t}{-\infty}0.
\end{equation*}
\end{proposition}

It follows from Proposition \ref{prop:pre-wave-op}  that whenever $\Phi^*P_K\binom{\widetilde{g}}{g}\in\widetilde{N}_{\rm e}^\varkappa$
and $\Phi^*P_K\binom{\widehat{g}}{g}\in\widetilde{N}_{\rm e}^0$ (with the same second component $g$), formally one has
\begin{align*}
\lim_{t\to-\infty}{\rm e}^{{\rm i}A_0t}{\rm e}^{-{\rm i}A_\varkappa t}\Phi^*P_K\binom{\widetilde{g}}{g}&=\Phi^*P_K\binom{\widehat{g}}{g}\\
&=\Phi^*P_K\binom{-(I+S)^{-1}(I+S^*)g}{g}\,.
\end{align*}

In view of the classical definition of the wave operator of a pair of self-adjoint operators, see {\it e.g.} \cite{MR0407617},
\begin{equation*}
  W_\pm(A_0,A_\varkappa):=\slim_{t\to\pm\infty}{\rm e}^{{\rm i}A_0t}{\rm e}^{-{\rm i}A_\varkappa t}P_{\rm ac}^\varkappa,
\end{equation*}
where $P_{\rm ac}^\varkappa$ is the projection onto the absolutely continuous subspace of $A^\varkappa,$ we obtain that, at least formally, for $\Phi^*P_K\binom{\widetilde{g}}{g}\in\widetilde{N}_{\rm e}^\varkappa$ one has
\begin{equation}
  \label{eq:formula-0-kappaw-}
  W_-(A_0,A_\varkappa)\Phi^*P_K\binom{\widetilde{g}}{g}=\Phi^*P_K\binom{-(I+S)^{-1}(I+S^*)g}{g}\,.
\end{equation}

By  considering the case $t\to+\infty$, one also obtains
\begin{equation*}
W_+(A_0,A_\varkappa)\Phi^*P_K\binom{\widetilde{g}}{g}=
\lim_{t\to+\infty}{\rm e}^{{\rm i}A_0t}{\rm e}^{-{\rm i}A_\varkappa t}\Phi^*P_K\binom{\widetilde{g}}{g}
=\Phi^*P_K\binom{\widetilde{g}}{-(I+S^*)^{-1}(I+S)\widetilde{g}}
\end{equation*}
again for $\Phi^*P_K\binom{\widetilde{g}}{g}\in\widetilde{N}_{\rm e}^\varkappa$.

Further, the definition of the wave operators $W_\pm(A_\varkappa,A_0)$
\begin{equation*}
  \norm{{\rm e}^{-{\rm i}A_\varkappa
      t}W_\pm(A_\varkappa,A_0)\Phi^*P_K\binom{\widetilde{g}}{g}-{\rm e}^{-{\rm i}A_0t}\Phi^*P_K\binom{\widetilde{g}}{g}
}_{\mathscr H}\convergesto{t}{\pm\infty}0
\end{equation*}
yields, for all $\Phi^*P_K\binom{\widetilde{g}}{g}\in\widetilde{N}_{\rm e}^0,$
\begin{equation*}
  W_-(A_\varkappa,A_0)\Phi^*P_K\binom{\widetilde{g}}{g}=\Phi^*P_K\binom{-(I+\chi_\varkappa^-(S-I))^{-1}(I+\chi_\varkappa^+(S^*-I))g}{g}
\end{equation*}
and
\begin{equation}
 \label{eq:formula-kappa-0w+}
  W_+(A_\varkappa,A_0)\Phi^*P_K\binom{\widetilde{g}}{g}=\Phi^*P_K
\binom{\widetilde{g}}{-(I+\chi_\varkappa^+(S^*-I))^{-1}
(I+\chi_\varkappa^-(S-I))\widetilde{g}}.
\end{equation}

In order to rigorously justify the above formal argument, {\it i.e.} in order to prove the existence and completeness of the wave operators, one needs to first show that
the right-hand sides of the formulae \eqref{eq:formula-0-kappaw-}--\eqref{eq:formula-kappa-0w+} make sense on dense subsets of the corresponding absolutely continuous subspaces, which is done in a similar way to \cite {MR1252228}. Below, we show how this argument works in relation to the wave operator \eqref{eq:formula-0-kappaw-} only, skipping the technical details in view of making the exposition more transparent.


Let $S(z)-I$ be of the class $\mathfrak {S}_\infty(\overline{ \mathbb {C}}_+),$ {\it i.e.} a compact analytic operator function in the upper half-plane up to the real line. Then so is $(S(z)-I)/2$, which is also uniformly bounded in the upper half-plane along with $S(z)$. We next use the result of \cite[Theorem 3]{MR1252228} about the non-tangential boundedness of operators of the form $(I+T(z))^{-1}$ for $T(z)$ compact up to the real line. We infer that, provided $(I+(S(z_0)-I)/2)^{-1}$ exists for some $z_0\in \mathbb C_+$ (and hence, see \cite{Brodski}, everywhere in $\mathbb C_+$ except for a countable set of points accumulating only to the real line), one has non-tangential boundedness of $(I+(S(z)-I)/2)^{-1}$, and therefore also of $(I+S(z))^{-1}$, for almost all points of the real line.

On the other hand, the latter inverse can be computed in $\mathbb C_+$:
\begin{equation}\label{id1}
\bigl(I+S(z)\bigr)^{-1}=\frac{1}{2}\bigl(I+ {\rm i}\alpha M(z)^{-1}\alpha/2\bigr).
\end{equation}

It follows from (\ref{id1}) and the analytic properties of $M(z)$ that the inverse $(I+S(z))^{-1}$ exists everywhere in the upper half-plane. Thus, Theorem 3 of \cite{MR1252228}  is indeed applicable, which yields that
$(I+S(z))^{-1}$ is $\mathbb R$-a.e. nontangentially bounded and, by the operator generalisation of the Calderon theorem (see \cite{Calderon}), which was extended to the operator context in \cite[Theorem 1]{MR1252228}, it admits measurable non-tangential limits in the strong operator topology almost everywhere on $\mathbb R$. As it is easily seen, these limits must then coincide with $(I+ S(k))^{-1}$ for almost all $k\in \mathbb R$.

Then the correctness of the formula (\ref{eq:formula-0-kappaw-}) for the wave operators follows: indeed, consider $\mathbbm{1}_n(k),$ the indicator of the set $\{k\in \mathbb R: \|(I+S(k))^{-1}\|\leq n\}.$
Clearly, $\mathbbm{1}_n(k)\to 1$ as $n\to \infty$ for almost all $k\in \mathbb R$. Next, suppose that $P_K(\tilde g, g)\in \tilde N_{\rm e}^\varkappa$. Then $P_K\mathbbm{1}_n(\tilde g, g)$ is shown to be a smooth vector as well as
$$
\binom{-(I+S)^{-1}\mathbbm{1}_n (I+S^*) g}{\mathbbm{1}_n g}\in \mathscr H.
$$
It follows, by the Lebesgue dominated convergence theorem, that the set of vectors $P_K\mathbbm{1}_n(\tilde g, g)$ is dense in $N_{\rm e}^\varkappa$.

Thus the following theorem holds.
\begin{theorem}
  \label{thm:existence-completeness-wave-operators}
  Let $A$ be a closed, symmetric, completely nonselfadjoint operator
  with equal deficiency indices and consider its extension $A_\varkappa$  under the assumptions that ${\rm ker}(\alpha)=\{0\}$  and that $A_\varkappa$ has at least one regular point in ${\mathbb C}_+$ and in ${\mathbb C}_-.$ If $S-I\in\mathfrak {S}_\infty(\overline{ \mathbb {C}}_+),$
  then the wave operators
  $W_\pm(A_0,A_\varkappa)$ and $W_\pm(A_\varkappa,A_0)$  exist on dense sets in $N_{\rm e}^\varkappa$ and
  $\mathcal{H}_{\rm ac}(A_0)$, respectively, and are given by the formulae
  (\ref{eq:formula-0-kappaw-})--(\ref{eq:formula-kappa-0w+}). The
  ranges of $W_\pm(A_0,A_\varkappa)$ and  $W_\pm(A_\varkappa,A_0)$ are dense in $\mathcal{H}_{\rm ac}(A_0)$ and $N_{\rm e}^\varkappa,$ respectively.\footnote{In the case when $A_\varkappa$ is self-adjoint, or, in general, the named wave operators are bounded, the claims of the theorem are equivalent (by the classical Banach-Steinhaus theorem) to the statement of the existence and completeness of the wave operators for the pair $A_0, A_\varkappa.$ Sufficient conditions of boundedness of these wave operators are contained in {\it e.g.} \cite[Section 4]{MR573902}, \cite {MR1252228} and references therein.}
\end{theorem}

\begin{remark}
\label{long_remark}
1. The condition $S(z)-I\in \mathfrak S_\infty(\overline{ \mathbb {C}}_+)$ can be replaced by the following equivalent condition: $\alpha M(z)^{-1}\alpha$ is nontangentially bounded almost everywhere on the real line, and $\alpha M(z)^{-1}\alpha\in \mathfrak S_\infty(\overline{ \mathbb {C}}_+)$ for $\Im z\geq 0$. 

2. The latter condition is satisfied \cite{Krein}, as long as the scalar function $\|\alpha M(z)^{-1}\alpha\|_{\mathfrak S_p}$ is nontangentially bounded almost everywhere on the real line for some $p<\infty,$ where ${\mathfrak S_p},$
$p\in(0, \infty],$ are the standard Schatten -- von Neumann classes of compact operators.

3. An alternative sufficient condition is the condition $\alpha\in\mathfrak S_2$ (and therefore $B_\varkappa \in \mathfrak S_1$), or, more generally, $\alpha M(z)^{-1}\alpha\in\mathfrak S_1,$ see \cite{MR1036844} for details.

\end{remark}


Finally, the scattering operator $\Sigma$ for the pair $A_\varkappa,$
$A_0$
is defined by
$$
\Sigma=W_{+}^{-1}(A_\varkappa,A_0) W_{-}(A_\varkappa,A_0).
$$
The above formulae for the wave operators lead ({\it cf.} \cite{MR573902}) to the following formula for the action of $\Sigma$ in the model representation:
\begin{equation}
\Phi\Sigma\Phi^*P_K \binom{\tilde g}{g}=
P_K \binom
{-(I+\chi_\varkappa^-(S-I))^{-1}(I+\chi_\varkappa^+(S^*-I))g}
{(I+S^*)^{-1}(I+S)(I+\chi_\varkappa^-(S-I))^{-1}(I+\chi_\varkappa^+(S^*-I))g},
\label{last_formula}
\end{equation}
whenever $\Phi^*P_K \binom{\tilde g}{g}\in\widetilde N_{\rm e}^0$. In fact, as explained
above, this representation holds on a dense linear set in
$\widetilde N_{\rm e}^0$ within the conditions of Theorem \ref{thm:existence-completeness-wave-operators}, which guarantees that
all the objects on the right-hand side of the formula (\ref{last_formula}) are correctly defined.

\subsubsection{Spectral representation for the absolutely continuous part of the operator $A_0$ and the scattering matrix}
\label{sec:spectral-repr-ac}
The identity
$$
\biggl\|P_K\binom{\tilde g}{g}\biggr\|^2_{\mathscr H}=\bigl\langle(I-S^*S)\tilde g, \tilde g \bigr\rangle
$$
which is derived in \cite[Section 7]{MR573902} for all
$P_K \binom{\tilde g}{g}\in\widetilde N_{\rm e}^0$  allows us to consider the
isometry
$F: \Phi\widetilde N_{\rm e}^0\mapsto L^2(E; I-S^*S)$
defined by the formula
\begin{equation*}
FP_K \binom{\tilde g}{g}=\tilde g.
\end{equation*}
Here $L^2(E; I-S^*S)$ is the Hilbert space of $E$-valued functions on
$\mathbb R$ square summable with the matrix ``weight'' $I-S^*S$.

Under the assumptions of Theorem \ref{thm:existence-completeness-wave-operators} one can show that the range of the operator $F$ is dense in the space $L^2(E; I-S^*S)$.
Thus, the operator $F$  admits an extension to the
unitary mapping between $\Phi N_{\rm e}^0$ and
$L^2(E; I-S^*S)$.

It follows that the self-adjoint operator $(A_0-z)^{-1}$
considered on $\widetilde N_{\rm e}^0$ acts as the multiplication by
$(k-z)^{-1},$ $k\in{\mathbb R},$
 in $L^2(E; I-S^*S)$. In particular,
if one considers the absolutely continuous ``part'' of the
operator $A_0$, namely the operator $A_0^{({\rm e})}:=A_0|_{N_{\rm e}^0},$ then  $F\Phi A_0^{({\rm e})}\Phi^*F^*$
is the operator of multiplication by the independent variable in the
space $L^2(E; I-S^*S)$.

In order to obtain a spectral representation from the above result, it
is necessary to diagonalise the ``weight'' in the definition of the
above $L^2$-space. The corresponding transformation is straightforward when, e.g.,
$\alpha=\sqrt{2}I.$ (This choice of $\alpha$ satisfies the conditions of Theorem \ref{thm:existence-completeness-wave-operators} {\it e.g.} when the boundary space $\mathcal K$ is finite-dimensional).  In this particular case one has
\begin{equation*}
S=(M-{\rm i}I)(M+{\rm i}I)^{-1},
\end{equation*}
and
consequently
\begin{equation*}
I-S^*S=-2i (M^*-{\rm i}I)^{-1}(M-M^*)(M+{\rm i}I)^{-1}.
\end{equation*}
Introducing the unitary transformation
\begin{equation*}
G: L^2(E; I-S^*S)\mapsto L^2(E; -2{\rm i}(M-M^*)),
\end{equation*}
 by the
formula $g\mapsto (M+{\rm i}I)^{-1}g$, one arrives at the fact that
$
GF\Phi A_0^{({\rm e})}\Phi^* F^*G^*$
is the operator of multiplication by the independent
variable in the space $L^2(E;-2{\rm i}(M-M^*))$.

\begin{remark}
The weight $M^*-M$ can be assumed to be naturally diagonal in many physically relevant settings, including the setting of quantum graphs considered in  Section \ref{sec:theor-appl-funct}.
\end{remark}

The above result
only pertains to the absolutely
continuous part of the self-adjoint operator $A_0$, unlike {\it e.g.}
the passage to the classical von Neumann direct integral, under which
the whole of the self-adjoint operator gets mapped to the
multiplication operator in a weighted $L^2$-space (see {\it e.g.}
\cite[Chapter 7]{MR1192782}). Nevertheless, it proves useful in scattering theory, since it
yields an explicit expression for the scattering matrix $\widehat{\Sigma}$ for the pair
$A_\varkappa,$ $A_0,$ which is the image of the scattering operator
$\Sigma$ in the spectral representation of the operator $A_0.$ Namely, one arrives at:

\begin{theorem}
The following formula holds:
\begin{equation}
\label{scat2}
\widehat{\Sigma}=GF\Sigma(GF)^*=
(M-\varkappa)^{-1}(M^*-\varkappa)(M^*)^{-1}M,
\end{equation}
where the right-hand side
represents the operator of multiplication by the corresponding function in the space $L^2(E;-2{\rm i}(M-M^*))$.
\end{theorem}




\subsection{Functional models for operators of boundary value problems}
\label{sec:funct-model-family}

The surjectivity condition in Definition~\ref{def:bonudaryTriple}
is a strong limitation that excludes many important problems
for extensions of symmetric operators with
infinite deficiency indices.
The standard textbook version of a boundary value
problem for the Laplace operator in a
bounded domain~$\Omega \subset \mathbb R^3$
with smooth boundary~$\partial \Omega$
is a typical example.
The ``natural'' boundary maps $\Gamma_0$ and $\Gamma_1$ are
two trace
operators~$\Gamma_0 : u \mapsto u|_{\partial\Omega}$,
$\Gamma_1 : u \mapsto -\partial u/\partial n|_{\partial\Omega}$,
where $\partial/\partial n$ denotes the derivative along the
exterior normal to the boundary~$\partial\Omega$.
The ranges of these operators do not coincide
with~$\mathcal H = L^2(\Omega)$
(the simplest possible Hilbert space
of functions defined on the boundary)
so the assumption of surjectivity does not
hold.
A simple argument reveals the source of
this problem: it appears due to
the limited
compatibility of the Green's formula
required to hold
on all of~$\dom(A^*)$ and the required
surjectivity of both
boundary maps~$\Gamma_0$,
$\Gamma_1$ also defined on
the same domain~$\dom(A^*)$.
This limitation of the boundary triples
formalism can be relaxed and the framework
extended to cover more general cases,
albeit at the cost of increased complexity,
see~\cite{BehrndtLanger2007, BehrndtLanger2012},
the book~\cite{BHS} and the
references therein for a detailed account.
%

%
%
%
%
%

Formally a more restrictive approach
applicable to semibounded
symmetric operators~$A$ and
not based on
the description of~$\dom A^*$
was developed
by M.~Birman, M.~Kre\u{\i}n and
M.~Vishik.
Despite its limited scope, this
theory proves to be indispensable
in applications to various
problems of ordinary and
partial differential operators.
The publication~\cite{AlonsoSimon}
contains a concise
exposition of these results.
It was realized later that the
Birman-Kre\u{\i}n-Vishik
method is closely related to the
theory of linear systems with
boundary control and to the
original ideas of M.~Liv\v{s}ic
from the open systems theory,
see e.~g.~\cite{Ryzhov_systems, Ryzh_spec}
in this connection.
Let us give a brief account of
relevant results
derived from the works cited above
and tailored to the purposes
of current presentation.

\subsubsection{Boundary value problem}

Let~$H$, $E$ be two separable
Hilbert spaces, $A_0$ an unbounded closed
linear operator on $H$ with the dense domain~$\dom A_0$
and $\Pi : E \to H$ a bounded linear
operator defined everywhere in~$E$.
\begin{theorem}\label{thm:BVP1}
 Assume the following:
 \begin{itemize}
  \item $A_0$ is self-adjoint and  boundedly invertible;
  \item There exists the left inverse~$\tilde
\Gamma_0$ of $\Pi$ so that $\tilde
\Gamma_0 \Pi \varphi = \varphi$ for all $\varphi \in E;$
  \item The intersection of $\dom A_0$ and $\ran \Pi$
   is trivial: $\dom A_0\cap \ran\Pi = \{0\}$.
\end{itemize}
%
%
Since~$\dom A_0$ and $\ran\Pi$ have trivial intersection,
the direct sum~$\dom A_0 \dotplus \ran\Pi$
form a dense linear set in~$H$ that
can be described as
$\{A_0^{-1} f + \Pi \varphi \mid f\in H, \varphi \in E\}$.
Define two linear operators~$A$ and $\Gamma_0$ with
the common
domain~$\dom A_0 \dotplus \ran\Pi$
as ``null extensions'' of $A_0$
and $\tilde \Gamma_0$ to the complementary
component of~$\dom A_0\dotplus \ran\Pi$
\begin{equation*}
 A: A_0^{-1} f + \Pi \varphi \mapsto f, \quad
  \Gamma_0 : A_0^{-1} f + \Pi\varphi  \mapsto \varphi ,
  \qquad f \in H, \varphi \in E
\end{equation*}
The spectral ``boundary value problem'' associated
with the pair~$\{A_0,\Pi\}$ satisfying these conditions
is the system of two linear
equations for the unknown
vector~$u \in \dom A := \dom A_0\dotplus \ran\Pi:$
\begin{equation}\label{eqn:BVP_system}
 \left\{\,\,
 \begin{aligned}
   (&A - zI) u = f \\
  &\Gamma_0 u = \varphi
 \end{aligned}
  \right . \qquad f \in H, \quad \varphi \in E,
\end{equation}
where $z \in \mathbb C$ is the spectral parameter.

Let $z\in \rho(A_0)$, $f\in H$, $\varphi \in E$.
Then the system~(\ref{eqn:BVP_system}) admits
the unique
solution~$u_z^{f,\varphi}$
given by the formula
\begin{equation*}
 u_z^{f,\varphi} = (A_0 - zI)^{-1} f + (I - z A_0^{-1})^{-1} \Pi \varphi
\end{equation*}
If the expression on the right hand side
is null for some~$f\in H$, $\varphi \in E$,
then $f = 0$ and $\varphi = 0$.
\end{theorem}


Let~$\Lambda$ be a linear operator on $E$ with the
domain~$\dom  \Lambda \subset E$ not necessarily dense in $E$.
Define the linear operator~$\Gamma_1$ on~$\dom\Gamma_1
 :=   \{A_0^{-1}f +\Pi\varphi \mid f \in H, \varphi \in \dom\Lambda \}$ as
the mapping
\[
\Gamma_1 :  A_0^{-1}f +\Pi\varphi  \mapsto \Pi^* f + \Lambda \varphi, \qquad
f\in H, \quad \varphi \in \dom\Lambda
\]
This definition implies~$\Lambda = \Gamma_1\Pi|_{\dom\Lambda}$.

Denote~$\mathscr D :=\dom\Gamma_1 =
 \{A_0^{-1}f +\Pi\varphi \mid f \in H, \varphi \in \dom\Lambda \}$.
 Obviously $\mathscr D \subset \dom A$.
 The next theorem is a form of the Green's formula for the operator~$A$.
\begin{theorem}\label{thm:GreenBVP}
 Assume that $\Lambda$ is
 selfadjoint (and therefore densely defined) in $E$. Then
 \[
(A u, v ) - (u, Av) = (\Gamma_1 u, \Gamma_0 v)_E  -
(\Gamma_0 u, \Gamma_1 v)_E, \qquad u, v \in \mathscr D
 \]
\end{theorem}

Notice the difference with the boundary
triples version~(\ref{GreenFormula}),
where the operator on the left hand side is the
adjoint of a symmetric operator.
In contrast,
Theorem~\ref{thm:GreenBVP} has no relation to
symmetric operators. The Green's formula is valid on a
set defined by the selfadjoint~$A_0$
and an arbitrarily chosen selfadjoint
operator~$\Lambda$.

\smallskip

Under the assumptions of Theorems~\ref{thm:BVP1}
and~\ref{thm:GreenBVP}
the operator-valued analytic function
\[
M(z) = \Lambda + z\Pi^*(I -zA_0^{-1})^{-1}\Pi, \quad  z \in \rho(A_0)
\]
defined on~$\dom M(z) = \dom\Lambda$ is the
Weyl function (cf.~\cite{BHS}) of the boundary value
problem~(\ref{eqn:BVP_system})
in the sense of equality (cf. Definition~\ref{def:WeylFunction})
\[
  M(z)\Gamma_0 u_z = \Gamma_1 u_z, \quad z \in \rho(A_0)
\]
where~$u_z= u_z^{f,\varphi}$ is the solution
to~(\ref{eqn:BVP_system}) with~$f = 0$
and~$\varphi \in \dom\Lambda$.

\smallskip

For the boundary value problem pertaining to the Laplace operator in a
bounded domain~$\Omega \subset \mathbb R^3$ with a
smooth boundary~$\partial \Omega$  the boundary maps~$\Gamma_0$,
$\Gamma_1$ are defined as
$\Gamma_0 : u \mapsto u|_{\partial\Omega}$,
$\Gamma_1 : u \mapsto -\partial u/\partial n|_{\partial\Omega}$.
Then~$A_0$ is the Dirichlet Laplacian in~$L^2(\Omega)$
and~$\Pi$ is the operator of harmonic continuation  from
the boundary space~$E = L^2(\partial \Omega)$ to $\Omega$.
The conditions~$\dom(A_0) \cap \ran(\Pi) = \{0\}$ and
$\Gamma_0\Pi = I_E$ are satisfied by virtue of
the embedding theorems for Sobolev
classes.
In this setting~$M(\cdot)$ is known as the Dirichlet-to-Neumann
map, which is a pseudodifferential operator
defined on~$H^1(\partial \Omega)$.
A special role of the operator~$\Lambda = M(0)$ for the study
of boundary value
problems was pointed out by M.~Vishik in his work~\cite{MR0051404} and
sometimes~$\Lambda$ in the settings of elliptic partial differential
operators is referred to as the {\it Vishik operator}.

\subsubsection{Family of boundary value problems}

General boundary value problems for the operator~$A$
have the form (cf.~\cite{MR0051404})
\begin{equation}\label{eqn:BVP_system_family}
\left\{\,\,
 \begin{aligned}
   (&A - zI) u = f \\
  &(\alpha\Gamma_0 + \beta\Gamma_1)u = \varphi
 \end{aligned}
  \right . \qquad f \in H, \quad \varphi \in E.
\end{equation}
Here $\alpha$, $\beta$ are linear operators on~$E$ such that
$\beta$ is bounded (and defined everywhere in $E$)
and $\alpha$ can be unbounded in which case $\dom\alpha \supset
\dom\Lambda = \mathscr D$.
Under certain verifiable conditions the
solutions to~(\ref{eqn:BVP_system_family}) exist and
are described by the following theorem.

\begin{theorem}[see~\cite{Ryzh_spec}]\label{thm:BVP_def}
Assume that the conditions of Theorems~\ref{thm:BVP1} and
\ref{thm:GreenBVP} are satisfied
and that the operator
sum~$\alpha + \beta \Lambda$ is correctly defined
on~$\dom\Lambda$ and closable in $E$.
Then~$\alpha + \beta M(z)$, $z \in \rho(A_0)$ is also
closable as an
additive perturbation of $\alpha + \beta \Lambda$
by the bounded operator
$M(z) - \Lambda$.
Denote by~$\mathscr B(z)$ the closure
of $\alpha + \beta M(z)$, $z\in
\rho(A_0)$ and
let $\mathscr B = \mathscr B(0)$.
\begin{itemize}
 \item
Consider the Hilbert space~$\mathscr H_{\mathscr B}$
formed by the vectors~$\left\{u = A_0^{-1} f + \Pi\varphi \mid
f \in H, \varphi \in \dom\mathscr B\right\}$ and endowed
with the norm
\[
 \|u\|_{\mathscr B} = \left( \|f\|^2 + \|\varphi\|^2 + \|\mathscr B \varphi
\|^2\right)^{1/2}
\]
The formal sum~$\alpha \Gamma_0 + \beta\Gamma_1$
is a bounded map from the Hilbert space~$\mathscr H_{\mathscr B}$ to $E$.
Note that
the summands in~$(\alpha \Gamma_0 + \beta\Gamma_1)u$,
$u\in \mathscr H_{\mathscr B}$  need not be defined individually.
\item
Assume that for some~$z \in \rho(A_0)$ the operator~$\mathscr B(z)$ has a bounded
inverse $[\mathscr B(z)]^{-1}$.
Then the problem~(\ref{eqn:BVP_system_family}) is uniquely solvable.
Under this condition there exists a closed operator~$A_{\alpha, \beta}$
with dense domain
\[
\dom A_{\alpha, \beta} = \left\{
u \in {\mathscr H}_{\mathscr B} \mid
 (\alpha \Gamma_0 + \beta\Gamma_1)u = 0
\right\} = \Ker(\alpha \Gamma_0 + \beta\Gamma_1)
\]
and the resolvent (Kre\u{\i}n formula) holds:
\[
 (A_{\alpha, \beta} - zI)^{-1} = (A_0 -zI)^{-1} -
 (I -zA_0^{-1})^{-1}\Pi [\mathscr B(z)]^{-1} \beta \Pi^* (I -zA_0^{-1})^{-1}
\]
%
\item
Denote by~$A_{00}$ the restriction of~$A_0$ to the
set~$\Ker{\Gamma_1}$,
that is,
$A_{00} = \left. A\right|_{\Ker\Gamma_0\cap\Ker\Gamma_1}$.
Then $A_{00}$ is a symmetric operator with its domain not necessarily
dense in $H$ and
\[
 A_{00} \subset A_{\alpha, \beta} \subset A
\]
\end{itemize}
\end{theorem}
\noindent
Notice that~$A_0$ is a self-adjoint extension of~$A_{00}$ contained in~$A$.
It is not difficult to recognize
the parallel with the von Neumann
theory of self-adjoint extensions of
symmetric operators.
The operator~$A_{00}$ is the ``minimal'' operator with
the ``maximal'' equal
to~$A_{00}^{*}$ (whenever the latter exists)
and all self-adjoint extensions~$A_{s.a.}$ of~$A_{00}$
satisfy~$A_{00} \subset A_{s.a.} \subset A_{00}^*$.
Within the framework of Theorem~\ref{thm:BVP_def}
the equivalent of~$A_{00}^*$ is the operator~$A$
of the boundary value problem~(\ref{eqn:BVP_system})
defined on the domain~$\dom(A)$.
The semiboundness condition for~$A_{00}$ is relaxed and
replaced by the bounded invertibility of~$A_0$, i.~e.,
the existence of a regular point of~$A_0$ on the real line.

\subsubsection{Functional model}

The results of previous sections hint at the possibility
of a functional model construction for the
family of operators~$A_{\alpha, \beta}$
with a suitably chosen pair~$(\alpha, \beta)$.
Having in mind the model
space~(\ref{ModelSpace}),
the selection of a ``close'' dissipative
operator is typically guided
by the properties of the problem at hand.
In the most general case when
parameters~$(\alpha, \beta)$ are
unspecified,
a reasonable approach seems to be
to construct a model suitable
for the widest possible range
of~$(\alpha, \beta)$.
In accordance with the work by
S.~Naboko~\cite{MR573902},
the action of the operator~$A_{\alpha, \beta}$
will then be
explicitly described in the functional model representation.
This program is realized in the recent paper~\cite{CherKisSilva2}.
The ``model'' dissipative operator~$L = A_{-iI, I}$
corresponds
to the boundary
condition~$(\Gamma_1 - i\Gamma_0)u = 0$
for $u \in \dom(L)$ in the
notation~(\ref{eqn:BVP_system_family}).
The characteristic function of~$L$ then coincides
with the Cayley transform of the Weyl
function,
\[
 S(z) = (M(z) - {\rm i}I)(M(z) +{\rm i}I)^{-1} : E \to E,
 \quad z \in \mathbb C_+.
\]
Under the mapping of the upper half plane to the unit disk,
the function~$S((z-{\rm i})/(z+{\rm i}))$ is the Sz.Nagy-Foia\c{s}
characteristic function of a contraction,
namely, the Cayley transform~$V_0$ of~$A_{00}$
extended to~$H \ominus \dom(V_0)$ by the null
operator, hence resulting in a partial isometry.
If~$A_{00}$  has no non-trivial
self-adjoint parts,
the dissipative operator~$L = A_{-{\rm i}I, I}$ is
also completely non-selfadjoint.
The standard assumptions of complete
non-selfadjoint\-ness and maximality of
$L$ are thus met.
The minimal self-adjoint dilation~$\mathscr A$ of~$L$ formally
coincides with the dilation obtained in
Section~\ref{sect:BoundaryTriplesModel}
for the case of boundary triples with
$A_+^* = A$, $B_+ = {\rm i}I_E$, $B_+^* = - {\rm i}I_E$ and
$\alpha = \sqrt{2} I_E$.

The description of operators~$A_{\alpha, \beta}$ in the
spectral representation of dilation~$\mathscr A$,
i.e. in the model space~(\ref{ModelSpace}),
cannot be easily obtained for arbitrary~$(\alpha,\beta)$.
However, under certain conditions imposed on the parameters~$(\alpha, \beta)$
the model construction becomes tractable. Namely, one assumes that the operator $\beta$ is boundedly invertible,
the operator~$B = \beta^{-1}\alpha$ is
bounded in~$E$ and that the operator-valued function~$M(z)$ is invertible
and $B M(z)^{-1}$ is compact at least for
some~$z$ in the upper and lower
half-plane of
$\mathbb C$ (and therefore for all $z \in \mathbb C \setminus \mathbb R$).

It follows, that the operator~$A_{\alpha, \beta}$ has at most discrete spectrum
in~$\mathbb C \setminus \mathbb R$ with possible
accumulation to the real line only.
Moreover, the resolvent set of~$A_{\alpha, \beta}$ coincides with the open set of
complex numbers~$z\in \mathbb C$ such that the closed operator~$\overline{B +
M(z)}$ has a bounded inverse, i.~e. $\rho(A_{\alpha, \beta}) = \{z \in \mathbb
C \mid 0 \in \rho(\overline{B + M(z)})\}$.

Finally, the
model representation of the resolvent~$(A_{\alpha,\beta} - zI)^{-1}$, $z \in
\rho(A_{\alpha,\beta})$ is explicitly computed in the model space~(\ref{ModelSpace}).

Once the latter are established, it is natural to expect that the absolutely continuous subspaces can be characterised for the operators of boundary value problems in the case of exterior domains and the scattering theory can then be constructed following the recipe of Naboko, as presented in Section 2.2. If this programme is pursued, this would yield a natural representation of the corresponding scattering matrix purely in terms of the $M-$operator defined above. A paper devoted to this subject is presently being prepared for publication.

\subsection{Generalised resolvents}

\label{gen_res1}

In the present Section, we briefly recall the notion of generalised resolvents (see \cite{AG} for details) of symmetric operators, which play a major role in the asymptotic analysis of highly inhomogeneous media as presented in the present paper. It turns out that generalised resolvents and their underlying self-adjoint operators in larger (dilated) spaces feature prominently in our approach; moreover, their setup turns out to be natural in the theory of time-dispersive and frequency-dispersive media. On the mathematical level, this area is closely interrelated with the theory of dilations and functional models of dissipative operators, the latter (at least, in the case of dissipative extensions of symmetric operators) being an important particular case of the former.

We start with an operator-function $R(z)$ in the Hilbert space $H$, analytic in $z\in \mathbb C_+$. Assuming that $\im R(z)\geq 0$ for $z\in \mathbb{C}_+$, and under the well-known asymptotic condition $$\limsup_{\tau\to+\infty} \tau \|R({\rm i}\tau)\|<+\infty,$$ one has due to the operator generalisation of Herglotz theorem by Neumark \cite{Naimark1943}:
$$
R(z)=\int_{-\infty}^{\infty} \frac{1}{t-z} dB(t),
$$
where $B(t)$ is a uniquely defined left-continuous operator-function such that $B(-\infty)=0$, $B(t_2)-B(t_1)\geq 0$ for $t_2>t_1$ and $B(+\infty)$ bounded. By the argument of \cite{Figotin_Schenker_2005}, it follows from the Neumark theorem \cite{Naimark1940,Naimark1943} (cf., e.g., \cite{MR1036844}) that there exists a bounded operator $X: \mathcal H \mapsto H$ with an auxiliary Hilbert space $\mathcal H$ and a self-adjoint operator $\mathcal A$ in $\mathcal H$ such that
$$
R(z)=X(\mathcal A -z)^{-1} X^*
$$
with
$$
X X^* =\text{\rm s-}\!\lim_{t\to +\infty} B(t)=\text{\rm s-}\limsup_{\tau\to +\infty} \tau \im R(i\tau).
$$

A particular case of this result, see \cite{Strauss}, holds when (a) for some $z_0\in \mathbb{C}_+$ there exists a subspace $\mathfrak L\subset H$ such that (i) for all non-real $z$ and all $f\in \mathfrak{L}$ one has
$$
R(z)f -R(z_0)f= (z-z_0)R(z)R(z_0)f,
$$
(ii) for any $z\in \mathbb{C}_+$ and any $g\in \mathfrak{L}^\bot$ one has
$$
\|R(z)g\|^2\leq \frac{1}{\im z} \im \langle R(z)g,g \rangle,
$$
(iii)  for all $g\in \mathfrak{L}^\bot$ the function $R(z)g$ is regular in $\mathbb{C}_+$;  (b)  $\overline{R(z_0) \mathfrak{L}}= H$.

Under these assumptions, the function $R(z)$ is ascertained to be a generalised resolvent of a densely defined symmetric operator $A$ in $H$. Moreover, the deficiency index of $A$ in $\mathbb{C}_+$ is equal to $\dim \mathfrak{L}^\bot$. Precisely, this means that $\mathcal H \supset H$ and $X=P$, where $P$ is the orthogonal projection of $\mathcal H$ to $H$, i.e.,
\begin{equation}\label{eq:genres}
R(z)=P(\mathcal A -z)^{-1}|_H,\quad \im z\not=0,
\end{equation}
where $\mathcal A$ is a self-adjoint out-of-space extension of the symmetric operator $A$ (or, alternatively, a zero-range model with an internal structure, see Section 4 below). Moreover, under the minimality condition $\bigvee_{\im z\not=0}  (\mathcal A-z)^{-1}H=\mathcal H$, it is defined uniquely up to a unitary transform which acts as unity on $H$, see \cite{Naimark1940}.

The latter representation takes precisely the same form as the dilation condition \eqref{eqn:DilationEq} in the case of maximal dissipative extensions of symmetric operators, with the generalised resolvent $R(z)$ replacing the resolvent of a dissipative operator. It is in fact shown that the property \eqref{eq:genres} generalises \eqref{eqn:DilationEq}.

Namely, it turns out \cite{Strauss,Strauss_ext} that
$$
R(z)=(A_{B(z)}-z)^{-1} \text{ for } z\in \mathbb{C}_+\cup \mathbb{C}_-.
$$
Here in the particular case of equal deficiency indices, which is of interest to us from the point of view of zero-range models with an internal structure, $A_{B(z)}$ is a $z-$dependant extension of $A$ such that there exists a boundary triple $(\mathscr K, \Gamma_0, \Gamma_1)$ defining this extension as follows:
$$
\dom A_{B(z)}=\{u\in \dom A^* | \Gamma_1 u= B(z)\Gamma_0 u\}
$$
with $B(z)$ being a $-R$ function (i.e., an analytic operator-function with a non-positive imaginary part in $\mathbb{C}_+$).

Because of $B^*(\bar z)=B(z)$, which is the standard extension of an $R-$function into $\mathbb{C}_-$  implied here, the extension $A_{B(z)}$ turns out to be dissipative for $z\in \mathbb{C}_-$ and anti-dissipative for $z\in \mathbb{C}_+$.  We henceforth refer to $\mathcal A$ as the Neumark-Strauss dilation of the generalised resolvent $R(z)=(A_{B(z)}-z)^{-1}$. In a particular case of constant $B(z)=B$ such that $\im B\leq 0$, we have
$$
(A_{B}-z)^{-1}=P(\mathcal A -z)^{-1}|_H \text{ for } z\in \mathbb{C}_-
\text{ and }
(A_{B^*}-z)^{-1}=P(\mathcal A -z)^{-1}|_H \text{ for } z\in \mathbb{C}_+,
$$
which are precisely \eqref{eqn:DilationEq} for both $A_B$ and $A_{B^*}$ at the same time.

From what has been said above it follows that generalised resolvents appear when one conceals certain degrees of freedom in a conservative physical system, either for the sake of convenience or because these are not known. In particular, we refer the reader to the papers \cite{Figotin_Schenker_2005,Figotin_Schenker_2007b}, where systems with time dispersion are analysed, with prescribed ``memory'' term. It turns out that passing over to the frequency domain one ends up with a generalised resolvent. It then proves possible to explicitly restore a conservative Hamiltonian (the operator $\mathcal A$ in our notation) which yields precisely the postulated time dispersion. In a nutshell, the idea here is to work with an explicit and simple enough model of the part of the space pertaining to the ``hidden degrees of freedom'' instead of the unnecessarily complicated physical equations which govern them. Similar ideas have been utilised in \cite{Tip_1998,Tip_2006}. The same technique has found its applications in numerics, and in particular in the so-called theory of absorbing boundary conditions, see, e.g., \cite{Absorbing1,Absorbing2}.

The problem of constructing a spectral representation for a Neumark-Strauss dilation of a given generalised resolvent thus naturally arises. In a number of special cases, where $\mathcal H$ and $\mathcal A$ admit an explicit construction (and in particular one has $\mathcal H  = H\oplus \mathbb{C}^k$ for $k\geq 1$), this can be done following essentially the same path as outlined in Section \ref{sec:ryzh-funct-model} above. This is due to the fact that in this case the operator $\mathcal A$ can be realised as a von Neumann extension of a symmetric operator in $\mathcal H$ with equal and finite deficiency indices. The corresponding construction in the case when $k=1$ is presented in \cite{CENS}. Surprisingly, this rather simple model already has a number of topical applications, see Sections 4 and 5 of the present paper, and also the papers \cite{Exner,KuchmentZeng,KuchmentZeng2004}, where a generalised resolvent of precisely the same class appears in the setting of thin networks converging to quantum graphs.

The generic case has been studied by Strauss in \cite{Strauss_survey}, where three spectral representations of the dilation are constructed, analogous to the ones of L.~de~Branges and J.~Rovnyak, B.~S.~Pavlov, and B.~Sz.-Nagy and C.~Foia\c{s}. These results however present but theoretical interest, as they are formulated in terms which apparently cannot be related to the original problem setup and are therefore not usable in applications.

A different approach was suggested by M.~D.~Faddeev and B.~S.~Pavlov in \cite{PavlovFaddeev}, where a problem originally studied by P.~D.~Lax and
R.~S.~Phillips in \cite{Lax1973} was considered. In \cite{PavlovFaddeev}, a five-component representation of the dilation was constructed, which further allowed to obtain the scattering matrix in an explicit form. It therefore comes as no surprise that, precisely as in the Lax-Phillips approach, the resonances are revealed to play a fundamental role in this analysis (cf. the analysis of the so-called Regge poles in the physics literature).

Later on, and again motivated in particular by applications to scattering, Neumark-Strauss dilations were constructed in some special cases by J.~Behrndt et al., see  \cite{Behrndt2007,OQS_Malamud}.

We remark that all the above results, except \cite{CENS} and \cite{Strauss_survey}, have stopped short of attempting to construct a spectral form of the Neumark-Strauss dilation. Any generic construction leading to the latter and formulated in ``natural'' terms is presently unknown, to the best of our knowledge.

\subsection{Universality of the model construction}

The general form of the
functional model of an unbounded
closed
operator~\cite{Ryzhov_closed}
is a generalization
of the special cases, as developed
in the papers by B.~Pavlov
and S.~Naboko cited above.
%
%
This section aims to clarify the relationship
between the models pertaining to
different representations
of the characteristic function of a
non-selfadjoint operator.
As an illustration, we consider
two special cases of
operators of mathematical physics
described above and link them to
the general model
construction of~\cite{Ryzhov_closed}.

\subsubsection{Characteristic function of a linear operator~\cite{Strauss1960}}

Let~$L$ be a closed linear operator on a (separable)
Hilbert space~$H$ with the
domain~$\dom(L)$.
Consider the form~$\Psi_L(\cdot, \cdot)$ defined
on~$\dom(L)\times \dom(L)$:
\begin{equation}\label{Psi_L}
 \Psi_L(f,g) = \frac{1}{\rm i}\left[ (Lf,g)_H  - (f, Lg)_H\right],
 \quad f,g \in \dom(L)
\end{equation}
\begin{definition}\label{defBoundarySpaceAndOperator}
 The {\it boundary space} of~$L$ is a
 linear
space~$E$ with a possibly indefinite
scalar product~$(\cdot, \cdot)_E$ such that
there exists a closed
linear operator~$\Gamma$ defined
on~$\dom(\Gamma) = \dom(L)$ and the
following identity holds
\begin{equation}\label{BoundaryOperator}
\Psi_L(f,g) = (\Gamma f, \Gamma g)_E,
\quad f,g \in\dom(L)
\end{equation}
The operator~$\Gamma$ is called the
{\it boundary operator} of~$L$.
\end{definition}
This definition is meaningful for any
linear operator on~$H$.
For the purposes of model construction,
it is sufficient to focus only on the case
when $L$ is densely defined and dissipative.
The model representation of a non-dissipative
operator is given in the model space of an
auxiliary dissipative one, as explained above.
When~$L$ is dissipative, one has
$\Psi_L(f,f) \geq 0$, $f\in\dom(L)$  and
therefore the space~$E$ can be chosen
as the Hilbert space obtained by
factorization and completion
of~$\{\Gamma f \mid f\in \dom(L)\}$
with respect to the norm~$\|\Gamma f\|_E$,
$f\in \dom(L)$.
Note that the  boundary
operator~$\Gamma$ defined
in~(\ref{BoundaryOperator})
is not uniquely defined.
Due to the Hilbert structure of~$E$,
for any isometry~$\pi$ on $E$
the operator~$\pi\Gamma$
also satisfies the
condition~(\ref{BoundaryOperator}).
Moreover, if (\ref{BoundaryOperator})
holds for some
operator~$\Gamma^\prime$
and space~$E^\prime$,
then there exists an
isometry~$\pi : E^\prime \to E$
such that
$\Gamma = \pi \Gamma^\prime$.

Denote by~$E_*$ and $\Gamma_*$
the boundary space
and the boundary operator for
the dissipative operator~$-L^*$ endowed
with the Hilbert metric.
Assume that~$L$ is maximal, i.~e.,
$\mathbb C_- \subset \rho(L)$.
Then $L^*$ is also maximal
and~$\mathbb C_+ \subset \rho(L^*)$.
The following definition is valid for general
non-selfadjiont operators.
\begin{definition}
Let~$E$ and  $\Gamma$
be the boundary space and the boundary
operator for a closed
densely defined operator~$L$. Let~$E_*$
and  $\Gamma_*$  be the boundary space
and the boundary operator
for the operator~$-L^*$.
The characteristic function of the operator~$L$
 is the analytic operator-valued
 function~$S(z): E \to E_*$
defined by
\[
  S(z) \Gamma f  =
  \Gamma_* (L^* - zI)^{-1}(L - zI)f,
  \quad f \in \dom(L),\quad z \in \rho(L^*)
\]
If $L$ is dissipative, then the spaces~$E$ and $E_*$
are Hilbert spaces, and
the operator~$S(z)$ is a contraction on $E$ for each $z\in\mathbb C_+$.
\end{definition}

Note that the actual form of~$S(z)$
depends on the choice
of  boundary spaces and boundary
operators.
If~$\Gamma^\prime : \dom(L) \to E^\prime$
and $\Gamma_*^\prime : \dom(L^*) \to E_*^\prime$
satisfy the condition~(\ref{BoundaryOperator})
and $S^\prime(z)$ is the
corresponding characteristic function, then
there exist two isometries~$\pi_* : E_* \to E_*^\prime$
and $\pi : E \to E^\prime$ such that
$\pi_* S(z) = S^\prime(z)\pi$, $z\in \rho(L^*)$.
Such characteristic functions of the operator~$L$
are often called {\it equivalent}.


All steps involved in the model
construction outlined above do not depend
on the concrete form
of the characteristic function~\cite{BMNW2018}.
In particular cases when the characteristic
function can be expressed in terms of
the original problem, the model
admits a ``natural'' form in relation to the
problem setup.
Examples of such calculations are
provided towards the end of this section.

In order to compute a
characteristic function of~$L$ one has to
come up with a suitable definition
of boundary spaces and operators.
Consider first the general case where
no specific assumptions on the
operator~$L$ are made, and
introduce the Cayley
transform of~$L$, i.e.,
$T = (L -{\rm i}I)(L+{\rm i}I)^{-1}$.
The operator~$T$ is clearly contractive.
The operators
\[
Q :=\frac{1}{\sqrt{2}}(I-T^*T)^{1/2},\qquad
Q_* :=\frac{1}{\sqrt{2}}(I-T T^*)^{1/2}
\]
are thus non-negative.
A straightforward calculation~\cite{Ryzhov_closed}
shows that
the boundary spaces~$E$, $E_*$ and
the boundary operators~$\Gamma$,
$\Gamma_*$  can be defined as follows:
\[
 E = \clos\ran(Q), \quad E_* = \clos\ran(Q_*),
 \quad \Gamma = \clos Q(L + {\rm i}I),
 \quad \Gamma_* = \clos Q_*(L^* - {\rm i}I).
\]
Here the operators $\Gamma$ and $\Gamma_*$
are the closures of the respective mappings
initially defined on~$\dom(L)$ and $\dom(L^*)$.
This choice leads to the following expression
for the characteristic function of
the operator~$L:$
\begin{equation}\label{abstractCharFunc}
 S(z) = \left.\bigl(T - (z-{\rm i})\Gamma_*(L^* - zI)^{-1}Q\bigr)\right|_E, \quad
 z \in \mathbb C_+.
\end{equation}
An explicit calculation
reveals that~$S(z)$ is closely related to
the characteristic function
of~$T$,
\[
S(z) = - \vartheta_T\left(\frac{z -{\rm i}}{z + {\rm i}}\right), \quad z \in \mathbb C_+,
\]
where
\[
\vartheta_T(\lambda) =\left.
\bigl(- T + 2 \lambda Q_*(I - \lambda T^*)^{-1}Q\bigr)\right|_E,
\qquad |\lambda| < 1
\]
is the Sz.-Nagy-Foia\c{s}
characteristic function of the contractive operator~$T$.
Therefore, the formula~(\ref{abstractCharFunc})
is the abstract form of the  characteristic function
of~$L$ regardless the ``concrete'' realization of
the operator~$L$.

\subsubsection{Examples}

The actual choice of boundary spaces
and operators is guided by the specifics
of the problem at hand.
Let us demonstrate the ``natural'' selection for these objects for some
of the models introduced in Section 2.

\paragraph{Additive perturbations}
This is the simplest (and canonical)
case of the characteristic
function calculations included here solely
for the completeness of exposition.
Let~$L$ be a dissipative operator of Section \ref{sec:add_perturbations},
defined as an additive perturbation
of a self-adjoint operator.
Then for~$f, g\in\dom(A) = \dom(L)$
one has
\[
 \Psi_L(f,g) = \frac{1}{\rm i}\left[(Lf ,g) - (f, Lg)\right] =
 \frac{1}{\rm i}\left[ \left({\rm i}\frac{\alpha^2}{2}f,g\right) -
 \left(f,{\rm i}\frac{\alpha^2}{2}g\right)\right] =
 (\alpha f, \alpha g)
\]
and therefore,  the boundary
space can be chosen as~$E = \clos\ran(\alpha)$
with
the boundary operator~$\Gamma$
defined as a
mapping~$\Gamma f \mapsto \alpha f$. In a similar way, $E_*=E$ and  $\Gamma_*=\Gamma.$
%
%
%

The characteristic function of~$L$ corresponding
to this selection of boundary spaces
and operators is then computed as \eqref{DissCharF}.
As explained above, this characteristic function is equivalent
to the function~(\ref{abstractCharFunc}).

\paragraph{Almost solvable extensions}

In the notation of Section~\ref{sec:ryzh-funct-model},
let~$L = A_B$ be a dissipative almost solvable
extension of a symmetric operator~$A$
corresponding to the bounded
operator~$B = B_R + i B_I$
with $B_R = B_R^*$ and $B_I = B_I^* \geq 0 $
defined on the space~$\mathcal K$.
Denote~$\alpha := \sqrt{2}(B_I)^{1/2}$.
From the Green's formula~(\ref{GreenFormula})
and the condition~$\Gamma_1 f= B \Gamma_0 f$, $f\in\dom(L)$
we obtain for~$f,g \in \dom(L)$:
\[
\Psi_L(f,g) = \frac{1}{\rm i}
\left[ (Lf,g) - (f,Lg) \right]=  \frac{1}{\rm i}
\bigl((B - B^*)\Gamma_0 f, \Gamma_0 g\bigr)_\mathcal K =
(\alpha\Gamma_0 f, \alpha\Gamma_0 g)_{\mathcal K}.
\]
Next we demonstrate two alternative approaches to
the derivation of the characteristic function of~$L$.

\paragraph{Approach 1.}
Define the boundary space~$\mathcal E$ of
the operator~$L = A_B$ as
the factorization and completion of
the linear set~$\mathcal L = \{ \Gamma_0 f, \mid f \in \dom (L)\}$
endowed with the norm~$\|u\|_{\mathcal E} =
\|\alpha^2 u\|_\mathcal K$,
$u \in \mathcal L$.
The norm~$\|\cdot\|_{\mathcal E}$ is degenerate
if~$\ker(\alpha)$ is non-trivial, thus the factorization
becomes necessary.
The corresponding boundary operator~$\Gamma$ is the
map~$\Gamma : f \mapsto \Gamma_0 f$
on the domain~$\dom(\Gamma) = \dom(L)$.
In a similar way, $\mathscr E_*$ is defined
as the factorization and completion of the linear
set ~$\mathcal L_*= \{ \Gamma_0 g, \mid g \in \dom (L^*)\}$
with respect to the norm~$\|u\|_{\mathcal E_*} =
\|\alpha^2  u\|_\mathcal K$, $u \in \mathcal L_*$.
The boundary operator~$\Gamma_*$ is the
mapping~$\Gamma_* : g \mapsto \Gamma_0 g$
defined on~$\dom(\Gamma_*) = \dom(L^*)$.
Thus, both boundary spaces~$\mathcal E$
and $\mathcal E_*$ are Hilbert spaces
with the norm associated with the ``weight''
equal to~$\alpha^2$.

An explicit computation then yields the following expression for the characteristic function $\mathscr S$:
\[
 \mathscr S(z) =\bigl(B^* - M(z)\bigr)^{-1}\bigl(B - M(z)\bigr),
 \quad z \in\rho(L^*),
\]
where $M(z)$ is the Weyl-Titchmarsh $M-$function of Section~\ref{sec:ryzh-funct-model}.

\paragraph{Approach 2.}
An alternative form of the characteristic function
is obtained based on the
boundary operators~$\Gamma$ and $\Gamma_*$ introduced
as the closures of
the mappings~$f \mapsto \alpha\Gamma_0 f$
and $g\mapsto \alpha\Gamma_0 g$
defined on the linear sets~$ \dom(L)$ and
$\dom(L^*)$,
respectively.
The boundary spaces~$E$ and $E_*$
in this case are chosen
as
\[
E = \clos\ran\bigl(\alpha\Gamma_0|_{\dom(L)}\bigr),\qquad
E_* = \clos\ran\bigl(\alpha\Gamma_0|_{\dom(L^*)}\bigr).
\]
In all applications considered in this paper, these
spaces coincide:~$E= E_*$.
Similarly to the situation of additive perturbations,
it is often convenient (and common) to
extend these spaces to~$\clos\ran(\alpha)$.

The
corresponding characteristic function is then represented
by the formula~(\ref{eq:charf_ext})
repeated here for the  sake of readers' convenience:
\[
S(z) = I_E + {\rm i}\alpha\bigl(B^* - M(z)\bigr)^{-1}\alpha : E \to E,
\quad z \in \rho(L^*).
\]
In contrast to Approach 1, this form captures the specifics
of the extension parameter~$B$.
In particular, the dimension of the space~$E$  equals
the dimension of the range of~$\alpha$.
If the operator~$B$ is a compact perturbation
of a self-adjoint, i.~e., $B_I = \alpha^2/2 \in \mathfrak S_\infty$, then
the characteristic function~$S(z)$ is an operator-valued
function of the form $I + \mathfrak{S_\infty}$ defined
on the
(unweighted) Hilbert space~$\clos\ran(\alpha)$.

\paragraph{Equivalence of $\mathscr S$ and $S$.}
The mapping~$\widehat\alpha: f \mapsto \alpha f$,
$f \in \mathcal E$
is an isometry from the ``weighted'' space~$\mathcal E$
to the space~$\mathcal K$.
The equality~$\widehat\alpha \mathscr S(z) = S(z)\widehat\alpha$,
$z\in\rho(L^*)$
expresses equivalence of the characteristic
functions~$\mathscr S$ and $S$ corresponding to
different choices of boundary spaces and operators.
Both $\mathscr S$ and $S$ functions are equivalent to the characteristic
function of~$L$ written in its abstract form~(\ref{abstractCharFunc}).
It is easy to see that the
boundary operators of Approach 1 and Approach 2
are also related by means of the isometric
mapping~$\widehat\alpha$.

In conclusion, we point out the recent
paper~\cite{BMNW2018}
where the
construction of the selfadjoint dilation and the
functional model
of a dissipative operator
is based entirely on the concept of
Strauss boundary
spaces and operators~(\ref{Psi_L}),
(\ref{BoundaryOperator}) with no reference to their
``concrete'' realizations.

\section{An application: inverse scattering problem for quantum graphs}
\label{sec:theor-appl-funct}

In the present section, we present an application of the theory introduced in Section \ref{sec:ryzh-funct-model}, and in particular of the explicit construction of wave operators and scattering matrices facilitated by the approach based on the functional model due to Sergey Naboko. We first obtain an explicit expression for the scattering matrix of a quantum graph which we take to be the Laplacian on a finite non-compact metric graph, subject to $\delta-$type coupling at graph vertices. Then we present an explicit constructive solution to the inverse scattering problem for this graph, i.e., explicit formulae for the coupling constants at graph vertices. The narrative of this Section is based upon the papers \cite{CherednichenkoKiselevSilva,CherKisSilva1}.

For simplicity of presentation we will only consider the case of a finite non-compact quantum graph, when the deficiency indices are finite. However, the same approach allows us to consider the general setting of infinite deficiency indices, which in  the quantum graph setting leads to an infinite graph. In particular, one could consider the case of an infinite compact part of the graph.

In what follows, we denote by ${\mathbb G}={\mathbb G}({\mathcal E},\sigma)$ a
finite metric graph, {\it i.e.} a collection of a finite non-empty set
${\mathcal E}$ of compact or semi-infinite intervals
$e_j=[x_{2j-1},x_{2j}]$ (for semi-infinite intervals we set
$x_{2j}=+\infty$), $j=1,2,\ldots, n,$ which we refer to as
\emph{edges}, and of a partition $\sigma$ of the set of endpoints
${\mathcal V}:=\{x_k: 1\le k\le 2n,\ x_k<+\infty\}$ into $N$ equivalence classes $V_m,$
$m=1,2,\ldots,N,$ which we call \emph{vertices}:
${\mathcal V}=\bigcup^N_{m=1} V_m.$ The degree, or valence, ${\rm deg}(V_m)$ of the vertex $V_m$ is defined as the number of elements in $V_m,$ {\it i.e.} ${\rm card}(V_m).$
Further, we partition the set ${\mathcal V}$ into the two non-overlapping
sets of \emph{internal} ${\mathcal V}^{({\rm i})}$ and \emph{external}
${\mathcal V}^{({\rm e})}$ vertices, where a vertex $V$ is classed as internal
if it is incident to no non-compact edge and external otherwise. Similarly, we partition the set of edges
${\mathcal E}={\mathcal E}^{({\rm i})}\cup {\mathcal E}^{({\rm e})}$, into the collection of
compact (${\mathcal E}^{({\rm i})}$) and non-compact (${\mathcal E}^{({\rm e})}$)
edges. We assume for simplicity that the number of non-compact edges
incident to any graph vertex is not greater than one.

For a finite metric graph ${\mathbb G},$ we consider the Hilbert spaces
\[
L^2({\mathbb G}):=\bigoplus_{j=1}^n L^2(e_j),\qquad W^{2,2}({\mathbb G}):=\bigoplus_{j=1}^n W^{2,2}(e_j).
\]
Further, for a function $f\in W^{2,2}({\mathbb G})$, we
define
the normal derivative at each vertex along each of the adjacent edges,
as follows:
\begin{equation*}
\partial_n f(x_j):=\left\{ \begin{array}{ll} f'(x_j),&\mbox{ if } x_j \mbox{ is the left endpoint of the edge},\\[0.35em]
-f'(x_j),&\mbox{ if } x_j \mbox{ is the right endpoint of the
edge.}
\end{array}\right.
\end{equation*}
In the case of semi-infinite edges we only apply this
definition at the left endpoint of the edge.

\begin{definition}
For $f\in W^{2,2}({\mathbb G})$
and ${a_m}\in{\mathbb C}$ (below referred to as the ``coupling
constant"), the condition of continuity of the function $f$ through
the vertex $V_m$ ({\it i.e.} $f(x_j)=f(x_k)$ if $x_j,x_k\in V_m$)
together with the condition
\[
\sum_{x_j \in V_m} \partial _n f(x_j)={a_m} f(V_m)
\]
is called the $\delta$-type matching at the vertex $V_m$.
\end{definition}

\begin{remark}
  Note that the $\delta$-type matching condition in a particular case
  when ${a_m}=0$ reduces to the standard Kirchhoff matching
  condition at the vertex $V_m$, see {\it e.g.} \cite{MR3013208}.
\end{remark}

\begin{definition}
\label{def6}
The quantum graph Laplacian $A_a,$ $a:=(a_1,...,a_N),$ on a graph ${\mathbb G}$ with
$\delta$-type
 matching conditions is the operator of minus second derivative $-d^2/dx^2$
 in the Hilbert space $L^2({\mathbb G})$ on the domain of functions that
 belong to the Sobolev space $W^{2,2}({\mathbb G})$ and satisfy the
 $\delta$-type matching conditions at every vertex $V_m$,
 $m=1,2,\dots,N.$ The Schr\"odinger operator on the same graph is
 defined likewise on the same domain in the case of summable edge
 potentials ({\it cf.} \cite{MR3484377}).
\end{definition}

If all coupling constants ${a_m}$, $m=1,\dots, N$, are real, it is
shown that the operator $A_a$ is a proper self-adjoint
extension of a closed
symmetric operator $A$ in $L^2({\mathbb G})$
\cite{MR1459512}. Note that, without loss of generality,
each edge $e_j$ of the graph ${\mathbb G}$ can be considered to be an
interval $[0,l_j]$, where $l_j:=x_{2j}-x_{2j-1}$, $j=1,\dots, n$ is
the length of the corresponding edge. Throughout the present Section we
will therefore only consider this situation.

In \cite{MR3484377}, the following result is obtained for the case of
finite \emph{compact} metric graphs.

\begin{proposition}[\cite{MR3484377}]\label{Prop_M}
  Let ${\mathbb G}$ be a finite compact metric graph with $\delta$-type
  coupling at all vertices. There exists a closed densely defined
  symmetric operator $A$ and a boundary triple such that the
  operator $A_a$ is an almost solvable extension of $A$,
  for which the parametrising matrix $\varkappa$
  is given by
  $\varkappa=\mathrm{diag}\{a_1,\dots,a_N\}$, whereas the Weyl
  function is an $N\times N$ matrix with elements
\begin{equation}\label{Eq_Weyl_Func_Delta}
m_{jk}(z)=
\begin{cases}
  -\sqrt{z}\biggl(\sum\limits_{e_p\in E_k}\cot\sqrt{z} l_p- 2\sum\limits_{e_p\in L_k}\tan\dfrac{\sqrt{z}
    l_p}{2}\biggr),
  & j=k,\\[0.4em]
  \sqrt{z}\sum\limits_{e_p\in C_{jk}}\dfrac{1}{\sin\sqrt{z} l_p},
  & j\not=k;\  V_j, V_k\ \mbox{adjacent},\\[0.5em]
  0, & j\not=k;\ V_j, V_k\ \mbox{non-adjacent}.\\
     \end{cases}
\end{equation}
Here the branch of the square root is chosen so that $\Im\sqrt{z}\geq 0,$
$l_p$ is the length of the edge $e_p$,  $E_k$ is the set of non-loop graph edges
incident to the vertex $V_k$, $L_k$ is the set of loops at the
vertex $V_k,$ and $C_{jk}$ is the set of graph edges
connecting vertices $V_j$ and $V_k.$
\end{proposition}

In \cite{CherednichenkoKiselevSilva} this is extended to non-compact metric graphs as follows.
Denote by ${\mathbb G}^{({\rm i})}$ the compact part of the graph
${\mathbb G}$, {\it i.e.} the graph ${\mathbb G}$ with all the non-compact edges
removed. Proposition \ref{Prop_M} yields an expression for the Weyl
function $M^{({\rm i})}$ pertaining to the graph ${\mathbb G}^{({\rm i})}$.

\begin{lemma}
The matrix functions $M,$ $M^{({\rm i})}$ described above are related by the formula
\begin{equation}
M(z)=M^{({\rm i})}(z) + {\rm i}\sqrt{z} P_{\rm e},\quad\quad z\in{\mathbb C}_+,
\label{M_Mi}
\end{equation}
where $P_{\rm e}$ is the orthogonal projection in the boundary space
$\mathcal K$ onto the set of external vertices $V_{\mathbb G}^{({\rm e})}$, {\it
  i.e.} the matrix $P_{\rm e}$ such that $(P_{\rm e})_{ij}=1$ if $i=j,$ $V_i\in V_{\mathbb G}^{({\rm e})},$ and
$(P_{\rm e})_{ij}=0$ otherwise.
\end{lemma}

The formula  (\ref{M_Mi}) leads to
$M(s)-M^*(s)=2{\rm i}\sqrt{s}P_{\rm e}$ a.e. $s\in{\mathbb R}$, and the
expression (\ref{scat2}) for $\widehat{\Sigma}$ leads to the classical
scattering matrix $\widehat{\Sigma}_{\rm e}(k)$ of the pair of operators $A_0$ (which is the Laplacian on the graph ${\mathbb G}$ with standard Kirchhoff matching at all the vertices) and $A_\varkappa,$ where $\varkappa=\varkappa=\mathrm{diag}\{a_1,\dots,a_N\}:$
\begin{equation}
\widehat{\Sigma}_{\rm e}(s)=
P_{\rm e} (M(s)-\varkappa)^{-1}(M(s)^*-
\varkappa)(M(s)^*)^{-1}M(s) P_{\rm e},\ \ \ \ s\in{\mathbb R},
\label{sigma_hat}
\end{equation}
which acts as the operator of multiplication in the space
$L^2(P_{\rm e}{\mathcal K}; 4\sqrt{s}ds)$.

We remark that in the more common approach to the construction of scattering matrices, based on comparing the asymptotic expansions of solutions to spectral equations, see {\it e.g.} \cite{{Faddeev_additional}}, one obtains
$\widehat{\Sigma}_{\rm e}$ as the scattering matrix. Our approach yields an explicit factorisation of
$\widehat{\Sigma}_{\rm e}$
into expressions involving the matrices $M$ and $\varkappa$ only, sandwiched between two projections. (Recall that  $M$ and $\varkappa$ contain the information about the geometry of the graph and the coupling constants, respectively.) From the same formula (\ref{sigma_hat}), it is obvious that without the factorisation the pieces of information pertaining to the geometry of the graph and the coupling constants at the vertices are present in the final answer in an entangled form.

We reiterate that the analysis above pertains not only to the cases when the coupling constants are real, leading to self-adjoint operators $A_a,$ but also to the case of non-selfadjoint extensions, {\it cf.} Theorem \ref{thm:existence-completeness-wave-operators}.

In what follows we often drop the argument $s\in{\mathbb R}$ of the
Weyl function $M$ and the scattering matrices $\widehat{\Sigma},$
$\widehat{\Sigma}_{\rm e}.$

It is easily seen that
a factorisation of $\widehat{\Sigma}_{\rm e}$ into a product of
$\varkappa-$dependent and $\varkappa-$independent factors ({\it cf.}
(\ref{scat2})) still holds in this case in $P_{\rm e} {\mathcal K},$ namely
\begin{equation}
\widehat{\Sigma}_{\rm e}=
\bigl[P_{\rm e} (M-\varkappa)^{-1}(M^*-\varkappa)P_{\rm e}\bigr]\bigl[P_{\rm e}(M^*)^{-1}M P_{\rm e}\bigr].
\label{bigstar}
\end{equation}

 We will now exploit the above approach in the analysis of the inverse
scattering problem for Laplace operators on finite metric graphs,
whereby the scattering matrix $\widehat{\Sigma}_{\rm e}(s),$ defined by (\ref{bigstar}), is assumed
to be known for almost all positive ``energies'' $s\in{\mathbb R},$
along with the graph ${\mathbb G}$ itself. The data to be determined is the
set of coupling constants $\{{a_j}\}_{j=1}^N$. For simplicity, in what follows we treat the inverse problem for graphs with
real coupling constants, which corresponds to self-adjoint operators.

First, for given $M,$ $\widehat{\Sigma}_{\rm e}$ we reconstruct the expression $P_{\rm e}(M^{({\rm i})}-\varkappa)^{-1}P_{\rm e}$ for almost all $s>0:$
\begin{equation}
P_{\rm e}(M^{({\rm i})}-\varkappa)^{-1}P_{\rm e}
=\frac{1}{{\rm i}\sqrt{s}}\biggl(2\bigl(P_{\rm e}+\widehat{\Sigma}_{\rm e}[P_{\rm e}(M^*)^{-1}MP_{\rm e}]^{-1}\bigr)^{-1}-I\biggr)P_{\rm e}.
\label{DtD}
\end{equation}
In particular, due to the property of analytic continuation, the expression $P_{\rm e}(M^{({\rm i})}-\varkappa)^{-1}P_{\rm e}$ is determined uniquely in the whole of $\mathbb C$ with the
exception of a countable set of poles, which coincides with the
set of eigenvalues of the self-adjoint Laplacian $A_\varkappa^{({\rm i})}$ on the
compact part ${\mathbb G}^{({\rm i})}$ of the graph ${\mathbb G}$ with matching
conditions at the graph vertices given by the matrix $\varkappa,$ {\it cf.} Proposition \ref{Prop_M}.

\begin{definition}
\label{definition7}
Given a partition ${\mathcal V}_1\cup{\mathcal V}_2$ of the set of graph vertices, for $z\in{\mathbb C}$ consider the linear set $U(z)$ of functions that satisfy the differential equation $-u_z''=zu_z$ on each edge, subject to the conditions of continuity at all vertices of the graph and the $\delta$-type matching conditions at the vertices in the set ${\mathcal V}_2.$ For each function $f\in U(z),$ consider the vectors
\[
\Gamma_1^{{\mathcal V}_1}u_z:=\Bigl\{\sum_{x_j \in V_m} \partial _n f(x_j)\Bigr\}_{V_m\in{\mathcal V}_1},\quad\quad\Gamma_0^{{\mathcal V}_1}u_z:=\bigl\{f(V_m)\bigr\}_{V_m\in{\mathcal V}_1}.
\]
The {\it Robin-to-Dirichlet map} of the set ${\mathcal V}_1$ maps the vector  $(\Gamma_1^{{\mathcal V}_1}-\varkappa^{{\mathcal V}_1}\Gamma_0^{{\mathcal V}_1})u_z$ to $\Gamma_0^{{\mathcal V}_1}u_z,$ where $\varkappa^{{\mathcal V}_1}:=\diag \{a_m:\ V_m\in{\mathcal V}_1\}$. (Note that the function $u_z\in U(z)$ is determined uniquely by $(\Gamma_1^{{\mathcal V}_1}-\varkappa^{{\mathcal V}_1}\Gamma_0^{{\mathcal V}_1})u_z$ for all $z\in\mathbb C$ except a countable set of real points accumulating to infinity).
\end{definition}

The above definition is a natural generalisation of the corresponding definitions of Dirichlet-to-Neumann and Neumann-to-Dirichlet maps pertaining to the graph boundary, considered
  in {\it e.g.} \cite{MR3013208}, \cite{MR2600145}.

We argue that the matrix $P_{\rm e}(M^{({\rm i})}-\varkappa)^{-1}P_{\rm e}$ is the Robin-to-Dirichlet map
for the set ${\mathcal V}^{({\rm e})}.$ Indeed, assuming $\phi:=\Gamma_1 u_z-\varkappa \Gamma_0 u_z$ and $\phi=P_{\rm e} \phi,$ where the latter condition ensures the correct $\delta$-type matching on the set $\mathcal{V}^{({\rm i})},$ one has
$P_{\rm e}\phi=(M^{({\rm i})}-\varkappa)\Gamma_0 u_z$ and hence $\Gamma_0 u_z=(M^{({\rm i})}-\varkappa)^{-1}P_{\rm e} \phi$. Applying $P_{\rm e}$ to the last identity yields the claim, in accordance with Definition \ref{definition7}.


We thus have the following theorem.
\begin{theorem}
\label{thm9.1}

  The Robin-to-Dirichlet map for the vertices ${\mathcal V}^{({\rm e})}$
  is determined uniquely by the scattering matrix
  $\widehat{\Sigma}_{\rm e}(s),$ $s\in{\mathbb R},$ via the formula (\ref{DtD}).

\end{theorem}

The following definition, required for the formulation of the next theorem, is a generalisation of the procedure of graph contraction, well studied in the algebraic graph theory, see {\it e.g.} \cite{Tutte}.

\begin{definition}[Contraction procedure for graphs and associated quantum graph Laplacians]
\label{def8}
For a given graph ${\mathbb G}$ vertices $V$ and $W$ connected  by an edge $e$ are ``glued'' together to form a
new vertex $(VW)$
of the contracted graph $\widetilde{\mathbb G}$ while simultaneously the edge $e$ is removed,  whereas the rest of the graph remains
unchanged. We do allow the situation of multiple edges, when $V$ and
$W$ are connected in ${\mathbb G}$ by more than one edge, in which case all such edges but the edge $e$ become loops of their
respective lengths attached to the vertex $(VW)$. The corresponding quantum graph Laplacian $A_a$ defined on ${\mathbb G}$ is contracted to the
quantum graph Laplacian $\widetilde{A}_{\widetilde{a}}$ by the application of the following rule pertaining to the coupling constants:
a coupling constant
at any unaffected vertex remains the same, whereas the coupling
constant at the new vertex $(VW)$ is set to be the sum of the coupling constants at $V$ and $W.$
Here it is always assumed that all quantum graph Laplacians are described by Definition \ref{def6}.
\end{definition}

\begin{theorem}
\label{NtD_NtD}
Suppose that the edge lengths of the graph ${\mathbb G}^{({\rm i})}$ are rationally independent.
The element\footnote{By renumbering if necessary, this does not lead to loss of generality.} $(1,1)$  of the Robin-to-Dirichlet map described above yields the element $(1,1)$ of the ``contracted'' graph
$\widetilde{\mathbb G}^{({\rm i})}$ obtained from the graph ${\mathbb G}^{({\rm i})}$ by removing a non-loop edge $e$
emanating from $V_1.$ The procedure of passing from the graph ${\mathbb G}^{({\rm i})}$ to the contracted graph
$\widetilde{\mathbb G}^{({\rm i})}$ is given in Definition \ref{def8}.

\end{theorem}

\begin{proof}

Due to the assumption
that the edge lengths of the graph ${\mathbb G}^{({\rm i})}$ are rationally
independent, the element (1,1), which we denote by $f_1,$ is expressed explicitly as a function of
$\sqrt{z}$ \emph{and} all the edge lengths $l_j,$ $j=1,2,\dots, n,$ in particular, of the length of the edge $e,$ which we assume  to be $l_1$ without loss of generality. This is
an immediate consequence of the explicit form of the matrix
$M^{({\rm i})},$ see (\ref{Eq_Weyl_Func_Delta}). Again without loss of generality, we also assume that the edge $e$ connects the vertices $V_1$ and $V_2.$

Further, consider the expression $\lim_{l_1\to 0}
f_1(\sqrt{z}; l_1,\dots,l_n; a)$. On the one hand, this limit is known from the explicit expression for $f_1$ mentioned above. On the
other hand, $f_1$ is the ratio of the determinant ${\mathcal D}^{(1)}(\sqrt{z}; l_1, \dots, l_n; a)$ of the principal minor of the matrix $M^{({\rm i})}(z)-\varkappa$ obtained by removing its first row and and first column and the
determinant of $M^{({\rm i})}(z)-\varkappa$ itself:
\[
f_1(\sqrt{z}; l_1,\dots,l_n; a)=\frac{{\mathcal D}^{(1)}(\sqrt{z}; l_1, \dots, l_n; a)}{{\rm det}\bigl(M^{({\rm i})}(z)-\varkappa\bigr)}
\]
Next, we multiply by $-l_1$ both the numerator and denominator of this ratio, and pass to the limit in each of them separately:
\begin{equation}
\lim_{l_1\to0}f_1(\sqrt{z}; l_1,\dots,l_n; a)=\frac{\lim\limits_{l_1\to0}(-l_1){\mathcal D}^{(1)}(\sqrt{z}; l_1, \dots, l_n; a)}{\lim\limits_{l_1\to0}(-l_1){\rm det}\bigl(M^{({\rm i})}(z)-\varkappa\bigr)}
\label{bulky_ratio}
\end{equation}
The numerator of (\ref{bulky_ratio}) is easily computed as the
determinant ${\mathcal D}^{(2)}(z; l_1, \dots, l_n; a)$ of the minor of $M^{({\rm i})}(z)-\varkappa$ obtained by removing its first two rows and first two columns.

As for the denominator of (\ref{bulky_ratio}), we add to the second row of the matrix
$M^{({\rm i})}(z)-\varkappa$ its first row multiplied by $\cos(\sqrt{z}l_1),$ which leaves the determinant unchanged.  This operation, due to the identity
\[
-\cot(\sqrt{z}l_1)\cos(\sqrt{z}l_1)+\frac{1}{\sin(\sqrt{z}l_1)}=\sin(\sqrt{z}l_1),
\]
cancels out the singularity of all matrix elements of the second row at the point $l_1=0.$ We introduce the factor $-l_1$ ({\it cf.} \ref{bulky_ratio}) into the first row and pass to the limit as $l_1\to 0.$ Clearly, all rows but the first are regular at $l_1=0$ and hence converge to their limits as $l_1\to0.$ Finally, we add to the second column of the limit its first column, which again does not affect the determinant, and note that the first row of the resulting matrix has
one non-zero element, namely the $(1,1)$ entry. This procedure reduces the denominator in (\ref{bulky_ratio}) to the determinant of a matrix of the size reduced by one.  As in \cite{MR3430381}, it is checked that this determinant is nothing but ${\rm det}(\widetilde{M}^{({\rm i})}-\widetilde\varkappa)$, where $\widetilde M^{({\rm i})}$ and $\widetilde\varkappa$ are the Weyl matrix and the (diagonal) matrix of coupling constants pertaining to the contracted graph $\widetilde{\mathbb G}^{({\rm i})}.$ This immediately implies that the ratio obtained as a result of the above procedure coincides with the entry (1,1) of the matrix $(\widetilde M^{({\rm i})}-\widetilde\varkappa)^{-1},$ {\it i.e.}
\begin{equation*}
\lim_{l_1\to0}f_1(\sqrt{z}; l_1,\dots,l_n; a)=f_1^{(1)}(\sqrt{z}; l_2,\dots,l_n; \widetilde{a}),
\end{equation*}
where $f_1^{(1)}$ is the element (1,1) of the Robin-to-Dirichlet map of the contracted graph $\widetilde{\mathbb G}^{({\rm i})},$ and $\widetilde{a}$ is given by Definition \ref{def8}.
\end{proof}

The main result of this section is the theorem below, which is obtained as a corollary of Theorems \ref{thm9.1} and \ref{NtD_NtD}. We assume without loss of generality that $V_1\in{\mathcal V}^{{\rm (e)}}$ and denote by $f_1(\sqrt{z})$ the (1,1)-entry of the Robin-to-Dirichlet map for the set ${\mathcal V}^{{\rm (e)}}.$ We set the following notation. Fix a spanning tree ${\mathbb T}$ (see {\it e.g.} \cite{Tutte}) of the graph ${\mathbb G}^{({\rm i})}.$ We let the vertex $V_1$ to be the root of $\mathbb T$ and assume, again without loss of generality, that the number of edges in the path $\gamma_m$ connecting $V_m$ and the root is a non-decreasing function of $m.$
Denote by $N^{(m)}$ the number of vertices in the path $\gamma_m,$
and by $\bigl\{l^{(m)}_k\bigr\},$ $k=1,\dots, N^{(m)}-1,$ the associated sequence of lengths of the edges in $\gamma_m,$ ordered along the path from the root $V_1$ to
$V_m.$ Note that each of the lengths $l^{(m)}_k$ is clearly one of the edge lengths $l_j$ of the compact part of the original graph ${\mathbb G}.$

\begin{theorem}
\label{last_theorem}
Assume that the graph ${\mathbb G}$ is connected and the lengths of its compact edges are rationally independent. Given the scattering matrix $\widehat{\Sigma}_{\rm e}(s),$ $s\in{\mathbb R},$ the Robin-to-Dirichlet map for the set ${\mathcal V}^{{\rm (e)}}$ and the matrix of coupling constants $\varkappa$ are determined constructively in a unique way. Namely,
the following formulae hold for $l=1, 2,\dots, N$ and determine $a_m,$ $m=1,\dots, N:$
\[
\sum_{m: V_m\in\gamma_l}a_m
=\lim_{\tau\to+\infty}\Biggl\{-\tau\Bigl(\sum_{V_m\in\gamma_l}{\rm deg} (V_m)-2(N^{(l)}-1)\Bigr)-\frac{1}{f_1^{(l)}(i\tau)}\Biggr\},
\]
where
\begin{equation}
f_1^{(l)}(\sqrt{z}):=\lim_{l^{(l)}_{N^{(l)}-1}\to0}\dots\lim_{l^{(l)}_2\to0}\lim_{l^{(l)}_1\to0}f_1(\sqrt{z}),
\label{last}
\end{equation}
where in the case $l=1$ no limits are taken in (\ref{last}).
\end{theorem}

\section{Zero-range potentials with an internal structure}
\label{sec:models-potent-zero}

\subsection{Zero-range models}

In many models of mathematical physics, most notably in the analysis of Schrodinger operators, an explicit solution can be obtained in a very limited number of special cases (essentially, those that admit separation of variables and thus yield solutions in terms of special functions). This deficit of explicitly solvable models has led physicists, starting with E.~Fermi in 1934, to the idea to replace potentials with some boundary condition at a point of three-dimensional space, i.e., a zero-range potential.

The rigorous mathematical treatment of this idea was initiated in \cite{Berezin_Faddeev}. It was shown that the corresponding model Hamiltonians are in fact self-adjoint extensions of a Laplacian which has been restricted to the set of $W^{2,2}$ functions vanishing in a vicinity of a fixed point $x_0$ in $\mathbb{R}^3$. These ideas were further developed in a vast series of papers and books, culminating in the monograph \cite{kura}, which also contains an comprehensive list of references.

Physical applications of zero-range models have been treated in, e.g., \cite{demkov}. It has been conjectured that zero-range models provide a good approximation of realistic physical systems in at least a far-away zone, where the concrete shape of the potential might be discarded, making them especially useful in the analysis of scattering problems. Here we also mention the celebrated Kronig-Penney model where a periodic array of zero-range potentials is used to model the atoms in a crystal lattice.

Despite the obvious success of the idea explained above, it still carries a number of serious drawbacks. In particular, it can be successfully applied to model spherically symmetric scatterers only. If one attempts to model a scatterer of a more involved structure, i.e., possessing a richer spectrum, by a finite set of zero-range potentials, the complexity of the model grows rapidly, essentially eliminating the main selling point of the model, i.e., its explicit solvability.

In 1980s, based in part on earlier physics papers by Shirokov et. al. where the idea was presented in an implicit form,
B.~S.~Pavlov \cite{Pavlov_internal_structure} rigorously introduced a model of zero-range potential with an internal structure. This idea was further developed by Pavlov and his students, see e.g. \cite{Pavlov1988,AdamyanPavlov} and references therein. In the mentioned works of Pavlov, one starts by considering the operator $A_0$ being a Laplacian restricted to the set of $W^{2,2}$ functions vanishing in a vicinity of a fixed point in $\mathbb{R}^3$, precisely as in \cite{Berezin_Faddeev}. Then, instead of considering von Neumann self-adjoint extensions of the latter, one passes over to the consideration of the so-called \emph{out of space} extensions, i.e., extensions to self-adjoint operators in a larger Hilbert space. The theory of out-of-space extensions generalising that of J.~von~Neumann was constructed by M.~A.~Neumark in \cite{Neumark1, Naimark1940} in the case of densely defined symmetric operators and by M.~A.~Krasnoselskii \cite{Krasnoselskii} and A.~V.~Strauss \cite{Strauss_ext} in the case opposite, see also \cite{Strauss_survey} for the connections with the theory of functional models.

In fact, Pavlov, being quite possibly unaware of these theoretical developments, has reinvented this technique in the following way. Alongside the original Hilbert space $H=L^2(\mathbb{R}^3)$, consider an auxiliary \emph{internal} Hilbert space $E$ (which can be in many important cases considered to be finite-dimensional) and a self-adjoint operator $A$ with simple spectrum in it. Let $\phi$ be its generating vector and consider the restriction $A_\phi$ of $A$ (non-densely defined)  to the space
$$
\dom A_\phi := \{(A-{\rm i})^{-1}\psi: \psi\in E, \langle \phi,\psi\rangle =0\}.
$$
This leads to the symmetric operator $\mathcal{A}_0$ on the Hilbert space $H\oplus E$, defined as $A_0\oplus A_\phi$ on the domain
$$
\dom \mathcal{A}_0:= \left\{ \binom{f}{v}: f\in \dom A_0, v\in \dom A_\phi   \right\},
$$
where $A_0$ is the restricted Laplacian on $H$ introduced above. The operator $\mathcal{A}_0$ is then a symmetric non-densely defined operator with equal deficiency indices, and one can consider its self-adjoint extensions $\mathcal{A}$. Among them, we will single out those which non-trivially couple the spaces $H$ and $E$ by feeding the boundary data at $x_0$ of a function $f\in W^{2,2}(\mathbb{R}^3)$ to the ``part'' of operator acting in $E$. The latter then serves as the operator of the ``internal structure'', which can be chosen arbitrarily complex. We elect not to dwell on the precise way in which the extensions $\mathcal{A}$ are constructed as an explicit examples of operators of this class will be presented below in the present section.

\subsection{Connections with inhomogeneous media}
\label{connections_inhom}

Leaving the subject of zero-range models with an internal structure aside for a moment, let us briefly consider a number of physically motivated models giving rise  to zero-range potentials in general. In particular, we will be interested in those models which lead to a distribution ``potential'' $\delta'$, where $\delta$ is the Dirac delta function. It is well-known, see, e.g., \cite{kura}, that the question of relating an operator of the form $-\Delta + \alpha\delta'$ to one of self-adjoint von Neumann extensions of a properly selected symmetric restriction $A_0$ of $-\Delta$ is far from being trivial, as $\delta'$, unlike $\delta$, is form-unbounded.

For the same reason, it is non-trivial to construct an explicit ``regularisation'' of a $\delta'-$perturbed Laplacian, i.e., a sequence of operators $A_\e$ being either potential perturbations of the Laplacian or, in general, any perturbations of the latter which converge in some sense (say, in the sense of resolvent convergence) to the Laplacian with a $\delta'$ perturbation. In particular, we point out among many others the paper \cite{Carreau} where $A_\e$ are chosen as first-order differential non-self-adjoint perturbations of the Laplacian of a special form and the paper \cite{Golovaty} where the perturbation is assumed to admit the form $\e^{-2} v(x/\e)$.

It turns out that additive $\e-$dependant  perturbations are not the most straightforward choice for the task described, as zero-range perturbations (and more precisely, zero-range perturbations with an internal structure) appear naturally in the asymptotic analysis of inhomogeneous media. In particular, in the paper \cite{KCher} we studied the norm-resolvent asymptotics of differential operators $A_\e$ with periodic coefficients with high contrast, defined by their resolvents $(A_\e-z)^{-1}:f\mapsto u$ as follows:
\begin{equation}
-\bigl(a^\varepsilon\bigl(x\bigr)u'\bigr)'-zu=f,\ \ \ \ f\in L^2({\mathbb R}),\ \ \ \varepsilon>0,\ \ \ z\in{\mathbb C},
\label{orig_problem}
\end{equation}
where, for all $\varepsilon>0,$ the coefficient $a^\varepsilon$ is $1$-periodic and
\begin{equation*}
a^\varepsilon(y):=\left\{\begin{array}{lll}a \e^{-2},\ \ \ y\in[0, l_1),\\[0.3em] 1,\ \ \ y\in[l_1, l_1+l_2),\\[0.3em] a \e^{-2},\ \ \ y\in[l_1+l_2, 1),\end{array}\right.
\end{equation*}
with $a>0,$ and $0<l_1<l_1+l_2<1.$ Here in \eqref{orig_problem} the ``natural'' matching conditions are imposed at the points of discontinuity of the symbol $a^\e(x)$, i.e., the continuity of both the function itself and of its conormal derivative, so that the operators $A_\e$ can be thought as being defined by the form $\int a^\e(x)|u'(x)|^2dx$. We remark that the operators $A_\e$ are unitary equivalent to the operators of the double-porosity model of homogenisation theory in dimension one, see, e.g., \cite{HempelLienau_2000}.

The main result of the named paper can be reformulated as follows.

\begin{theorem}\label{kc}
The norm-resolvent limit of the sequence $A_\e$  is unitarily equivalent to the operator $A_{\rm hom}$
 in $L^2({\mathbb R})$ given by the differential expression
$-l_2^{-2}d^2/{dx^2}$ on
\begin{multline*}
{\rm dom}({ A}_{\rm hom})\\[0.25em]
=\bigl\{U: \forall n\in\mathbb Z\ \ U\in W^{2,2}(n, n+1),\ \ U'\in C(\mathbb R),\ \forall n\in\mathbb Z\ \ U(n+0)-U(n-0)=
l_2^{-1}(l_1+l_3)U'(n)\bigr\},
\end{multline*}
where $l_3:=1-(l_1+l_2)$.
Moreover, for all $z$ in a compact set $K_\sigma$ such that the distance of the latter from the positive real line is not less than a fixed $\sigma>0$, this norm resolvent convergence is uniform, with the (uniform) error bound $O(\e^2)$.
\end{theorem}

By inspection, the operator $A_{\rm hom}$ defined above corresponds to the formal differential expression
$$
-l_2^{-2}\frac{d^2}{dx^2} + \frac{(l_1+l_3)} {l_2}\sum_{n\in \mathbb{Z}} \delta'(x-n),
$$
i.e., it is the operator of the Kronig-Penney dipole-type model on the real line. It is also quite clear that the periodicity of the model considered has nothing to do with the fact that the effective operator acquires the $\delta'$-type potential perturbation.  Thus it leads to the understanding that strong inhomogeneities in the media in generic (i.e., not necessarily periodic) case naturally give rise to zero-range potentials of $\delta'$-type.

In order to relate this exposition to our subject of zero-range potentials with an internal structure, let us describe the main ingredient leading to the result formulated above. As usual in dealing with periodic problems, we apply the Gelfand transform
\begin{equation}
\hat{U}(y,\tau)=(2\pi)^{-d/2}\sum_{n\in{\mathbb Z}^d}U(y+n)\exp\bigl(-{\rm i}\tau\cdot(y+n)\big),\ \ \ y\in
[0,1],
\ \tau\in [-\pi,\pi),
\label{Gelfand_formula}
\end{equation}
to the operator family $A_\e$, which yields the operator family $A_\e^{(\tau)}$ corresponding to the differential expression
\[
-\biggl(\frac d{dx}+{\rm i}\tau\biggr){a^\e(x)}\biggl(\frac d{dx}+{\rm i}\tau\biggr)
\]
on the interval $[0,1]$ with periodic boundary conditions at the endpoints. Here $\tau\in[-\pi, \pi)$ is the quasimomentum. As above, the matching conditions at the points of discontinuity of the symbol $a^\e(x)$ are assumed to be natural.

The asymptotic analysis of the operator family $A_\e^{(\tau)}$, as shown in \cite{KCher,CEKN}, yields the following operator as its norm-resolvent asymptotics.
Let $H_\hom=H_\soft\oplus \mathbb C^1$. For all values $\tau\in[-\pi, \pi),$ consider a self-adjoint operator $\mathcal A_{\rm hom}^{(\tau)}$ on the space $H_\hom,$ defined as follows. Let the domain $\dom \mathcal A_{\hom}^{(\tau)}$ be defined as
\begin{equation*}
\dom \mathcal A_{\hom}^{(\tau)}=\Bigl\{(u,\beta)^\top\in H_\hom:\ u\in W^{2,2}(0,l_2), u(0)=\overline{\xi^{(\tau)}} u(l_2)=\beta/\sqrt{l_1+l_3}\Bigr\}.
\end{equation*}

On $\dom \mathcal A_{\hom}^{(\tau)}$ the action of the operator is set by
$$
\mathcal A_{\hom}^{(\tau)}\binom{u}{\beta}=
\left(\begin{array}{c}\biggl(\dfrac{1}{\rm i}\dfrac{d}{dx}+\tau\biggr)^2\\[0.7em]
-\dfrac{1}{\sqrt{l_1+l_3}}
\bigl(\partial^{(\tau)} u\bigr|_0 - \overline{\xi^{(\tau)}}\partial^{(\tau)} u\bigr|_{l_2}\bigr)
\end{array}\right).
$$
Here
\[
\xi^{(\tau)}:=\exp\bigl({\rm i}(l_1+l_3) \tau\bigr),\qquad
\partial^{(\tau)} u:=\biggl(\frac d {dx}+{\rm i}\tau\biggr) u.
\]

\begin{theorem}\label{thm:ex0}
The resolvent $(A_\e^{(\tau)}-z)^{-1}$ admits the following estimate in the uniform operator norm topology:
$$
(A_\e^{(\tau)}-z)^{-1}-\Psi^* (\mathcal A_{\hom}^{(\tau)}-z)^{-1}\Psi=O(\e^2),
$$
where $\Psi$ is a partial isometry from $H=L^2(0,1)$ to $H_\hom$.
This estimate is uniform in $\tau\in[-\pi,\pi)$ and $z\in K_\sigma$.
\end{theorem}

It is clear now that the operator $\mathcal A_{\hom}^{(\tau)}$ is nothing but the simplest possible example of Pavlov's zero-range perturbations with an internal structure, corresponding to the case where the dimension of the internal space $E$ is equal to one. The definition of $\mathcal A_{\hom}^{(\tau)}$ implies that the support of the zero-range potential here is located at the point $x_0=0$, which is identified due to quasi-periodic (of Datta--Das Sarma type) boundary conditions  with the point $x=l_2$.

Next it is shown (see \cite{KCher} for details) that under an explicit unitary transform the operator family $\mathcal A_{\hom}^{(\tau)}$    is unitary equivalent to the family $\mathcal A'_{\rm hom}(\tau')$ at the quasimomentum point $\tau'=\tau+\pi ({\rm mod\ } 2\pi)$. Here $\mathcal A'_{\rm hom}(\tau')$ acts in the space $L^2[0,l_2]$
and is defined by the same differential expression as $\mathcal A_{\rm hom}^{(\tau)},$ with the parameter $\tau$ replaced by
$\tau':$
\[
\biggl(\frac{1}{\rm i}\frac {d}{dx}+\tau'\biggr)^2,
\]
on the domain described by the conditions
\begin{eqnarray*}
u(0)+{\rm e}^{-{\rm i}(l_1+l_3)\tau'} u(l_2)=(l_1+l_3) {\partial^{(\tau')}} u\bigr|_0,
\\[0.4em]
{\partial^{(\tau')}} u\bigl|_0=-{\rm e}^{-{\rm i}(l_1+l_3)\tau'}{\partial^{(\tau')}} u\bigr|_{l_2}.
\end{eqnarray*}

An application of the inverse Gelfand transform then yields Theorem \ref{kc}. This shows that the operator $\mathcal A_{\hom}^{(\tau)}$ which is an operator of a zero-range model with the internal space $E$ of dimension one is in fact a differential operator with a $\delta'-$potential, up to a unitary transformation. In view of \cite{Shkalikov_1983,Kurasov_supersingular} it is plausible that by a similar argument it could be shown, that an operator with a $\delta^{(n)}-$potential can be realized as a zero-range model with $\dim E= n$, for any natural $n$.

It is interesting to note that an operator admitting the same form as  $\mathcal A_{\hom}^{(\tau)}$ (with $\tau=0$)   appears naturally
in the setting of
\cite{Exner, KuchmentZeng, KuchmentZeng2004},
who discuss the behaviour of the spectra of operator sequences associated with domains
``shrinking'' as $\e\to 0$ to a metric graph embedded into $\mathbb{R}^d$.
Here the rate of shrinking of the ``edge'' parts
is assumed to be related to the rate of shrinking of the ``vertex'' parts of the domain via
\begin{equation}
\frac{\text{vol}(V^\varepsilon_{\rm vertex})}{\text{vol}(V^\varepsilon_{\rm edge})}\to\alpha>0, \ \ \ \varepsilon\to0.
\label{shrinking_rate}
\end{equation}
It is shown in the above works
that the spectra of the corresponding Laplacian operators with Neumann boundary conditions
converge to the spectrum of a
quantum graph associated with a Laplacian on the metric graph obtained as the limit of the domain as $\varepsilon\to0.$
The ``weight'' $l_1+l_3$ in $\mathcal A_{\hom}^{(\tau)}$
plays the r\^{o}le of the constant $\alpha$ in (\ref{shrinking_rate}).

By a similar argument to the one presented above one can show, that in the case of domains shrinking to a graph under the ``resonant'' condition (\ref{shrinking_rate}) one obtains, under a suitable unitary transform, the matching condition of $\delta'$ type at the internal graph vertices, with the corresponding coupling constant equal to $\alpha$.

\subsection{A PDE model: BVPs with a large coupling}

\label{PDE_model_bounded}

\subsubsection{Problem setup}

In \cite{CherKisSilva3}, we studied a prototype large-coupling transmission problem, posed on a bounded domain $\Omega\subset{\mathbb R}^d,$ $d=2,3,$ see Fig.\,\ref{fig:kaplya}, containing a ``low-index" (equivalently, ``high propagation speed") inclusion $\Omega_-,$ located at a positive distance to the boundary $\partial\Omega.$ Mathematically, this is modelled by a ``weighted" Laplacian $-a_{\pm}\Delta$, where $a_+=1$ (the weight on the domain
$\Omega_+:=\Omega\setminus\overline{\Omega}_-$), and $a_-\equiv a$ (the weight on the domain $\Omega_-$) is assumed to be large, $a_-\gg1.$ This is supplemented by the Neumann boundary condition
$\partial u/\partial n=0$ on the outer boundary $\partial\Omega,$
where $n$ is the exterior normal to $\partial\Omega,$ and ``natural" continuity conditions on the ``interface" $\Gamma:=\partial\Omega_-.$ For each $a,$ we consider time-harmonic vibrations
of the physical domain
represented by $\Omega,$ described by the eigenvalue problem for an appropriate operator in $L^2(\Omega).$
\begin{figure}[h!]
	\begin{center}
		\includegraphics[scale=0.45]{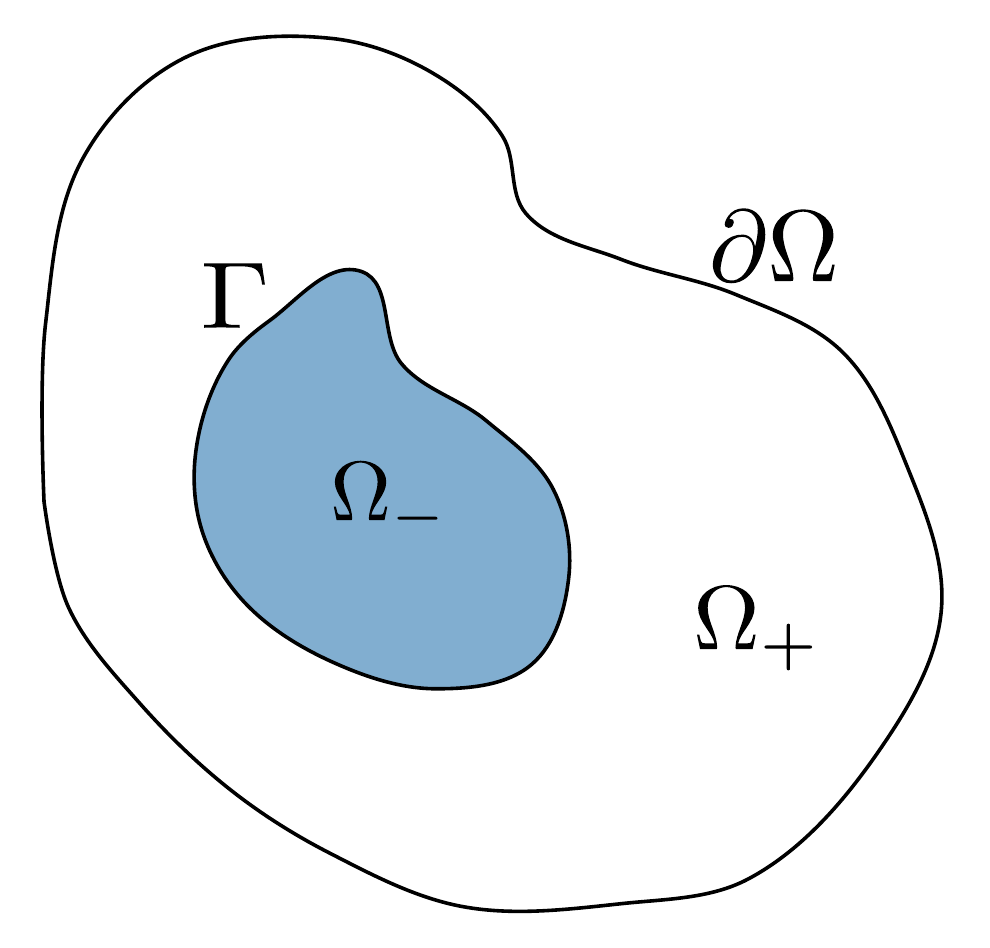}
	\end{center}
	\caption{Domain with a ``stiff" inclusion.\label{fig:kaplya}}
\end{figure}
It is easily seen that eigenfunction sequences for these eigenvalue problems converge, as $a\to\infty,$ to either a constant or a function of the form
\[
v-\frac{1}{|\Omega|}\int_{\Omega_+}v,
\]
where
$v$ satisfies the spectral boundary-value problem (BVP)
\begin{equation}
	\label{electrostatic}
	-\Delta v=z\biggl(v-\frac{1}{|\Omega|}\int_{\Omega_+}v\biggr)\ \ \ {\rm in}\ \Omega_+,\ \quad v\vert_\Gamma=0,
	\qquad \dfrac{\partial v}{\partial n}\biggr\vert_{\partial\Omega}=0.
\end{equation}
Here the spectral parameter $z$ represents
the ratio of the size of the original physical domain to the wavelength in its part
represented by $\Omega_+.$

The problem (\ref{electrostatic})
is isospectral to the so-called ``electrostatic problem'' discussed in \cite[Lemma 3.4]{Zhikov_2004}, see also \cite{AKKL} and references therein, namely the eigenvalue problem for the self-adjoint operator $Q$ defined by the quadratic form
\begin{equation*}
	q(u, u)=\int_{\Omega_+}\nabla v\cdot\overline{\nabla v},\qquad u=v+c,\quad v\in H^1_{0,\Gamma}:=\bigl\{v\in H^1(\Omega_+),\ v\vert_\Gamma=0\bigr\},\quad c\in{\mathbb C}
\end{equation*}
on the Hilbert space $L^2(\Omega_+)\dotplus \mathbb{C}$, treated as a subspace of $L^2(\Omega)$.

Denote by $A_0^+$ the Laplacian $-\Delta$ on $\Omega_+,$ subject to the Dirichlet condition on $\Gamma$ and the Neumann boundary condition
on $\partial\Omega$ and let
$\lambda^+_j,$ $\phi^+_j,$ $j=1,2,\dots,$ be the eigenvalues and the corresponding orthonormal eigenfunctions, respectively, of $A_0^+.$

It is then easily shown, that the spectrum of the electrostatic problem is the union of two sets: a) the set of $z$ solving the equation
\begin{equation*}
	z\Biggl[|\Omega|+z\sum_{j=1}^\infty(\lambda_j^+-z)^{-1}\biggl|\int_{\Omega_+}\phi^+_j\biggr|^2\Biggr]=0.
\end{equation*}
and b) the set of those eigenvalues $\lambda^+_j$
for which the corresponding eigenfunction $\phi^+_j$ has zero mean over $\Omega_+.$

\subsubsection{Norm-resolvent convergence to a zero-range model with an internal structure}
\label{large_coupling_lim}

Suppose that $\Omega$ is a bounded $C^{1,1}$ domain, and $\Gamma\subset\Omega$ is a closed $C^{1,1}$ curve, so that $\Gamma=\partial\Omega_-$ is the common boundary of domains $\Omega_+$ and $\Omega_-,$ where $\Omega_-$ is strictly contained in $\Omega,$ such that $\overline{\Omega}_+\cup\overline{\Omega}_-=\overline{\Omega},$ see Fig.\,\ref{fig:kaplya}.

For $a>0,$ $z\in{\mathbb C}$
we consider the ``transmission" eigenvalue problem ({\it cf.} \cite{Schechter})
\begin{equation}
\label{eq:transmissionBVP}
\begin{cases}
&-\Delta u_+=zu_+\ \ {\rm in\ } \Omega_+,\\[0.4em]
&-a\Delta u_-=zu_-\ \ {\rm in\ } \Omega_-,\\[0.4em]
&u_+=u_-,\quad \dfrac{\partial u_+}{\partial n_+}+a\dfrac{\partial u_-}{\partial n_-}=0\ \ {\rm on\ } \Gamma,\\[0.7em]
&
\dfrac{\partial u_+}{\partial n_+}=0\ \
{\rm \ on\ } \partial\Omega,
\end{cases}
\end{equation}
where $n_\pm$ denotes the exterior normal (defined a.e.) to the corresponding part of the boundary. The above problem is understood in the strong sense,
{i.e.} $u_\pm\in H^2(\Omega_\pm),$ the Laplacian differential expression $\Delta$ is the corresponding combination of second-order weak derivatives, and the boundary values of $u_\pm$ and their normal derivatives are understood in the sense of traces according to the embeddings of $H^2(\Omega_\pm)$ into $H^s(\Gamma),$ $H^s(\partial\Omega),$ where $s=3/2$ or $s=1/2.$

Denote by $A_a$ the operator of the above boundary value problem. Its precise definition is given on the basis of the boundary triples theory in the form of \cite{Ryzh_spec}.

Consider the space $H_{\rm eff}=L^2(\Omega_+)\oplus \mathbb C$ and the following linear subset of $L^2(\Omega):$
\begin{equation*}
	\dom \mathcal A_{\rm eff}=\biggl\{
	\binom{u_+}{\eta}\in H_{\rm eff}:\ u_+\in H^2(\Omega_+),\ \ u_+|_\Gamma=\frac{\eta}{\sqrt{|\Omega_-|}}{\mathbbm 1}_\Gamma, \ \ \dfrac{\partial u_+}{\partial n_+}\biggr\vert_{\partial\Omega}=0\biggr\},
\end{equation*}
where $u|_\Gamma$
is the trace of the function $u$  and ${\mathbbm 1}_\Gamma$ is the unity function on $\Gamma.$
On $\dom \mathcal A_{\rm eff}$ we set the action of the operator ${\mathcal A}_{\rm eff}$ by the formula
\vskip -0.3cm
\begin{equation}
	\mathcal A_{\rm eff}\binom{u_+}{\eta}=
	\left(\begin{array}{c}-\Delta u_+\\[0.6em]
		\dfrac{1}{\sqrt{|\Omega_-|}}\mathop{{{\int}}}_{\Gamma}\dfrac{\partial u_+}{\partial n_+}
	\end{array}\right).
\label{A_eff_form}
\end{equation}



\begin{theorem}
	\label{mmm}
	The  operator $\mathcal A_{\rm eff}$ is the norm-resolvent limit of the operator family $A_a$. This convergence is uniform for $z\in K_\sigma$, with an error estimate by $O(a^{-1})$.
\end{theorem}

This theorem yields in particular the convergence (in the sense of Hausdorff) of the spectra of $A_a$ to that of $\mathcal A_{\rm eff}$. This convergence is uniform in $K_\sigma$, and its rate is estimated as $O(a^{-1})$. Moreover, it is shown that the spectrum of  $\mathcal A_{\rm eff}$ coincides with the spectrum of the electrostatic problem (\ref{electrostatic}).


Note that the form of $\mathcal A_{\rm eff}$ is once again identical to that of a zero-range model with an internal structure in the case when the internal space $E$ is one-dimensional. The obvious difference is that here the effective model of the medium is no longer ``zero-range'' per se; rather it pertains to a singular perturbation supported by the boundary $\Gamma$. Therefore, the result described above allows one to extend the notion of internal structure to the case of distributional perturbations supported by a curve, see also \cite{Kurasov_surfaces} where this idea was first suggested, although unlike above no asymptotic regularisation procedure was considered. Moreover, well in line with the narrative of preceding sections, the internal structure appears owing exclusively to the strong inhomogeneity of the medium considered.

We remark that a ``classical'' zero-range perturbation with an internal structure can still be obtained by a rather simple modification of the problem considered. Namely, let $a$ be fixed, and let the \emph{volume} of the inclusion $\Omega_-$ now wane to zero as the new parameter $\e\to 0$. This represents a model that has been studied in detail, see, e.g., \cite{Cap} and references therein. In this modified setup, a virtually unchanged argument leads to the inclusion being asymptotically modelled by a zero-range potential with an internal structure. Moreover, the dimension of the internal space $E$ is again equal to one, provided that a uniform norm-resolvent convergence is sought for the spectral parameter belonging to the compact $K_\sigma$.

\subsubsection{Internal structure with higher dimensions of internal space $E$}
\label{internal_sec}

A natural question must therefore be posed: can strongly inhomogeneous media only give rise to  simplest possible zero-range models with an internal structure, pertaining to the case of $\dim E=1$, or  is it possible to obtain effective models with more involved internal structures? It turns out that the second mentioned possibility is realised, which we will demonstrate briefly using the material of the preceding section.

Recall that in all  the results formulated above the uniform convergence was claimed under the additional assumption that the spectral parameter belongs to a fixed compact. If one drops this assumption, within the setup of the previous section one has the following statement.
\begin{theorem}
Up to a unitary equivalence, for and $k\in{\mathbb N}$ there exists a self-adjoint operator $\mathcal A_{\rm eff}$ of a zero-range model with an internal structure on the space $H_{\rm eff}:=L^2(\Omega_+)\oplus \mathbb{C}^k$ such that
\begin{equation}
(A_a-z)^{-1}\simeq \mathfrak{P}(\mathcal A_{\rm eff}-z)^{-1}\mathfrak{P} + O\bigl(\max\{a^{-1},|z|^{k+1}a^{-k}\}\bigr)
\label{double_star}
\end{equation}
in the uniform operator norm topology. Here $\mathfrak{P}$ is the orthogonal projection of $H_{\rm eff}$ onto $L^2(\Omega_+)\oplus \mathbb C$ (i.e., the space $H_{\rm eff}$ of the previous section).
\end{theorem}
Note that unlike the results pertaining to the situation of the spectral parameter contained in a compact, here the leading-order term of the asymptotic expansion of the resolvent of the original operator $A_a$ is not the resolvent of some self-adjoint operator (unless $k=1$), but rather a generalised resolvent. It is also obvious that the concrete choice of $k$ to be used in the last theorem depends on the concrete relationship between $z$ and $a$ and on the error estimate sought: the error estimate of the theorem becomes tighter at $k$ increases. In essence, this brings about the understanding that despite the fact that on the face of it the problem at hand is one-parametric, it must be treated as having two parameters, $z$ and $a$.

The operator $\mathcal A_{\rm eff}$ of the last Theorem admits an explicit description for any $k\in\mathbb{N}$, but this description is rather involved. In view of better readability of the paper, we only present its explicit form in the case $k=2$:
\begin{equation*}
\mathcal A_{\rm eff}\left(\begin{array}{c}{u_+}\\{\eta_1}\\\eta_2\end{array}\right):=
\left(\begin{array}{c}-\Delta u_+\\[0.2em]
\dfrac{1}{\kappa}\mathop{{{\int}}}_{\Gamma}\dfrac{\partial u_+}{\partial n_+ }+ a(B^2 D^{-1}\eta_1 + B\eta_2)\\[0.8em]
a({B}\eta_1+D \eta_2)
\end{array}\right).
\end{equation*}
Here $B,D$ and $\kappa$ are real parameters, which are explicitly computed.

It should be noted that similar results can be obtained in the homogenisation-related setup of the previous section, see also Section \ref{sec:appl-cont-mech}.

\medskip

We can therefore conclude that zero-range models with an internal structure appear naturally in the asymptotic analysis of highly inhomogeneous media. Moreover, in the generic case they appear as Neumark-Strauss dilations (see \cite{Naimark1940, Naimark1943, Strauss, Strauss_survey}) of main order terms in the asymptotic expansions of the resolvents of problems considered. The complexity of the internal structure can be arbitrarily high (i.e., the dimension of the internal space $E$ can be made as high as required), provided that the spectral parameter $z$ is allowed to grow with the parameter $a\to\infty$ (or $\e\to 0$). Further, owing to the remark made above that an operator with a $\delta^{(n)}$-potential could be realized as a zero-range model with $\dim E= n$ for any natural $n$, we expect models with strong inhomogeneities to admit the role of the tool of choice in the regularisation of singular and super-singular perturbations, beyond the form-bounded case and including the case of singular perturbations supported by a curve or a surface.

\subsection{The r\^{o}le of generalised resolvents}

\label{gen_res2}

We close this section with a brief exposition of how precisely the asymptotic results formulated above are obtained. The analysis starts with the family of resolvents, say (in the case of Section \ref{connections_inhom}) $(A_\e -z)^{-1}$, describing the inhomogeneous medium at hand. One then passes over to the generalised resolvent $R_\e(z):=P(A_\e-z)^{-1}P^*,$ where $P$ denotes the orthogonal projection onto the ``part'' of the medium which is obtained by removing the inhomogeneities. Note that the generalised resolvent thus defined is a solution operator of a BVP pertaining to homogeneous medium, albeit subject to non-local $z-$dependant boundary conditions. The problem considered therefore reduces to the asymptotic analysis of the operator $B_\e(z)$, parameterising these conditions. As such, it becomes a classical problem of perturbation theory.

Assuming now, for the sake of argument, that $R_\e(z)$ has a limit, as $\varepsilon\to0,$ in the uniform operator topology for $z$ in a domain $D\subset\mathbb C$, and, further, 
that the resolvent $({A}_\e-z)^{-1}$ also admits such limit, one clearly has
\begin{equation}
P\bigl(A_{\text{eff}}-z\bigr)^{-1}P^*=R_0(z),\quad z\in D\subset \mathbb C,
\label{R_def}
\end{equation}
where $R_0$ and $A_{\text{eff}}$ are the limits introduced above. The idea of simplifying the required analysis by passing to the resolvent ``sandwiched'' by orthogonal projections onto a carefully chosen subspace is in fact the same as in \cite{MR0217440}, where the resulting sandwiched operator is shown to be the resolvent of a dissipative operator. 

The function  $R_0$ defined by (\ref{R_def}) is a generalised resolvent, whereas $A_{\text{eff}}$ is its out-of-space self-adjoint extension (or {\it Neumark-Strauss dilation} \cite{Strauss}). By a theorem of Neumark \cite{Naimark1940} (see Section \ref{gen_res1} of the present paper) this dilation is defined uniquely up to a unitary transformation of a special form, provided that the minimality condition
holds. The latter can be reformulated along the following lines: one has minimality, provided that there are no eigenmodes in the effective media modelled by the operator $A_\varepsilon,$ and therefore in the medium modelled by the operator $A_{\text{eff}}$ as well, such that they ``never enter'' the part of the medium without inhomogeneities. A quick glance at the setup of our models helps one immediately convince oneself that this must be true.
It then follows that the effective medium is completely determined, up to a unitary transformation, by $R_0(z).$ Once this is established, it is tempting to construct its Neumark-Strauss dilation and conjecture, that it is precisely this dilation that the original operator family converges to in the norm-resolvent sense (of course, up to a unitary transformation).

This conjecture in fact holds true, although it is impossible to prove it on the abstract level: taking into account no specifics of problems at hand, one can claim weak convergence at best. Still, the approach suggested seems to be very transparent in allowing to grasp the substance of the problem and to almost immediately ``guess'' correctly the operator modelling the effective medium.


\section{Applications to continuum mechanics and wave propagation}
\label{sec:appl-cont-mech}

Parameter-dependent problems for differential equations have traditionally attracted much interest within applied mathematics, by virtue of their potential for
replacing complicated formulations with more straightforward, and often explicitly solvable, ones. This drive has led to a plethora of asymptotic techniques, from perturbation theory to multi-scale analysis, covering a variety of applications to physics, engineering, and materials science. While this subject area can be viewed as ``classical", problems that require new ideas continue emerging, often motivated by novel wave phenomena. One of the recent application areas of this kind is provided by composites and structures involving components with highly contrasting material properties  (stiffness, density, refractive index). Mathematically, such problems lead to boundary-value formulations for differential operators with parameter-dependent coefficients. For example, problems of this kind have arisen in the study of periodic composite media with ``high contrast" (or ``large coupling")
between the material properties of the components, see \cite{HempelLienau_2000, Zhikov2000, CherErKis}.

In what follows, we outline how the contrast parameter emerges as a result of dimensional analysis, using a scalar elliptic equation of second order with periodic coefficients as a prototype example.

\subsection{Scaling regimes for high-contrast setups}

 We will consider the physical context of elastic waves propagating through a medium with whose elastic moduli vary periodically in a chosen plane (say $(x_1,x_2)$-plane) and are constant in the third, orthogonal, direction (say, the $x_3$ direction). For example, one could think of a periodic arrangement of parallel fibres of a homogeneous elastic material within a ``matrix" of another homogeneous elastic material. We will look at the ``polarised" anti-plane shear waves, which can be described completely by a scalar function representing the displacement of the medium in the $x_3$ direction. In the case of the fibre geometry mentioned above, the relevant elastic moduli ${G}$ then have the form
\[
{G}(y)=\left\{\begin{array}{c}{G}_0,\quad y\in Q_0,\\[0.2em] {G}_1,\quad y\in Q_1,\end{array}\right.\ \ =:\left\{\begin{array}{c}{G}_0\\[0.1em] {G}_1\end{array}\right\}(y).
\]
where $Q_0,$ $Q_1$ are the mutually complementary cross-sections of the fibre and matrix components, respectively, so that $\overline{Q}_0\cup\overline{Q}_1=[0,1]^2.$ the mass density of the described composite medium is assumed to be constant. (The constants $G_0,$ $G_1$ are the so-called shear moduli of the materials occupying $Q_0,$ $Q_1.$) This physical setup was considered in \cite{MMP, poulton}.

Denote by $d$ the period of the original ``physical" medium and consider time-harmonic wave motions, i.e. solutions of the wave equation that have the form
\begin{equation}
U(x, t)={\rm e}^{{\rm i}\omega t}u(x),\qquad x\in{\mathbb R}^2,\ \ t\ge0,
\label{harmonic}
\end{equation}
where $\omega$ is a fixed frequency. In the setting of time-harmonic waves, see (\ref{harmonic}), the function $u=u(x)$  satisfies the following equation, written in terms of the original physical units:
\begin{equation}
-\nabla_x\cdot\left\{\begin{array}{c}{G}_0\\[0.1em] {G}_1\end{array}\right\}(x/d)\nabla_xu=\rho\omega^2u,
\label{sigma_eq}
\end{equation}
Multiply both sides by ${G}_1^{-1}$ and denote $\delta:={G}_0/{G}_1.$ the parameter $\delta$ represents the ``inverse contrast", which will be assumed ``small" later, and corresponds to the value $a^{-1}$ of the ``large" parameter of Sections \ref{large_coupling_lim}, \ref{internal_sec}. The equation (\ref{sigma_eq}) takes the form
\[
-\nabla_x\cdot\left\{\begin{array}{c}\delta\\ 1\end{array}\right\}(x/d)\nabla_xu=\frac{\rho}{{G}_1}\omega^2u
\]
Note that
\begin{equation}
	\omega=\frac{2\pi c_1}{\lambda_1}=\frac{2\pi c_0}{\lambda_0},
	\label{omega_form}
\end{equation}
where $c_j,$ $\lambda_j$ are the wave speed and wavelength in the relevant media ($j=0,1$).

Introduce a non-dimensional spatial variable $\tilde{x}=2\pi x/\lambda_1:$
\[
-\frac{4\pi^2}{\lambda_1^2}\nabla_{\tilde{x}}\cdot\left\{\begin{array}{c}\delta\\ 1\end{array}\right\}
\biggl(\frac{\tilde{x}}{2\pi d/\lambda_1}\biggr)\nabla_{\tilde{x}}u=\frac{\rho}{{G}_1}\omega^2u,
\]
equivalently, with $\varepsilon:=2\pi d/\lambda_1:$
\begin{equation*}
	-\nabla_{\tilde{x}}\cdot\left\{\begin{array}{c}\delta\\ 1\end{array}\right\}(\tilde{x}/\varepsilon)\nabla_{\tilde{x}}u=\frac{\rho}{{G}_1}c_1^2u,
\end{equation*}
where we have used (\ref{omega_form}). Note that $c_1\sqrt{\rho/{G}_1}=1$
and relabel $\tilde{x}$ by $x:$
\[
-\nabla_x\cdot\left\{\begin{array}{c}\delta\\ 1\end{array}\right\}(x/\varepsilon)\nabla_xu=u,
\]

Let us ``scale to the period one" i.e.  consider the change of variable $y=\tilde{x}/\varepsilon=x/d:$
\begin{equation*}
	-\varepsilon^{-2}\nabla_y\cdot\left\{\begin{array}{c}\delta\\ 1\end{array}\right\}(y)\nabla_yu=u,
\end{equation*}
or
\begin{equation*}
	-\nabla_y\cdot\left\{\begin{array}{c}\delta\\ 1\end{array}\right\}(y)\nabla_yu=\biggl(\frac{2\pi d}{\lambda_1}\biggr)^2u.
\end{equation*}
The different scaling regimes, ranging from  what we know as ``finite frequency, high contrast" to ``high frequency, high contrast", are described by setting
\begin{equation}
\varepsilon^2=\delta^\nu\widetilde{z},
\label{epsdelz}
\end{equation}
where $\widetilde{z}$ is obviously dimensionless is assumed to vary over the compact $K_\sigma,$ and $0\le\nu\le 1.$ Note that $\widetilde{z}$ can be alternatively expressed as
\begin{equation}
	\widetilde{z}=\delta^{-\nu}\rho{G}_1^{-1}(d\omega)^2.
	\label{zform}
\end{equation}

In particular, the setup analysed in the paper \cite{CherErKis} corresponds to the case $\nu=1:$
\begin{equation}
	-\delta^{-1}\nabla_y\cdot\left\{\begin{array}{c}\delta\\ 1\end{array}\right\}(y)\nabla_yu=zu.
	\label{ff_PDE}
\end{equation}

In terms of the original spatial variable $x$ the equation (\ref{ff_PDE}) takes the form
\[
-d^2\delta^{-1}\nabla_{x}\cdot\left\{\begin{array}{c}\delta\\ 1\end{array}\right\}(x)\nabla_xu=zu,
\]
or
\[
-\nabla_{x}\cdot\left\{\begin{array}{c}\delta\\ 1\end{array}\right\}(x)\nabla_xu=k^2u,
\]
where the wavenumber (i.e. ``spatial frequency") is given by
$k:=d^{-1}\sqrt{\delta z},$
so that
\begin{equation*}
(kd)^2=\frac{z}{n^2},
\end{equation*}
 where $n^2=\delta^{-1}=G_0/G_1$ is the shear modulus of the material occupying $Q_0$ relative to the material occupying $Q_1.$

The setup (\ref{ff_PDE}) is the ``periodic" version of the formulation discussed in Section \ref{PDE_model_bounded}, see also Section \ref{connections_inhom} for the one-dimensional version of a high-contrast homogenisation problem that gives rise to the same formulation. Similarly, choosing the values $\nu=2/(k+1),$ $k=2,3,\dots,$ in (\ref{epsdelz}) gives rise to ``high-frequency large-coupling" formulations, which in turn lead, in the limit as $\delta\to0,$ to effective operators with an ``internal space" of dimension $k,$ see Section \ref{internal_sec}. The parameter $\widetilde{z}$ is the related to the spectral parameter $z$ in (\ref{double_star}) via
\[
z=\widetilde{z}\delta^{1-\nu}=\widetilde{z}\delta^{\frac{k-1}{k+1}}=\widetilde{z}a^{\frac{k+1}{k-1}},\qquad k=2,3,\dots,
\]
so that the error estimate in (\ref{double_star}) is optimal for $\widetilde{z}\in K_\sigma,$ in the sense that for such $\widetilde{z}$ it yields an error of order $a^{-1}\sim |z|^{k+1}a^{-k}$ for large $a.$


\subsection{Homogenisation of composite media with resonant components}

\subsubsection{Physical motivation}

The mathematical theory of homogenisation (see
{\it e.g.} \cite{Lions, Bakhvalov_Panasenko, Jikov_book}) aims at characterising limiting, or ``effective'', properties of small-period composites. Following an appropriate non-dimensionalisation procedure, a typical problem here is to study the asymptotic behaviour of solutions to equations of the type
\begin{equation}
	-{\rm div}\bigl(A^\varepsilon(\cdot/\varepsilon)\nabla u_\e\bigr)-\widetilde{\omega}^2u_\e=f,\ \ \ \ f\in L^2({\mathbb R}^d),\quad d\ge2,\qquad \widetilde{\omega}^2\notin{\mathbb R}_+,
	\label{eq:generic_hom}
\end{equation}
where for all $\varepsilon>0$ the matrix $A^\varepsilon$ is $Q$-periodic, $Q:=[0,1)^d,$ non-negative, bounded, and symmetric. The parameter $\widetilde{\omega}$ here represents a ``non-dimensional frequency": $\widetilde{\omega}^2=\widetilde{z},$ where $\widetilde{z}$ is the spectral parameter introduced in (\ref{epsdelz}), so for example for $\nu=1$ one can set $\widetilde{\omega}=d\sqrt{\rho/G_0}\,\omega,$ see (\ref{zform}).

One proves (see \cite{Zhikov_1989, BirmanSuslina} and references therein) that when $A$ is uniformly elliptic, there exists a constant matrix $A^\hom$ such that solutions $u_\e$ to \eqref{eq:generic_hom} converge to $u_\hom$ satisfying
\begin{equation}
	-{\rm div}\bigl(A^\hom \nabla u_\hom\bigr)-\widetilde{\omega}^2u_\hom=f.
	\label{eq:generic_Ahom}
\end{equation}
In what follows we write $\omega,$ $z$ in place of $\widetilde{\omega},$
$\widetilde{z},$ implying that either the dimensional or non-dimensional version of the equation is chosen.

In recent years, the subject of modelling and engineering a class of composite media with ``unusual" wave properties (such as negative refraction) has been brought to the forefront of materials science. Such media are generically referred to as metamaterials, see {\it e.g.} \cite{Phys_book}. In the context of 
homogenisation, the result sought ({\it i.e.}, the ``metamaterial" behaviour in the limit of vanishing $\e$) belongs to the domain of the so-called time-dispersive media (see, {\it e.g.}, \cite{Tip_1998,Figotin_Schenker_2005,Tip_2006,Figotin_Schenker_2007b}). For such media, in the frequency domain one faces equations of the form
\begin{equation}
	-{\rm div}\bigl(A\nabla u\bigr)+\mathfrak B(\omega^2)u=f,\ \ \ \ f\in L^2({\mathbb R}^d),\
	\label{eq:generic_td}
\end{equation}
where $A$ is a constant matrix and $\mathfrak B(\omega^2)$ is a frequency-dependent operator in $L^2({\mathbb R}^d)$ taking the place of $-\omega^2$ in (\ref{eq:generic_hom}), if, for the sake of argument, in the time domain we started with an equation of second order in time.
If, in addition, the matrix function $\mathfrak B$ is scalar, {\it i.e.}, $\mathfrak B(\omega^2)=\beta(\omega^2)I$ with a scalar function $\beta$, the problem of the type
\begin{equation}
	-{\rm div}\bigl(A^{\rm hom}(\omega^2)\nabla u\bigr)=\omega^2 u
	\label{eq:generic_spectral_td}
\end{equation}
appears in place of the spectral problem after a formal division by $-\beta(\omega^2)/\omega^2$, with frequency-dependent (but independent of the spatial variable) matrix $A^{\rm hom}(\omega^2).$

In the equation (\ref{eq:generic_spectral_td}), in contrast to (\ref{eq:generic_Ahom}), the matrix elements of $A^{\rm hom}$, interpreted as material parameters of the medium, acquire a non-trivial dependence on the frequency, which may lead to their taking negative values in some frequency intervals. The possibility of electromagnetic media exhibiting negative refraction was envisaged in an early work \cite{Veselago}, who showed theoretically that the material properties of such media must be frequency-dependent, and the last two decades have seen a steady advance towards realising such media experimentally.
One may hope that upon relaxing the condition of uniform ellipticity on $A^\varepsilon$ one may be able to achieve a metamaterial-type response to wave propagation for sufficiently small values of $\varepsilon.$ It is therefore important to understand how inhomogeneity in the spatial variable in \eqref{eq:generic_hom} can lead, in the limit $\e\to0,$ to frequency dispersion as in (\ref{eq:generic_td}).

\subsubsection{Operator-theoretic motivation}

Already in the setting of finite-dimensional matrix algebra equations of the
form (see (\ref{eq:generic_td}))
\begin{equation}
Au+{\mathfrak B}(z)u=f,\qquad u\in{\mathbb C}^d,
\label{zdep_form}
\end{equation}
where $A=A^*\in{\mathbb C}^{d\times d},$
$f\in{\mathbb C}^d,$ ${\mathfrak B}$ is a Herglotz function with values in ${\mathbb C}^{d\times d},$  emerge when one seeks solutions to the standard resolvent equation for a block matrix:
\begin{equation}
\left(\begin{array}{cc}A&B\\[0.2em]B^*&C\end{array}\right)\left(\begin{array}{c}u\\[0.2em]v\end{array}\right)-z\left(\begin{array}{c}u\\[0.2em]v\end{array}\right)=\left(\begin{array}{c}f\\[0.2em]0\end{array}\right),\qquad\left(\begin{array}{c}u\\[0.2em]v\end{array}\right)\in{\mathbb C}^{d+k}.
\label{block_matrix}
\end{equation}
where $B\in{\mathbb C}^{k\times d},$  $C=C^*\in{\mathbb C}^{k\times k}.$

Indeed, it is the result of a straightforward calculation that (\ref{block_matrix}) implies
\[
Au-(B(C-z)^{-1}B^*+zI)u=f,
\]
whenever $-z$ is not an eigenvalues of $C,$ so (\ref{block_matrix}) implies (\ref{zdep_form}) with ${\mathfrak B}(z)=-B(C-z)^{-1}B^*-zI.$ Another consequence of the above calculation is that for any vector $(f, g)^\top\in {\mathbb C}^{d+k}$ one has
\[
u=P\biggl\{\left(\begin{array}{cc}A&B\\[0.2em]B^*&C\end{array}\right)+zI\biggr\}^{-1}P*\left(\begin{array}{c}f\\[0.2em]g\end{array}\right),
\]
where $P$ is the orthogonal projection of ${\mathbb C}^{d+k}$ onto ${\mathbb C}^k$ and $P^*$ is interpreted as a restriction to the $k$-dimensional subspace of vectors of the form $(f,0)^\top,$ $f\in{\mathbb C}^k.$

The above argument, in the more general setting of block operator matrices in a Hilbert space, likely appeared for the first time in \cite{Figotin_Schenker_2005}.  ``Generalised resolvents", i.e. objects of the form
\begin{equation}
	P({\mathcal A}-z)^{-1}P^*,
	\label{AP_form}
\end{equation}	
	 where ${\mathcal A}$ is an operator in a Hilbert space ${\mathcal H}$ and $P$ is an orthogonal projection of $\mathcal H$ onto its subspace $H$ have already been discussed  in the present survey, see Sections \ref{gen_res1}, \ref{gen_res2}. As discussed in Section \ref{gen_res1}, abstract results of Neumark and Strauss \cite{Naimark1940, Strauss} establish that solution operators for formulations (\ref{zdep_form}), where $A$ is a self-adjoint operator in a Hilbert space $H$ can be written in the form (\ref{AP_form}) for a suitable ``out-of-space" extension ${\mathcal A}.$ Therefore, a natural question is whether formulations (\ref{eq:generic_td}) can be viewed as generalised resolvents obtained by an asymptotic analysis of some parameter-dependent operator family describing a heterogeneous medium. One piece of evidence pointing at the validity of such a conjecture is the result of Section \ref{large_coupling_lim}, where the role of the operator ${\mathcal A}$ in (\ref{AP_form}) is played by ${\mathcal A}_{\rm eff},$ see (\ref{A_eff_form}).

In \cite{KCher, CEK, CEKN}
a model of a high-contrast graph periodic along a single direction was considered. A unified treatment of critical-contrast homogenisation was proposed and carried out in three distinct cases: (i) where neither the soft nor the stiff component of the medium is connected; (ii) where the stiff component of the medium is connected; (iii) where the soft component of the medium is connected. The analytical toolbox presented in these works  was then amplified to the PDE setting in \cite{CherErKis}. In the wider context of operator theory and its applications, this provides a route towards:
constructing explicit spectral representations and functional models for both homogenisation limits of critical-contrast composites and the related time-dispersive models, as well as solving the related direct and inverse scattering problems.

\subsubsection{Prototype problem setups in the PDE context}

\label{setup_section}
Consider the problem \eqref{eq:generic_hom} under the following assumptions:
\begin{equation*}
	A^\e(y)=
	\begin{cases}
		aI,& y\in Q_{\text{stiff}},\\[0.25em]
		\varepsilon^2I, & y\in Q_{\text{soft}},
	\end{cases}
\end{equation*}
where $Q_{\text{soft}}$ ($Q_{\text{stiff}}$) is the soft (respectively, stiff) component of the unit cube $Q=[0,1)^d\subset\mathbb R^d$, so that $\overline{Q}=\overline{Q}_{\text{soft}}\cup \overline{Q}_{\text{stiff}},$ and $a>0.$

Two distinct setups were studied in \cite{CherErKis}.
For one of them (``Model I"), which is unitary equivalent to the model of \cite{HempelLienau_2000, Friedlander}, the component $Q_{\text{soft}}\subset Q$ is simply connected and
and its distance to $\partial Q$ is positive, {\it cf.} \cite{Zhikov2000, CherCoop}. For the other one (``Model II") the component $Q_{\text{stiff}}$ has the described properties.
It is assumed that the Dirichlet-to-Neumann maps for $Q_{\rm soft}$ and $Q_{\rm stiff},$ which map the boundary traces of harmonic functions in $Q_{\rm soft}$ and $Q_{\rm stiff}$  to their boundary normal derivatives, are well-defined as pseudo-differential operators of order one in the $L^2$ space on the boundary \cite{Hoermander, Friedlander_old, Arendt, AGW}.
\begin{figure}[h!]
	\begin{center}
		\includegraphics[scale=0.3]{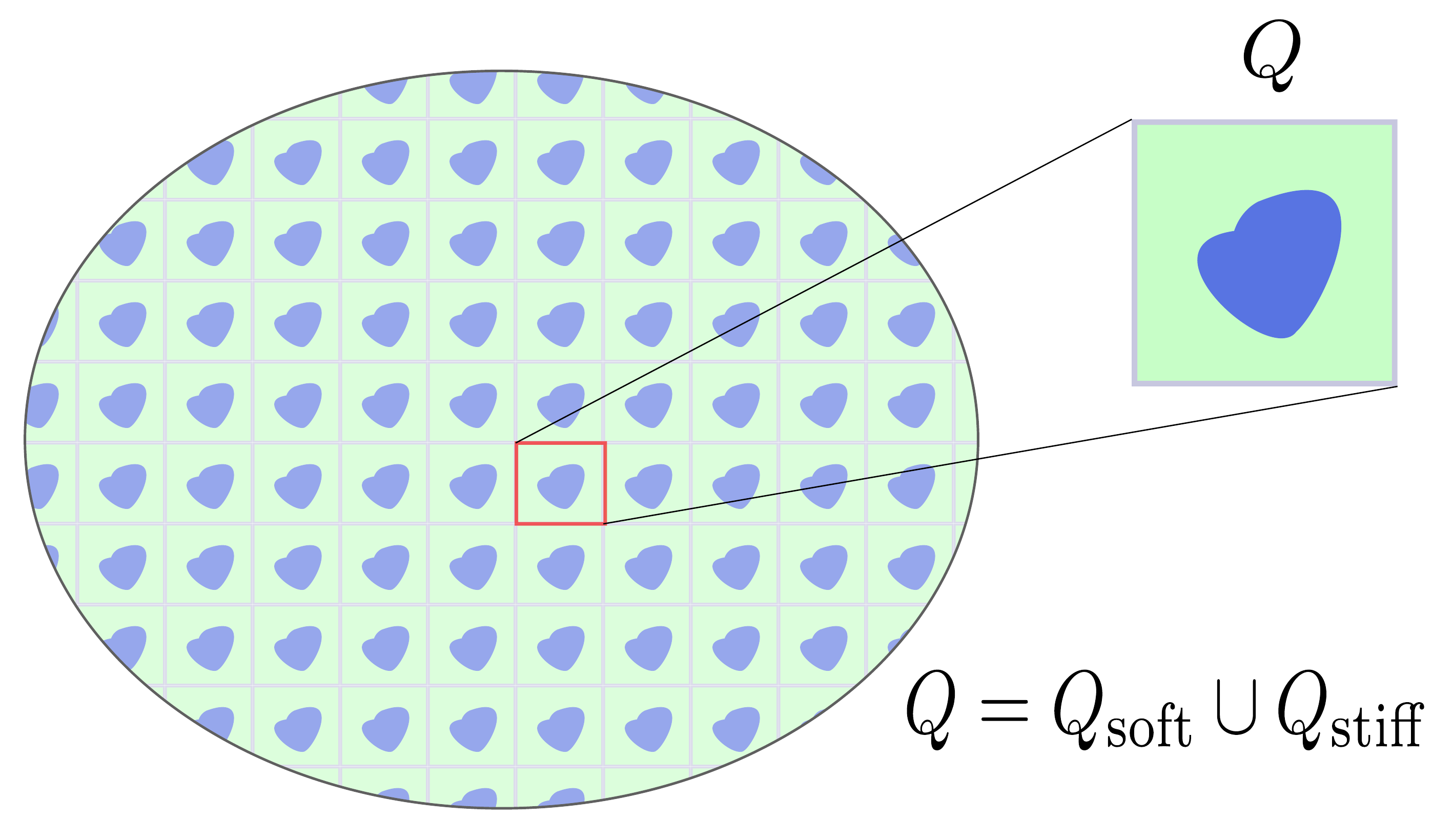}
	\end{center}
	\caption{{\scshape Model setups.} {\small Model I: soft component $Q_\soft$ in blue, stiff component $Q_\stiff$ in green. Model II: soft component $Q_\soft$ in green, stiff component $Q_\stiff$ in blue.}\label{fig:media}}
\end{figure}


For both above setups, \cite{CherErKis} deals with the resolvent $({A}_\e-z)^{-1}$ of a self-adjoint operator in $L^2(\mathbb R^d)$ corresponding to the problem \eqref{eq:generic_hom}, so that its solutions are expressed as $u_\e=({A}_\e-z)^{-1}f$ with $z=\omega^2.$ For each $\varepsilon>0,$ the operator ${A}_\e$ is defined by the forms
$$
\int_{{\mathbb R}^d}A^\e(\cdot/\e)\nabla u\cdot \overline{\nabla u},\qquad u, v\in H^1({\mathbb R}^d).
$$
It is assumed that $z\in \mathbb C$ is separated from the spectrum of the original operator family, in particular $z\in K_\sigma,$ where $K_\sigma$ is defined in Theorem \ref{kc}.

In order to deal with operators having compact resolvents, it is customary to apply
Gelfand  transform \cite{Gelfand}, which we review next.

\subsubsection{Gelfand transform and direct integral}
\label{Gelfand_section}
The version of the Gelfand transform convenient for the analysis of the operators $A_\varepsilon$ is defined on functions $u\in L^2(\mathbb R^d)$ by the formula\footnote{The formula (\ref{Gelfand_formula}) is first applied to continuous functions $U$ with compact support, and then extended to the whole of $L^2({\mathbb R}^d)$ by continuity.} ({\it cf.} (\ref{Gelfand_formula}))
\begin{equation*}
	G_\varepsilon u(y,t):=
	\biggl(\frac{\varepsilon^2}{2\pi}\biggr)^{d/2}\sum_{n\in{\mathbb Z}^d}u\bigl(\varepsilon(y+n)\bigr)\exp\bigl(-{\rm i}\tau\cdot(y+n)\big),\ \ \ y\in Q,\ \ \tau\in Q':=[-\pi,\pi)^d,
\end{equation*}
This yields a unitary operator $G_\varepsilon: L^2({\mathbb R}^d)\longrightarrow L^2(
Q\times Q'),$ and the inverse
with the inverse mapping given by
\begin{equation*}
	u(x)=(2\pi)^{-d/2}\int_{Q'}G_\varepsilon u\biggl(\frac{x}{\varepsilon},\tau\biggr)\exp\biggl({\rm i}\tau\cdot \frac{x}{\varepsilon}\biggr)d\tau,\ \ \ \ x\in{\mathbb R}^d,
\end{equation*}
where $G_\varepsilon u$ is extended to ${\mathbb R}^d\times Q'$ by $Q$-periodicity in the spatial variable.

As in \cite{BirmanSuslina}, an application of the Gelfand transform $G$ to the operator family ${A}_\e$ corresponding to the problem \eqref{eq:generic_hom} yields the two-parametric family $A_\e^{(\tau)}$ of operators in $L^2(Q)$ given by the differential expression
$$
-(\nabla + {\rm i}\tau)A^\e(x/\e)(\nabla +{\rm i}\tau), \qquad \varepsilon>0,\quad \tau\in Q',
$$
subject to periodic boundary conditions $\partial(\varepsilon Q)$ and defined by the corresponding closed coercive sesquilinear form.
For each $\varepsilon>0,$ the operator ${A}_\e$ is then unitary equivalent to the von Neumann integral (see {\it e.g.} \cite[Chapter 7]{MR1192782}) of $A_\varepsilon^{(\tau)}:$
$$
{A}_\e=G_\varepsilon^*\biggl(\oplus \int_{Q'} {A}_\e^{(\tau)}d\tau\biggr)G_\varepsilon.
$$

Similar to \cite{Friedlander} and facilitated by the abstract framework of \cite{Ryzh_spec},
the operator ${A}_\e^{(\tau)}$ can be associated to transmission problems \cite{Schechter}, akin to those considered in Section \ref{large_coupling_lim}. To this end, consider $Q$ as a torus with the opposite parts of $\partial Q$ identified, and view $Q_\soft$ and $Q_\stiff$ as subsets of this torus.
Furthermore, in line with the notation of Section \ref{large_coupling_lim}, denote by $\Gamma$ the interface between $Q_{\rm soft}$ and $Q_{\rm stiff}.$  For each $\varepsilon,$ $\tau,$ $f\in L^2(Q),$ the transmission problem is formulated as finding a function $u\in L^2(Q)$ such that $u\vert_{Q_\soft}\in H^1(Q_\soft),$ $u\vert_{Q_\stiff}\in H^1(Q_\stiff),$
that solves, in the weak sense, the boundary-value problem ({\it cf.} (\ref{eq:transmissionBVP}))
\begin{equation*}
	\begin{cases}
		&-\e^{-2}(\nabla +{\rm i}\tau)^2u_+-zu_+=f\quad {\rm in}\ \  Q_\stiff,\\[0.45em]
		&-(\nabla +{\rm i}\tau)^2u_--zu_-=f\quad {\rm in}\ \  Q_\soft,\\[0.45em]
		&u_+=u_-,\qquad
		\biggl(\dfrac{\partial}{\partial n_+}+{\rm i}\tau\cdot n_+\biggr)u_++\e^{-2}\biggl(\dfrac{\partial}{\partial n_-}+{\rm i}\tau\cdot n_-\biggr)u_-=0,\quad {\rm on}\ \ \Gamma.
	\end{cases}
\end{equation*}
where
 $n_+$ and $n_-=-n_+$ are the outward normals to $\Gamma$ with respect to $Q_\soft$ and $Q_\stiff.$
By a classical argument the weak solution of the above problem is shown to coincide with $({A}_\e^{(\tau)}-z)^{-1}f.$

\subsubsection{Homogenised operators and convergence estimates}
Throughout this section, $H_{\rm hom}:=L^2(Q_\soft)\oplus \mathbb C^1,$ ${\mathcal H}:=L^2(\Gamma),$ and $\partial_n^\tau u:=-(\partial u/\partial n+{\rm i}\tau\cdot n u)|_\Gamma$ is the co-normal boundary derivative for $Q_\soft.$


{\sc Model I.}
Set
\begin{equation*}
	\dom \mathcal A_{\hom}^{(\tau)}=\Bigl\{(u,\beta)^\top\in H_\hom:\ u\in H^2(Q_\soft),  u|_\Gamma=\bigl\langle u|_\Gamma,\psi_0\bigr\rangle_{{\mathcal H}}\psi_0 \text{ and }  \beta = \kappa\bigl\langle u|_\Gamma, \psi_0\bigr\rangle_{{\mathcal H}}\Bigr\},
\end{equation*}
where
 $\kappa:=|Q_\stiff|^{1/2}/|\Gamma|^{1/2},$ $\psi_0(x)=|\Gamma|^{-1/2},$ $x\in \Gamma,$ and define
\begin{equation*}
	\mathcal A_{\hom}^{(\tau)}\binom{u}{\beta}=
	\left(\begin{array}{c}-(\nabla+{\rm i}\tau)^2 u\\[0.5em]
		- \kappa^{-1}\bigl\langle\partial_n^\tau u|_\Gamma, \psi_0\bigr\rangle_{{\mathcal H}}
		- \kappa^{-2}\e^{-2}(\mu_*\tau\cdot\tau)\beta
	\end{array}\right),\qquad \binom{u}{\beta}\in \dom \mathcal A_{\hom}^{(\tau)},
\end{equation*}
where $\mu_*\tau\cdot\tau$ is the leading-order term (for small $\tau$) of the first Steklov eigenvalue for $-(\nabla+{\rm i}\tau)^2$ on $Q_{\rm soft}.$



{\rm Model II.} 
Set 
\begin{equation*}
	\dom \mathcal A_{\hom}^{(\tau)}=\Bigl\{(u,\beta)^\top\in H_\hom:\ u\in H^2(Q_\soft),  u|_\Gamma=\bigl\langle u|_\Gamma,\psi_\tau\bigr\rangle_{{\mathcal H}}\psi_\tau \text{ and }  \beta = \kappa\bigl\langle u|_\Gamma, \psi_\tau\bigr\rangle_{{\mathcal H}}\Bigr\},
\end{equation*}
where $\kappa$ is as above
and $\psi_\tau(x)= |\Gamma|^{-1/2}\exp(-{\rm i}\tau\cdot x)|_\Gamma,$ $x\in\Gamma.$
The action of the operator is set by
\begin{equation*}
	\mathcal A_{\hom}^{(\tau)}\binom{u}{\beta}=
	\left(\begin{array}{c}-(\nabla+{\rm i}\tau)^2 u\\[0.5em]
		-\kappa^{-1}\bigl\langle \partial_n^\tau u|_\Gamma, \psi_\tau\bigr\rangle_{{\mathcal H}}
	\end{array}\right),\qquad \binom{u}{\beta}\in \dom \mathcal A_{\hom}^{(\tau)}.
\end{equation*}

{\sc Convergence estimate.} Set $\gamma=2/3$ for the case of Model I and $\gamma=2$ for the case of Model II. The resolvent $({A}_\e^{(\tau)}-z)^{-1}$ admits the following estimate in the uniform operator-norm topology:
\begin{equation}
	\bigl({A}_\e^{(\tau)}-z\bigr)^{-1}-\Theta^*\bigl(\mathcal A_{\hom}^{(\tau)}-z\bigr)^{-1}\Theta=O(\e^\gamma),
	\label{main_est}
\end{equation}
where $\Theta$ is a partial isometry from $L^2(Q)$ onto $H_\hom:$ on the subspace $L^2(Q_{\rm soft})$ it coincides with the identity, and each function from $L^2(Q_\stiff)$ represented as an orthogonal sum
\[
c_\tau\Vert \Pi_\stiff \psi_\tau\Vert^{-1}\Pi_\stiff \psi_\tau\oplus \xi_\tau,\qquad c_\tau\in{\mathbb C}^1,
\]
is mapped to $c_\tau$ unitarily. Here $\Pi_{\rm stiff}$ maps $\varphi$ on $\Gamma$ to the solution $u_\varphi$ of $-(\nabla+{\rm i}\tau)^2u_\varphi=0$ in $Q_{\rm soft},$ $u_\varphi\vert_{\Gamma}=\varphi.$
The estimate (\ref{main_est}) is uniform in $\tau\in Q'$ and $z\in K_\sigma$.

\section*{Acknowledgements}
KDC is grateful for the financial support of EPSRC Grants EP/L018802/2,
EP/V013025/1. The work of all authors has been
supported by CONACyT CF-2019 No.\,304005.




\begin{thebibliography}{99}

\bibitem{AGW} H. Abels, G. Grubb, I. G. Wood, 2014. Extension theory and Kre\u\i n-type resolvent formulas for non-smooth boundary value problems. {\it J. Func. Anal.} {\bf 266}(7), 4037--4100.	

\bibitem{MR0206711}
V.~M.~Adamjan, D.~Z.~Arov, 1966.
\newblock Unitary couplings of semi-unitary operators. (Russian)
\newblock {\em Mat. Issled.} 1(2), 3--64,
\newblock English translation in: {\it Amer. Math Soc. Transl. Ser. 2,} 95, 1970

\bibitem{AdamyanPavlov}
V. M. Adamyan, B. S. Pavlov, 1986. Zero-radius potentials and M. G. Kre\u\i n's formula for generalized resolvents. {\em J. Soviet Math.} 42(2), 1537--1550.

\bibitem{kura}
S. Albeverio, P. Kurasov, 2000.
\emph{Singular perturbations of differential operators.} Cambridge University Press.


\bibitem{AlonsoSimon}
A. Alonso, B.~Simon, 1980. The Birman-Krein-Vishik theory of self-adjoint extensions
of semibounded operators, {\it J. Oper. Th.} 4, 251--270.



\bibitem{Agranovich} M. S. Agranovich, 2015. {\it Sobolev Spaces, Their Generalizations, and Elliptic Problems in Smooth and Lipschitz Domains.} Springer.

\bibitem{AG} I. M. Akhiezer, N. I. Glazman, 1963. {\it Theory of Linear Operators in Hilbert Space, Vol.\,II,} Frederick Ungar Publishing Co.





\bibitem{Cap}  G. S. Alberti, Y. Capdeboscq, 2018. {\it Lectures on elliptic methods for hybrid inverse problems.} Cours Sp\'ecialis\'es 25. Soci\'et\'e Math\'ematique de France, Paris, vii\,+\,230\,pp.


\bibitem{AKKL}
H. Ammari, H. Kang, K. Kim, H. Lee, 2013. Strong convergence of the solutions of the linear elasticity and uniformity of asymptotic expansions in the presence of small inclusions. {\it J. Differential Equations} 254, 4446--4464.


\bibitem{Arendt} W. Arendt, A. F. M. ter Elst, J. B. Kennedy, M. Sauter, 2014. The Dirichlet-to-Neumann operator via hidden compactness. {\it J. Funct. Anal.} 266(3), 1757--1786.



\bibitem{Behrndt2007}
J.~Behrndt, M.~M.~Malamud, H.~Neidhardt, 2007.
\newblock Scattering theory for open quantum systems with finite rank coupling.
\newblock {\em Math. Phys. Anal. Geom.} 10(4), 313--358.




\bibitem{Amrein}W. O. Amrein, D. B. Pearson, 2004.
\newblock $M$-operators: a generalisation of Weyl-Titchmarsh theory.
\newblock {\em J. Comput. Appl. Math.} 171(1--2), 1--26.

\bibitem{Arov} D.~Z.~Arov, 1979.
\newblock{Passive linear steady-state dynamical systems.}
\newblock {\em Siberian Math. J.}, 20(1), 149--162.



\bibitem{AzizovIokhvidov}T.~Ya.~Azizov, I.~S.~Iokhvidov, 1989. {\it Linear Operators in Spaces with an Indefinite Metric.}
Wiley.



\bibitem{Bakhvalov_Panasenko} N.~Bakhvalov, G.~Panasenko, 1989. {\it Homogenisation: Averaging Processes in Periodic Media.} Kluwer Academic Publishers, Dordrecht.

\bibitem{Ball} J. A. Ball, 1975
Models for non contractions. {\em J. Math. Anal. Appl.} 52, 235--254.

\bibitem{BehrndtLanger2007} J. Behrndt, M. Langer, 2007.
Boundary value problems for elliptic partial differential operators on bounded domains. {\it J. Func. Anal.} 243(2), 536--565.


\bibitem{BehrndtLanger2012} J. Behrndt, M. Langer, 2012.
Elliptic operators, Dirichlet-to-Neumann maps and
quasi boundary triples. {\it In: Operator Methods for Boundary Value Problems,}
London Math. Soc. Lecture Notes 404, 121–-160.

	
	
\bibitem{OQS_Malamud}
J. Behrndt, M. M. Malamud, H. Neidhardt, 2006. Scattering theory for open quantum systems. {\it arXiv:math-ph/0610088,} 48\,pp.








\bibitem{BHS}
J.~Behrndt, S.~Hassi, H.~de Snoo, 2020.
{\it Boundary Value Problems, Weyl Functions, and Differential Operators.} Monographs in Mathematics 108, Birkh\"{a}user.


\bibitem{Lions}
A. Bensoussan, J.-L. Lions, G. Papanicolaou, 1978. {\it Asymptotic Analysis for Periodic Structures.} North Holland, Amsterdam.

\bibitem{Berezin_Faddeev}
F. A. Berezin, L. D.  Faddeev, 1961. Remark on the Schr\"{o}dinger equation with singular potential. (Russian) {\it Dokl. Akad. Nauk SSSR} 137, 1011--1014.



\bibitem{MR3013208}
G.~Berkolaiko and P.~Kuchment, 2013.
\newblock {\em Introduction to quantum graphs.} Mathematical Surveys and Monographs 186,
\newblock American Mathematical Society, Providence, RI.




\bibitem{MR0080271}
M.~\v{S}. Birman, 1956.
\newblock On the theory of self-adjoint extensions of positive definite operators.
\newblock {\em Mat. Sb. N.S.} 38(80), 431--450.

\bibitem{Birman}
M. S. Birman, 1962. Perturbations of the continuous spectrum of a singular elliptic operator by varying the boundary and the boundary conditions, {\it Vestnik Leningrad. Univ.} 17, 22--55. English translation in: Spectral theory of differential operators, {\it Amer. Math. Soc. Transl. Ser. 2} 225 (2008), 19--53, Amer. Math. Soc., Providence, RI.







  \bibitem{Birman_Solomyak}
M. S. Birman, M. Z. Solomiak, 1980. Asymptotics of the spectrum of variational problems on
solutions of elliptic equations in unbounded domains. {\it Funkts. Analiz Prilozhen.}
14, 27--35. English translation in: {\it Funct. Anal. Appl.} 14 (1981), 267--274.


\bibitem{MR1192782}
M.~\v{S}.~Birman and M.~Z.~Solomjak, 1987.
\newblock {\em Spectral theory of selfadjoint operators in Hilbert space}.
\newblock Mathematics and its Applications (Soviet Series). D. Reidel
  Publishing, Dordrecht.

  \bibitem{BirmanSuslina} M. Sh. Birman, T. A. Suslina, 2004. Second order periodic differential operators. Threshold properties and homogenisation. {\it St.\,Petersburg Math. J.} 15(5), 639--714.










\bibitem{deBranges}
L.~de~Branges,  J.~Rovnyak, 1966.
\newblock{{\it Square summable power series.} Holt, Rinehart and Winston, New
York.}

\bibitem{deBrangesRovnyak}
L.~de~Branges,  J.~Rovnyak, 1966.
\newblock{\textit{ Canonical  models in quantum scattering  theory.}
	In: Perturbation  Theory  and  its Applications in  Quantum Mechanics, Wiley, New  York, 295--392.}




\bibitem{BdeM} L. Boutet de Movel, 1971. Boundary problems for pseudo-differential operators. {\it Acta Math.} 126, 11--51.

\bibitem{Brodski}
M.~S.~Brodskij, 1971. \emph{Triangular and Jordan representations of linear
  operators.} Translations of Mathematical Monographs, Vol.\,32. American Mathematical Society, Providence, RI.


\bibitem{BrodskiLivsic}
M.~S.~Brodski\v{i}, M.~S.~Liv\v{s}ic, 1958. {\it Spectral analysis
of non-self-adjoint operators and intermediate systems}
(Russian). {\it Uspehi Mat. Nauk (N.S.)} 13, no. 1(79), 3--85.
English translation in: {\it Amer. Math. Soc. Transl. (2)} 13 (1960), 265--346.


\bibitem{MR2418300}
M.~Brown, M.~Marletta, S.~Naboko, I.~Wood, 2008.
\newblock Boundary triples and {$M$}-functions for non-selfadjoint operators,
  with applications to elliptic {PDE}s and block operator matrices.
\newblock {\em J. Lond. Math. Soc. (2)} 77(3) 700--718.

\bibitem{BMNW2018} M. Brown, M. Marletta, S. Naboko, I. Wood, 2020.
The functional model for maximal dissipative operators: An
approach in the spirit of operator knots. {\it Trans. Amer. Math. Soc.} 373, 4145-4187.

\bibitem{Pavlov_Helmholtz_resonator}
J. Br\"{u}ning, G. Martin, B. Pavlov, 2009. Calculation of the Kirchhoff coefficients for the Helmholtz resonator.
\emph{Russ. J. Math. Phys.} 16(2), 188--207.

\bibitem{Calkin}
J. W. Calkin, 1939. Abstract symmetric boundary conditions. {\it Trans. Amer. Math. Soc.} 45, 369--442.

\bibitem{Phys_book} F. Capolino, 2009. {\it Theory and Phenomena of Metamaterials.} Taylor \& Francis.

\bibitem{Carreau}
M. Carreau, 1993. Four-parameter point-interaction in 1D quantum systems. {\it J. Phys. A: Math. Gen.} 26, 427-432.



\bibitem{CherCoop} K. Cherednichenko, S. Cooper, 2016. Resolvent estimates for high-contrast homogenisation problems. {\it Arch. Rational Mech. Anal.} 219(3), 1061--1086.



\bibitem{KCher}
K. D. Cherednichenko, A. V. Kiselev, 2017. Norm-resolvent convergence of
one-dimensional high-contrast periodic problems to a Kronig-Penney
dipole-type model. {\it Comm. Math. Phys.} 349(2), 441--480.


\bibitem{CEK}
K.~D.~Cherednichenko, Yu.~Yu.~Ershova, A.~V.~Kiselev, 2019. Time-dispersive behaviour as a feature of critical contrast media,  {\it SIAM J. Appl. Math.} 79(2), 690--715.

\bibitem{CherErKis}
K. D. Cherednichenko, Yu. Ershova, A. V. Kiselev, 2020.
\newblock Effective behaviour of critical-contrast PDEs: micro-resonances, frequency conversion, and time dispersive properties. I.
{\it Commun. Math. Phys.} 375, 1833--1884.

\bibitem{CEKN}
K.~Cherednichenko, Y.~Ershova, A.~Kiselev, S.~Naboko, 2019. Unified approach to critical-contrast homogenisation with explicit links to time-dispersive media. {\it Trans. Moscow Math. Soc.} 80(2), 295--342.

\bibitem{CENS}
K. Cherednichenko, Y. Ershova, S. Naboko, L. Silva, 2021. Functional model for generalised resolvents and its application to time-dispersive media. {\it arXiv: 2111.05387}, 24\,pp.


\bibitem{CherednichenkoKiselevSilva}
K.~D.~Cherednichenko, A.~V.~Kiselev, L.~O.~Silva, 2018.
\newblock Functional model for extensions of symmetric operators and applications to scattering theory.
\newblock {\em Netw. and Heterog. Media} 13(2) 191--215.


\bibitem{CherKisSilva1} K. D. Cherednichenko, A. V. Kiselev, L. O. Silva, 2020. Scattering theory for non-selfadjoint  extensions of symmetric operators. {\it Oper. Theory Adv. Appl.} 276, 194--230.

\bibitem{CherKisSilva2}
K. Cherednichenko, A. Kiselev, L. Silva, 2021. Functional model for boundary-value problems. {\it Mathematika} 67(3), 596--626.










\bibitem{Davis}C.~Davis, 1980.
\newblock{J-unitary dilation of a general operator},
{\it Acta Sci.Math.} 31, 75--86.

\bibitem{DavisFoias} C.~Davis, C.~Foia\c{s}, 1971.
Operators with bounded characteristic function and their J-dilation,
{\it  Acta Sci. Math.} 32(1--2), 127-139.



\bibitem{demkov}
Y. N. Demkov, V. N. Ostrovskii, 1975.
\emph{The use of zero-range potentials in atomic physics.} Nauka, Moskva. (Russian)









\bibitem{DyM}H.~Dym, H.~McKean, 1976.
\newblock{\textit{Gaussian Processes, Function Theory, and the Inverse Spectral Problem.} Academic Press.}



\bibitem{Absorbing1} B. Engquist, A. Majda, 1977. Absorbing boundary conditions for the numerical simulation of waves, {\it Math. Comp.} 31(139) 629-–651.


\bibitem{MR3484377}
Y.~Ershova, I.~I. Karpenko, A.~V. Kiselev, 2016.
\newblock Isospectrality for graph {L}aplacians under the change of coupling at
  graph vertices.
\newblock {\em J. Spectr. Theory} 6(1), 43--66.


\bibitem{MR3430381}
Y.~Ershova, I.~I. Karpenko, A.~V. Kiselev, 2016.
\newblock Isospectrality for graph Laplacians under the change of coupling at
  graph vertices: necessary and sufficient conditions.
\newblock {\em Mathematika} 62(1), 210--242.


\bibitem{MR1459512}
P.~Exner. 1997.
\newblock A duality between Schr\"odinger operators on graphs and certain Jacobi matrices.
\newblock {\em Ann. Inst. H. Poincar\'e Phys. Th\'eor.} 66(4), 359--371.

\bibitem{Exner}
P. Exner, O. Post, 2005. Convergence of spectra of graph-like thin manifolds. {\it J. Geom. Phys.} 54(1), 77--115.






\bibitem{Faddeev_additional}
L.~D.~Faddeev, 1974.
\newblock The inverse problem in the quantum theory of scattering. II. (Russian) {\it Current problems in mathematics,} Vol. 3 (Russian),
\newblock Akad. Nauk SSSR Vsesojuz. Inst. Nau\v{c}n. i Tehn. Informacii, Moscow, 93--180.
\newblock English translation in:{\em J. Sov. Math.} 5 (1976), 334--396.





\bibitem{Figotin_Schenker_2005}
A. Figotin, J. H. Schenker, 2005. Spectral analysis of time dispersive and dissipative systems, {\it J. Stat. Phys.} 118(1--2), 199--263.

\bibitem{Figotin_Schenker_2007b}
A. Figotin, J. H. Schenker, 2007. Hamiltonian structure for dispersive and dissipative dynamical systems. {\it J. Stat. Phys.} 128(4), 969--1056.


	\bibitem{Friedlander_old}
	L. Friedlander, 1991. Some inequalities between Dirichlet and Neumann eigenvalues. {\it Arch. Rational Mech. Anal.} 116, 153--160.

	\bibitem{Friedlander} L. Friedlander, 2002. On the density of states of periodic media in the large coupling limit. {\it Comm. Partial Diff. Equations} 27(1--2), 355--380.

\bibitem{Fu}P.~Fuhrmann, 1981.
\newblock{\it Linear Systems and Operators in Hilbert Space.} McGraw-Hill, New York,


\bibitem{Gelfand}
I. M. Gel'fand, 1950. Expansion in characteristic functions of an equation with periodic coefficients. (Russian) {\it Doklady Akad. Nauk SSSR (N.S.)} 73, 1117--1120.

\bibitem{Gesztesy_Mitrea}
F. Gesztesy, M. Mitrea, 2011. A description of all self-adjoint extensions of the Laplacian and Kre\u\i n-type resolvent formulas on non-smooth domains. {\it J. Anal. Math.} 113, 53--172.


\bibitem{Krein}  I. C. Gohberg, M. G. Krein, 1969.
\emph{Introduction to the theory of linear nonself-adjoint operators},
Translations of Mathematical Monographs, Vol. 18. AMS, Providence, R.I.


\bibitem{Golovaty}
Yu. D. Golovaty, R. O. Hryniv, 2013. Norm resolvent convergence of singularly scaled Schrödinger operators and $\delta'$-potentials. {\it Proc. Roy. Soc. Edinburgh Sect. A} 143(4), 791–-816.



\bibitem{Grubb_classic} G. Grubb, 1984. Singular Green operators and their spectral asymptotics. {\it Duke Math. J.} 51(3), 477--528.

\bibitem{Grubb_Robin}
G. Grubb, 2011. Spectral asymptotics for Robin problems with a discontinuous coefficient. {\it J. Spectr. Theory} 1(2), 155--177.

\bibitem{Grubb_mixed}
G. Grubb, 2011. The mixed boundary value problem, Krein resolvent formulas and spectral asymptotic estimates. {\it J. Math. Anal. Appl.} 382(1), 339--363.





\bibitem{Absorbing2} L. Halpern, O. Lafitte, 2007. Dirichlet to Neumann map for domains with corners and
approximate boundary conditions. {\it J. Comp. Appl. Math.}  204, 505–-514.

\bibitem{HempelLienau_2000}
R. Hempel, K. Lienau, 2000. Spectral properties of the periodic media in
large coupling limit. {\it Commun. Partial Diff. Equations}  25, 1445--1470.

\bibitem{Hoermander} L. H\"{o}rmander, 2003. {\it The Analysis of Linear Partial Differential Operators III. Pseudo-Differential Operators.} Springer, Berlin.

\bibitem{Jikov_book} V. V. Jikov, S. M. Kozlov, O. A. Oleinik, 1994. {\it Homogenisation of Differential Operators and Integral Functionals.} Springer, Berlin.


%

\bibitem{MR0407617}
T.~Kato, 1976.
\newblock {\em Perturbation theory for linear operators}.
\newblock Springer, Berlin.






\bibitem{CherKisSilva3}
A. V. Kiselev, L. O. Silva, K. D. Cherednichenko, 2022. Operator-norm resolvent asymptotic analysis of continuous media with low-index inclusions. {\it Math. Notes} 111(3), 373--387.

\bibitem{MR0365218}
A.~N.~Ko{\v{c}}ube{\u\i}, 1975.
\newblock Extensions of symmetric operators and of symmetric binary relations.
\newblock {\em Math. Notes} 17(1), 25--28.




\bibitem{Krasnoselskii} M. A. Krasnoselskii, 1949. On self-adjoint extensions of Hermitian operators. (Russian) {\it Ukrain. Mat. Zhurnal} 1(1), 21--38.


\bibitem{MR0024574}
M.~Kre\u\i n, 1947.
\newblock The theory of self-adjoint extensions of semi-bounded {H}ermitian transformations and its applications. {I}.
\newblock {\em Rec. Math. [Mat. Sbornik] N.S.} 20(62), 431--495.

\bibitem{MR0024575}
M.~G. Kre\u\i n, 1947.
\newblock The theory of self-adjoint extensions of semi-bounded {H}ermitian
transformations and its applications. {II}.
\newblock {\em Mat. Sbornik N.S.} 21(63), 365--404.


\bibitem{KuchmentZeng}
 P. Kuchment, H. Zeng, 2001. Convergence of spectra of mesoscopic systems collapsing onto a graph. {\it J. Math. Anal. Appl.} 258(2), 671--700.

 \bibitem{KuchmentZeng2004}
 P. Kuchment,  H. Zeng, 2004. Asymptotics of spectra of Neumann Laplacians in thin domains. {\it Contemporary Mathematics} 327,
Amer. Math. Soc., Providence, RI, 199--213.


\bibitem{Kuzhel}A.~Kuzhel, 1996.
\newblock {\em Characteristic Functions and Models of Nonself-Adjoint Operators.}
\newblock Kluwer Academic Publishers,
\newblock Dordrecht, The Netherlands.



\bibitem{MR2600145}
P.~Kurasov, 2010.
\newblock Inverse problems for {A}haronov-{B}ohm rings.
\newblock {\em Math. Proc. Cambridge Philos. Soc.} 148(2), 331--362.


\bibitem{Kurasov_supersingular}
P. Kurasov, 2003. ${\mathscr H}_{-n}$-perturbations of self-adjoint operators and Krein's resolvent formula. {\it Integral Equations Operator Theory} 45(4), 437--460.

\bibitem{Kurasov_surfaces}
P. Kurasov, B. Pavlov, 1989. Surfaces with an Internal Structure. In: P.Exner, P.Seba (eds) ``Applications of Self-Adjoint Extensions in Quantum Physics'' (Proc. Dubna, USSR, 1987), Lectures Notes in Physics 324, Springer-Verlag, Berlin,




\bibitem{MR0217440}
P.~D.~Lax, R.~S.~Phillips, 1967.
\newblock {\em Scattering theory}.
\newblock Pure and Applied Mathematics, Vol. 26. Academic Press, New
  York-London.

\bibitem{Lax1973}
P. D. Lax, R. S. Phillips, 1973. Scattering theory for dissipative hyperbolic systems. {\it J. Functional Analysis} 14, 172--235.


\bibitem{MR0020719}
M.~S.~Liv\v{s}ic, 1946
\newblock On a certain class of linear operators in {H}ilbert space.
\newblock {\em Rec. Math. [Mat. Sbornik] N.S.} 19(61), 239--262.

\bibitem{Livsic1954}
M.~S.~Liv\v{s}ic, 1954. On spectral decomposition of linear non-self-adjoint operators. {\it Mat. Sbornik N.S.} 34(76) 145--199 (Russian). English translation in: {\it Amer. Math. Soc. Transl. (2)} 5 (1957), 67--114.

\bibitem{LivsicBook}
M.~S.~Liv\v{s}ic, 1973.
\newblock {\em Operator, Oscillations, Waves. Open Systems.}
\newblock Translations of Mathematical Monographs. Vol. 34. American Mathematical Society, Providence, RI.


\bibitem{Mak_Vas}
N.~G.~Makarov, V.~I.~Vasjunin, 1981. {\it A model for noncontractions and stability of the continuous spectrum.} Lecture Notes in Mathematics 864, 365--412.




\bibitem{McEnnis}
B. McEnnis, 1981. Models for operators with bounded characteristic function, {\it J. Oper. Theory} 43(1-2), 71--90.




 \bibitem{MMP}
A. B. Movchan, N. V. Movchan, C. G. Poulton, 2002. {\it Asymptotic Models of Fields in Dilute and Densely Packed Composites.} Imperial College Press.





\bibitem{MR0500225}
S.~N. Naboko, 1976.
\newblock Absolutely continuous spectrum of a nondissipative operator, and a
  functional model. {I}.
\newblock {\em Zap. Nau\v cn. Sem. Leningrad. Otdel Mat. Inst. Steklov.
  (LOMI)} 65, 90--102.

\bibitem{nabokozapiski2}
S.~N. Naboko, 1977.
\newblock Absolutely continuous spectrum of a nondissipative operator, and a functional model. {II}.
\newblock {\em Zap. Nau\v cn. Sem. Leningrad. Otdel Mat. Inst. Steklov.
  (LOMI)} 73, 118--135.

\bibitem{MR573902}
S.~N. Naboko, 1980.
\newblock Functional model of perturbation theory and its applications to scattering theory.
\newblock {\em Trudy Mat. Inst. Steklov.} 147, 86--114.





\bibitem{MR1036844}
S.~N.~Naboko, 1989.
\newblock Nontangential boundary values of operator {$R$}-functions in a half-plane.
\newblock {\em Algebra i Analiz} 1(5), 197--222.

\bibitem{MR1252228}
S.~N.~Naboko, 1993. On the conditions for existence of wave operators in the nonselfadjoint case. {\it Wave propagation. Scattering theory, Amer. Math. Soc. Transl. Ser. 2} 157, 127--149, American Mathematical Society, Providence, RI.



\bibitem{Neumark1}
M. Neumark, 1940 Self-adjoint extensions of the second kind of a symmetric operator. (Russian) {\it Bull. Acad. Sci. URSS. Sér. Math.} 4, 53--104.


\bibitem{Naimark1940}
M. Neumark, 1940. Spectral functions of a symmetric operator. (Russian)
{\it Bull. Acad. Sci. URSS. Ser. Math.} 4, 277--318.

\bibitem{Naimark1943}
M. Neumark, 1943 Positive definite operator functions on a commutative
group. (Russian) {\it Bull. Acad. Sci. URSS Ser. Math.} 7, 237--244.


\bibitem{Nikolski} N. K. Nikolski, 2002. {\it Operators, Functions, and Systems: An Easy Reading.} Vol. 1, 2., Mathematical Surveys and Monographs 92, American Mathematical Society, Providence, RI.



\bibitem{NikolskiiKhrushchev}
 N.~K.~Nikol'skii, S.~V.~Khrushchev, 1988. A functional model and some problems of the spectral theory of functions. {\it Proc. Steklov Inst. Math.} 176, 101--214.

\bibitem{NikolskiiVasyunin1986}
N.~K.~Nikol'skii, V.~I.~Vasyunin, 1986. {\it Notes on two function models},
Proc. Conter. on the Occasion of the proof of the Bieberbach conjecture, 113--141.

\bibitem{NikolskiiVasyunin1989}
 N.~K.~Nikol'skii, V.~I.~Vasyunin, 1989. A unified approach to function models, and the transcription problem. In:{\it The Gohberg anniversary collection} (Calgary, AB, 1988), Vol.~2, {\it Oper. Theory Adv. Appl.} 41, Birkh\"auser, Basel, 405--434.

\bibitem{NikolskiiVasyunin1998}
 N.~K.~Nikol'skii, V.~I.~Vasyunin, 1998. Elements of spectral theory in terms of the free function model. In: {\it Holomorphic Spaces},
Math. Sci. Research Inst. Publications vol. 33, Cambridge University Press, 211--302.








\bibitem{MR0365199}
B.~S.~Pavlov, 1975.
\newblock Conditions for separation of the spectral components of a dissipative operator.
\newblock {\em Izv. Akad. Nauk SSSR Ser. Mat.} 39, 123--148. English translation in: {\em Math. USSR Izvestija} 9 (1975), 113--137.

\bibitem{MR0510053}
B.~S.~Pavlov, 1977. Selfadjoint dilation of a dissipative Schr\"{o}dinger operator, and expansion in its eigenfunction. (Russian) {\it Mat. Sb. (N.S.)} 102(144), 511--536.


\bibitem{Drogobych}
B.~S.~Pavlov, 1976
\newblock Diation theory and the spectral analysis of non-selfadjoint
differential operators.
\newblock {\em Proc. 7th Winter School, Drogobych, 1974}, TsEMI, Moscow, 2--69, English translation: {\it Transl., II Ser., Am. Math. Soc} 115 (1981), 103--142.

 \bibitem{Pavlov_internal_structure}
B.~S.~Pavlov, 1984. A model of zero-radius potential with internal
structure. (Russian) {\it Teoret. Mat. Fiz.} 59(3), 345--353.




\bibitem{PavlovFaddeev}
B. S. Pavlov, M. D. Faddeev, 1986.  Construction of a self-adjoint dilatation for a problem with impedance boundary condition. {\it J. Soviet Math.} 34, 2152--2156.


\bibitem{Pavlov1988}
B. S. Pavlov, A. A. Shushkov, 1988. Extension theory and zero-range potentials with internal structure. (Russian) {\it Sb. Mat.} 137(179), no.\,2(10), 147--183.







\bibitem{poulton}
N. A. Nicorovici, C. G. Poulton, R. C. McPhedran, 1996. Analytical results for a class of sums involving Bessel functions and square arrays. {\it Journal of Mathematical Physics} 37, 2043.






\bibitem{MR822228}
M.~Rosenblum and J.~Rovnyak, 1985.
\newblock {\em Hardy classes and operator theory}.
\newblock Oxford Mathematical Monographs. Oxford University Press, New York.


\bibitem{Ryzh_ac_sing}
V.~Ryzhov, 1997. Absolutely continuous and singular subspaces of a nonselfadjoint operator. {\it J. Math. Sci.
(New York)} 87(5), 3886--3911.

\bibitem{MR2330831}
V.~Ryzhov, 2007.
\newblock Functional model of a class of non-selfadjoint extensions of
  symmetric operators.
\newblock In: {\em Operator theory, analysis and mathematical physics}, {\em Oper. Theory Adv. Appl.} 174, 117--158. Birkh\"auser, Basel.


\bibitem{Ryzhov_closed}
 V.~Ryzhov, 2008. Functional model of a closed non-selfadjoint operator. {\it Integral Equations Operator Theory} 60(4), 539--571.



\bibitem{Ryzhov_systems} V.~Ryzhov, 2009.
Weyl-Titchmarsh function of an abstract boundary value problem,
operator colligations, and linear systems with boundary control.
{\it Compl. Anal. Oper. Theory} 3, 289-322.


\bibitem{Ryzh_spec} V. Ryzhov, 2020.
Linear operators and operator functions associated with spectral boundary value problems. In: Analysis as a Tool in Mathematical
Physics. {\it Oper. Theory: Adv. Appl.} 276, 576--626.




\bibitem{Schechter}
M. Schechter, 1960. A generalization of the problem of transmission. {\it
	Ann. Scuola Norm. Sup. Pisa} 14(3), 207--236.


\bibitem{MR2953553}
K.~Schm{\"u}dgen, 2012.
\newblock {\em Unbounded self-adjoint operators on Hilbert space.} Graduate Texts in Mathematics 265.
\newblock Springer, Dordrecht.


\bibitem{Shkalikov_1983}
A. A. Shkalikov, 1983. Boundary problems for ordinary differential equations with parameter in the boundary conditions. {\it J. Soviet. Math.} 33(6), 1311--1342.


\bibitem{Solomyak2}
B.~M.~Solomyak, 1992. On the functional model for dissipative operators. The coordinate-free approach. {\it J. Soviet Math.} 61(2), 1981-–2002.


\bibitem{Calderon} I. M. Stein, 1970. \emph{Singular integrals and differentiability properties of functions.} Princeton Math. Ser. vol. 30, Princeton University Press, Princeton, NJ.





\bibitem{Strauss} A. V. \v{S}traus, 1954. Generalised resolvents of symmetric operators (Russian) \emph{Izv. Akad. Nauk SSSR, Ser. Mat.} 18, 51--86.

\bibitem{Strauss1960}
A. V. \v{S}trauss, 1960. Characteristic functions of linear operators. (Russian) {\it Izv. Akad. Nauk SSSR Ser. Mat.} 24(1), 43--74.



 \bibitem{Strauss_ext}
A. V. Strauss, 1970. Extensions and generalized resolvents of a non-densely defined symmetric operator. (Russian) {\it Izv. Akad. Nauk SSSR Ser. Mat.} 34, 175--202.


\bibitem{Strauss_survey} A. V. \v{S}traus, 1999. Functional models and  generalized spectral functions of symmetric operators. {\it St. Petersburg Math. J.} 10(5), 733-784

\bibitem{MR2760647}
B.~Sz.-Nagy, C.~Foias, H.~Bercovici, L.~K{\'e}rchy, 2010.
\newblock {\em Harmonic Analysis of Operators on Hilbert Space}. Springer, New York.









\bibitem{Tip_1998}
A. Tip, 1998. Linear absorptive dielectrics. {\it Phys. Rev. A} 57, 4818--4841.

\bibitem{Tip_2006}
A. Tip, 2006. Some mathematical properties of Maxwell's equations for macroscopic dielectrics. {\it J. Math. Phys.} 47, 012902.




\bibitem{Tutte}
W.~T.~Tutte, 1984.
\emph{Graph theory.}
\newblock Encyclopedia of Mathematics and its Applications, 21. Addison-Wesley Publishing Company, Advanced Book Program, Reading, MA.


\bibitem{Veselago} V. G. Veselago, 1968. The electrodynamics of substances with simultaneously negative values of $\e$ and $\mu$. {\it Soviet Physics Uspekhi.} 10(4), 509--514.

%




\bibitem{MR0051404} M.~I. Vi\v{s}ik, 1952. \newblock On general boundary
problems for elliptic differential equations.  \newblock {\em Trudy
	Moskov. Mat. Ob\v{s}\v{c}.} 1, 187--246.











\bibitem{Zhikov_1989}
V. V. Zhikov, 1989. Spectral approach to asymptotic diffusion problems (Russian). {\it Differentsial'nye uravneniya} 25(1), 44--50.

\bibitem{Zhikov2000}
V. V. Zhikov, 2000. On an extension of the method of two-scale convergence and its applications, {\it Sbornik: Mathematics} 191(7), 973--1014.


\bibitem{Zhikov_2004} V. Zhikov, 2005. On spectrum gaps of some divergent elliptic operators with periodic coefficients. {\it St.\,Petersburg Math. J.} 16(5), 773--790.

\end{thebibliography}
\end{document}